\documentclass[11pt,reqno,oneside,a4paper]{article}
\usepackage[a4paper,includeheadfoot,left=25mm,right=25mm,top=00mm,bottom=20mm,headheight=20mm]{geometry}

\usepackage{amssymb,amsmath,amsthm}
\usepackage{xcolor,graphicx}
\usepackage{verbatim}
\usepackage{mathtools}
\usepackage{hyperref}
\usepackage{titling}
\usepackage{fancyhdr}
\newcommand{\runningtitle}{Running Title}
\newtheorem{thm}{Theorem}
\newtheorem{lem}[thm]{Lemma}

\newtheorem{prop}[thm]{Proposition}
\newtheorem{cor}[thm]{Corollary}

\theoremstyle{definition}
\newtheorem{defn}[thm]{Definition}

\newtheorem{asm}[thm]{Assumption}

\theoremstyle{remark}
\newtheorem{rmk}[thm]{Remark}
\newtheorem{eg}[thm]{Example}

\newcommand{\NN}{\mathbb N}              
\newcommand{\ZZ}{\mathbb Z}              
\newcommand{\RR}{\mathbb R}              
\newcommand{\CC}{\mathbb C}              
\renewcommand{\Re}{\operatorname*{Re}} \renewcommand{\Im}{\operatorname*{Im}}
\newcommand{\D}{\ensuremath{\,\mathrm{d}}}
\newcommand{\ri}{\ensuremath{\mathrm{i}}}
\newcommand{\re}{\ensuremath{\mathrm{e}}}
\newcommand{\la}{\ensuremath{\lambda}}
\renewcommand{\epsilon}{\varepsilon}
\renewcommand{\geq}{\geqslant}
\renewcommand{\leq}{\leqslant}

\providecommand{\BVec}[1]{\mathbf{#1}}

\providecommand{\clos}{\operatorname{clos}}
\newcommand{\abs}[1]{\left\lvert#1\right\rvert}

\newcommand{\Mspacer}{\;} 
\newcommand{\M}[3]{#1_{#2\Mspacer#3}} 
\newcommand{\Msup}[4]{#1_{#2\Mspacer#3}^{#4}} 
\newcommand{\Msups}[5]{#1_{#2\Mspacer#3}^{#4\Mspacer#5}} 

\usepackage{scalerel}
\usepackage{stackengine}
\stackMath
\newcommand\reallywidecheck[1]{%
\savestack{\tmpbox}{\stretchto{%
  \scaleto{%
    \scalerel*[\widthof{\ensuremath{#1}}]{\kern-.6pt\bigwedge\kern-.6pt}%
    {\rule[-\textheight/2]{1ex}{\textheight}}
  }{\textheight}%
}{0.5ex}}%
\stackon[1pt]{#1}{\scalebox{-1}{\tmpbox}}%
}

\newcommand\reallywidehat[1]{%
\savestack{\tmpbox}{\stretchto{%
  \scaleto{%
    \scalerel*[\widthof{\ensuremath{#1}}]{\kern-.6pt\bigwedge\kern-.6pt}%
    {\rule[-\textheight/2]{1ex}{\textheight}}
  }{\textheight}%
}{0.5ex}}%
\stackon[1pt]{#1}{\tmpbox}%
}

\usepackage{esint}

\usepackage{multirow}

\numberwithin{equation}{section}

\providecommand{\bigoh}[1]{\mathcal{O}\left(#1\right)}

\newcommand{\AC}{\operatorname{AC}} 

\providecommand{\argdot}{{}\cdot{}}

\newsavebox{\accentbox}

\usepackage{mathrsfs} 

\providecommand{\FourRealTransNoArg}{\mathrm{F}_\RR} 
\providecommand{\FourRealTrans}[1]{\FourRealTransNoArg\left[#1\right]} 
\providecommand{\FourHalfTransNoArg}{\mathrm{F}_{[0,\infty)}} 
\providecommand{\FourHalfTrans}[1]{\FourHalfTransNoArg\left[#1\right]} 
\providecommand{\FourSineTransNoArg}{\mathrm{F}_{\mathrm{s}}} 
\providecommand{\FourSineTrans}[1]{\FourSineTransNoArg\left[#1\right]} 
\providecommand{\FourUnitTransNoArg}{\mathrm{F}_{[0,1]}} 
\providecommand{\FourUnitTrans}[1]{\FourUnitTransNoArg\left[#1\right]} 

\providecommand{\FourSineSeriesNoArg}{\mathrm{F}_{\mathrm{s}\Mspacer\mathrm{ser}}} 
\providecommand{\FourSineSeries}[1]{\FourSineSeriesNoArg\left[#1\right]} 
\providecommand{\FourInvSineSeriesNoArg}{\FourSineSeriesNoArg^{-1}} 
\providecommand{\FourInvSineSeries}[1]{\FourInvSineSeriesNoArg\left[#1\right]} 

\providecommand{\GenericForTransNoArg}{\BVec{F}} 
\providecommand{\GenericForTrans}[1]{\GenericForTransNoArg\left[#1\right]} 
\providecommand{\GenericInvTransNoArg}{\GenericForTransNoArg^{-1}} 
\providecommand{\GenericInvTrans}[1]{\GenericInvTransNoArg\left[#1\right]} 
\providecommand{\GenericRemTransNoArg}{\BVec{R}} 
\providecommand{\GenericRemTrans}[1]{\GenericRemTransNoArg\left[#1\right]} 

\providecommand{\FormalForTransNoArg}{\mathscr{F}} 
\providecommand{\FormalForTrans}[1]{\FormalForTransNoArg\left[#1\right]} 
\providecommand{\FormalInvTransNoArg}{\FormalForTransNoArg^{-1}} 
\providecommand{\FormalInvTrans}[1]{\FormalInvTransNoArg\left[#1\right]} 

\providecommand{\FokForTransNoArg}{\mathcal{F}} 
\providecommand{\FokForTrans}[1]{\FokForTransNoArg\left[#1\right]} 
\providecommand{\FokForTransPNoArg}{\FokForTransNoArg^+} 
\providecommand{\FokForTransMNoArg}{\FokForTransNoArg^-} 
\providecommand{\FokForTransPMNoArg}{\FokForTransNoArg^\pm} 
\providecommand{\FokForTransP}[1]{\FokForTransPNoArg\left[#1\right]} 
\providecommand{\FokForTransM}[1]{\FokForTransMNoArg\left[#1\right]} 
\providecommand{\FokForTransPM}[1]{\FokForTransPMNoArg\left[#1\right]} 
\providecommand{\FokInvTransNoArg}{\FokForTransNoArg^{-1}} 
\providecommand{\FokInvTrans}[1]{\FokInvTransNoArg\left[#1\right]} 

\providecommand{\FokRemTransNoArg}{\mathcal{R}}
\providecommand{\FokRemTrans}[1]{\FokRemTransNoArg\left[#1\right]}

\providecommand{\rFormalForTransNoArg}[1]{\mathscr{F}_{#1}} 
\providecommand{\rFormalForTrans}[2]{\rFormalForTransNoArg{#1}\left[#2\right]} 
\providecommand{\rFormalInvTransNoArg}[1]{\rFormalForTransNoArg{#1}^{-1}} 
\providecommand{\rFormalInvTrans}[2]{\rFormalInvTransNoArg{#1}\left[#2\right]} 

\providecommand{\rFokForTransNoArg}[1]{\mathcal{F}_{#1}} 
\providecommand{\rFokForTransPNoArg}[1]{\rFokForTransNoArg{#1}^+} 
\providecommand{\rFokForTransMNoArg}[1]{\rFokForTransNoArg{#1}^-} 
\providecommand{\rFokForTransPMNoArg}[1]{\rFokForTransNoArg{#1}^\pm} 
\providecommand{\rFokForTransP}[2]{\rFokForTransPNoArg{#1}\left[#2\right]} 
\providecommand{\rFokForTransM}[2]{\rFokForTransMNoArg{#1}\left[#2\right]} 
\providecommand{\rFokForTransPM}[2]{\rFokForTransPMNoArg{#1}\left[#2\right]} 
\providecommand{\rFokInvTransNoArg}[1]{\rFokForTransNoArg{#1}^{-1}} 
\providecommand{\rFokInvTrans}[2]{\rFokInvTransNoArg{#1}\left[#2\right]} 

\providecommand{\rFokRemTransNoArg}[1]{\mathcal{R}_{#1}} 
\providecommand{\rFokRemTrans}[2]{\rFokRemTransNoArg{#1}\left[#2\right]} 

\providecommand{\propertiesnuradius}{\mathscr{R}} 
\providecommand{\MMinors}{\BVec{M}} 
\providecommand{\CPVwithoutintegral}{\mathrm{p.v.}} 
\providecommand{\CPV}{\CPVwithoutintegral\hspace{-0.3em}\int} 

\providecommand{\complementary}{\ensuremath\mathrm{c}} 
\providecommand{\cuts}{\ensuremath\mathrm{cuts}} 
\providecommand{\principalpart}{\ensuremath\mathrm{pp}} 

\providecommand{\nall}{mn} 
\providecommand{\nallgeneral}{\ensuremath{n_{\mathrm{all}}}} 
\providecommand{\bccount}{\ensuremath{N}} 
\providecommand{\transpose}{\ensuremath{\top}} 
\providecommand{\conjtrans}{\ensuremath{\dag}} 

\definecolor{colorGAMMAcutsDARK}{rgb}{1,0.65,0}
\colorlet{colorGAMMAcuts}{colorGAMMAcutsDARK!25}
\definecolor{colorGAMMAaP}{rgb}{1,0.9,0.9}
\definecolor{colorGAMMAaM}{rgb}{0.9,1,0.9}
\definecolor{colorGAMMAzP}{rgb}{0.9,0.9,1}
\colorlet{colorGAMMAzM}{black!05}
\usepackage[most]{tcolorbox}
\newtcbox{\boxGAMMAc}[1][]{enhanced, colback=colorGAMMAcuts, frame style={opacity=0}, interior style={opacity=1}, nobeforeafter, tcbox raise base, shrink tight, extrude by=1mm, #1}
\newtcbox{\boxGAMMAaP}[1][]{enhanced, colback=colorGAMMAaP, frame style={opacity=0}, interior style={opacity=1}, nobeforeafter, tcbox raise base, shrink tight, extrude by=1mm, #1}
\newtcbox{\boxGAMMAaM}[1][]{enhanced, colback=colorGAMMAaM, frame style={opacity=0}, interior style={opacity=1}, nobeforeafter, tcbox raise base, shrink tight, extrude by=1mm, #1}
\newtcbox{\boxGAMMAzP}[1][]{enhanced, colback=colorGAMMAzP, frame style={opacity=0}, interior style={opacity=1}, nobeforeafter, tcbox raise base, shrink tight, extrude by=1mm, #1}
\newtcbox{\boxGAMMAzM}[1][]{enhanced, colback=colorGAMMAzM, frame style={opacity=0}, interior style={opacity=1}, nobeforeafter, tcbox raise base, shrink tight, extrude by=1mm, #1}

\newcommand{\AckYNCSRP}[2]{#1 gratefully acknowledges support from Yale-NUS College summer research programme #2.}
\newcommand{\AckYNCProj}[1]{#1 gratefully acknowledges support from Yale-NUS College project B grant IG18-PRB102.}

\author{
    S. Aitzhan\textsuperscript{\textasteriskcentered}, S. Bhandari\textsuperscript{\textdagger}, and D. A. Smith\textsuperscript{\textdagger\textdaggerdbl} \\
    \footnotesize\textsuperscript{\textasteriskcentered} Department of Mathematics, Drexel University, Philadelphia, PA, \\
    \footnotesize\textsuperscript{\textdagger} Division of Science, Yale-NUS College, Singapore, \\
    \footnotesize\textsuperscript{\textdaggerdbl} Corresponding author: \href{mailto:dave.smith@yale-nus.edu.sg}{\texttt{dave.smith@yale-nus.edu.sg}}
}
\newcommand\theauthorshort{S. Aitzhan, S. Bhandari, and D. A. Smith}
\title{Fokas diagonalization of piecewise constant coefficient linear differential operators on finite intervals and networks}
\renewcommand{\runningtitle}{Fokas diagonalization}
\date{\today}

\begin{document}
\maketitle
\thispagestyle{fancy}

\begin{abstract}
    We describe a new form of diagonalization for linear two point constant coefficient differential operators with arbitrary linear boundary conditions.
    Although the diagonalization is in a weaker sense than that usually employed to solve initial boundary value problems (IBVP), we show that it is sufficient to solve IBVP whose spatial parts are described by such operators.
    We argue that the method described may be viewed as a reimplementation of the Fokas transform method for linear evolution equations on the finite interval.
    The results are extended to multipoint and interface operators, including operators defined on networks of finite intervals, in which the coefficients of the differential operator may vary between subintervals, and arbitrary interface and boundary conditions may be imposed; differential operators with piecewise constant coefficients are thus included.
    Both homogeneous and inhomogeneous problems are solved.
\end{abstract}

\subsection*{MSC2010 classification}
35P10 (primary).
35C15, 35G16, 47A70 (secondary).

\subsection*{Keywords}
Spectral method for PDE,
Fourier transform,
Unified transform method,
Initial boundary value problem,
Interface problem.

\tableofcontents

\section{Introduction} \label{sec:Introduction}

Spectral methods for two point initial boundary value problems for linear evolution equations have been studied since Fourier introduced his eponymous transforms and series~\cite{Fou1822a}.
The classical Sturm-Liouville theory extends Fourier's results to selfadjoint second order differential operators with variable coefficients~\cite{SL1837a}.
The extensions necessary for nonselfadjoint differential operators and operators of higher spatial order fill many volumes (see, for example,~\cite{CL1955a,Dav2007a,DS1963a,GS1967a,GV1964a,RS1975a}) and have inspired a great deal of 20\textsuperscript{th} century mathematical physics, which we will not attempt to survey here.
Suffice it to note that Birkhoff's 1908 work~\cite{Bir1908b} identified certain ``regularity'' criteria on arbitrary order nonselfadjoint differential operators under which solution representations similar to the Sturm-Liouville generalised Fourier series exist, and the 21\textsuperscript{th} century monographs of Locker~\cite{Loc2000a,Loc2008a} and the many works cited in the survey of Freiling~\cite{Fre2012a} provide extensions of such results beyond the classes Birkhoff identified.

In advances essentially independent of the rich post Sturm-Liouville theory, Fokas and collaborators developed the Fokas transform method (or unified transform method) for solving such problems~\cite{Fok2008a}.
Although it arose initially in the setting of completely integrable nonlinear equations such as the nonlinear Schr\"{o}dinger and Korteweg-de Vries equations, it provides novel solutions even for linear evolution equations with constant coefficients.
This method is unusual among spectral methods for two point boundary value problems in that it provides a solution representation as an integral along a complex contour in spectral space, where other methods customarily yield a series in the spectral variable.
We provide below a survey of the development of the Fokas transform method or rather, as we shall discuss, the family of Fokas transform methods, but we begin by specifying the class of problems of primary interest in the present paper.

\subsection{Two point constant coefficient linear differential operators} \label{ssec:Introduction.Operator}

We study the general constant coefficient linear differential operators defined formally by
\begin{equation} \label{eqn:Lformal}
    \mathcal{L}\phi = \omega\left(-\ri\partial_x\right)\phi = (-\ri)^n\phi^{(n)} + \sum_{j=0}^{n-2} (-\ri)^j c_j \phi^{(j)},
\end{equation}
for the complex coefficient polynomial $\omega$ of degree $n\geq2$.
Without loss of generality, we have assumed that $c_n=1$ and $c_{n-1}=0$.

Let $\Phi=\AC^{n-1}[0,1]$, the space of complex valued functions on $[0,1]$ which are, along with their first $n-1$ derivatives, absolutely continuous.
For $k\in\{1,2\ldots,n\}$, we define two point \emph{boundary forms} $B_k:\Phi\to\CC$ by
\begin{equation} \label{eqn:BoundaryForms}
    B_k\phi = \sum_{j=1}^{n} \M{b}{k}{j}\phi^{(j-1)}(0) + \M{\beta}{k}{j}\phi^{(j-1)}(1),
\end{equation}
for $b,\beta$ two square matrices of complex \emph{boundary coefficients} such that the concatenated matrix $(b:\beta)$ has rank $n$.
We denote the vector of boundary forms $\BVec{B}=(B_1,B_2,\ldots,B_n)$.
The precise differential operator of interest is $L:\Phi_\BVec{B}\to \mathrm{L}^1[0,1]$ given by
\begin{equation} \label{eqn:L}
    L\phi = \mathcal{L}\phi,
\end{equation}
where
\begin{equation} \label{eqn:PhiB}
    \Phi_\BVec{B} = \{\phi\in\Phi: \BVec{B}\phi=\BVec0\}.
\end{equation}


\subsection{Two point initial boundary value problems} \label{ssec:Introduction.IBVP}

The operator $L$ represents the spatial part of the initial boundary value problem (IBVP)
\begin{subequations} \label{eqn:IBVP}
\begin{align}
    \label{eqn:IBVP.PDE} \tag{\theparentequation.PDE}
    \partial_t q(x,t) + a Lq(\argdot,t) &= 0 & (x,t) &\in (0,1) \times (0,T), \\
    \label{eqn:IBVP.IC} \tag{\theparentequation.IC}
    q(x,0) &= Q(x) & x &\in[0,1], \\
    \label{eqn:IBVP.BC} \tag{\theparentequation.BC}
    \BVec{B}q(\argdot,t) &= \BVec{0} & t &\in[0,T],
\end{align}
\end{subequations}
in which we assume $Q\in\Phi_\BVec{B}$.
For the problem to be wellposed, we must require at least that $\arg(a)\in[-\pi/2,\pi/2]$ and, if $n$ is odd, then $\arg(a)\in\{-\pi/2,\pi/2\}$.
If it happens that $\arg(a)\in\{-\pi/2,\pi/2\}$ (that is all odd order equations we consider and some even order equations such as the time dependent linear Schr\"{o}dinger equation), then we make the further restriction that all coefficients of $\omega$ are real.

If the operator $L$ is self adjoint or Birkhoff regular, then the approaches of~\cite{Bir1908b} are sufficient for the solution of IBVP~\eqref{eqn:IBVP}, yielding a solution representation as a series in the eigenfunctions of the differential operator $L$, with coefficients obtained via inner products with the eigenfunctions of its classical adjoint.
If the operator is Locker simply irregular, then the problem is solvable via~\cite{Loc2000a,Loc2008a} and the solution may still be represented as a series, albeit a series in not only the eigenfunctions but also the associated functions of $L$.
However, there exist natural and simply formulated examples of IBVP~\eqref{eqn:IBVP} for which no such approach is possible.
Indeed, it was known as early as 1915~\cite{Jac1915a,Hop1919a} that the operator specified by
\begin{equation} \label{eqn:IBVPJacksonExample}
    \omega(\la)=\la^3, \qquad B_1\phi=\phi(0), \qquad B_2\phi=\phi(1), \qquad B_3(\phi)=\phi'(0)
\end{equation}
has the inconvenient property that it is easy to construct polynomials whose eigenfunction expansions diverge everywhere in $(0,1)$.
Papanicolaou~\cite{Pap2011a} provides a more modern treatment of this divergence phenomenon.
With such a divergent series, it is clear that associated functions cannot be employed to save a series representation in terms of spectral functions.
Nevertheless, selecting $a=-\ri$, IBVP~\eqref{eqn:IBVP} is well posed and solvable using the Fokas transform method~\cite{Pel2004a,Pel2005a}.

\subsection[Survey of Fokas transform methods for IBVP~(\ref*{eqn:IBVP})]{Survey of Fokas transform methods for IBVP~\eqref{eqn:IBVP}} \label{ssec:Introduction.FokasMethods}
\subsubsection{Original version}
Although~\cite{FG1994a} contains many of the ingredients of the method, it is widely accepted that the Fokas transform method for linear evolution equations was first described in~\cite{Fok1997a}.
Both papers, like many of Fokas's early works on the method, also concern integrable nonlinear equations and equations in more than one spatial dimension.
Those results will not be discussed here.
In~\cite{Fok1997a}, the method is outlined for problems on the spatial half line $x\in(0,\infty)$.
The more detailed~\cite{Fok2000a} presents similar results as pertain to half line analogues of IBVP~\eqref{eqn:IBVP}, and announces the two point method that would be detailed in~\cite{FP2001a} a year later.
Fokas and Pelloni~\cite{FP2001a,Pel2002a,FP2005a} studied the two point IBVP~\eqref{eqn:IBVP} for simple separated boundary conditions: in each row of the concatenated matrix $(b:\beta)$ only one entry is nonzero.
The generalisation to the full class of two point boundary forms~\eqref{eqn:BoundaryForms} was begun in~\cite{Chi2006a} and completed by Smith~\cite{Smi2011a,Smi2012a,Smi2012b}, but only for PDE in which $\omega$ is monomial.
That the method has never been implemented in general for the full IBVP~\eqref{eqn:IBVP} may be due to the complexity of the construction of the linear system central to stage~2 (see below) in~\cite{Smi2012a}.

All of these works use essentially the same ``three stage method'' implementation of the Fokas transform method:
\begin{enumerate}
    \item[1.]{
        Under the assumption of existence of a solution to IBVP~\eqref{eqn:IBVP}, derive two equations satisfied by the solution: the \emph{global relation}, which relates time transforms of the $2n$ boundary values $\partial_x^jq(0,\argdot),\partial_x^jq(1,\argdot)$ to the finite interval Fourier transforms of the initial datum and the solution at final time $T$,
        and
        the a complex contour integral representation (or \emph{Ehrenpreis form}), which provides a representation of the solution $q(x,t)$ in terms of the initial datum and all $2n$ of the time transformed boundary values.

        The global relation is obtained via an application of Green's theorem to the rectangular spacetime domain.
        The derivation of the Ehrenpreis form is considerably more involved.
        It begins with the derivation of a scalar first order Lax pair formulation of~\eqref{eqn:IBVP.PDE} to introduce a spectral parameter whose domain is naturally complex.
        Because the temporal and spatial Lax equations are both first order (regardless of $n$), the system is straightforward to integrate, yielding particular solutions independent of the paths of their spacetime integrals.
        Using four such solutions, with path integrals originating at each of the four spacetime corners, a scalar additive Riemann-Hilbert problem is posed and solved (readily, via the Plemelj formulae) to provide a useful complex contour integral representation of the solution to the Lax pair.
        By comparison with the spatial Lax ODE, the Ehrenpreis solution representation is recovered.
        The complex contour integrals in the Ehrenpreis form are inherited from the jump contours of the Riemann-Hilbert problem.
        In the Riemann-Hilbert problem setting, the contours arise naturally as the asymptotic boundaries of the $\mathcal{O}(\abs{\la}^{-1})$ decay condition.

        The ``Lax pair to Riemann-Hilbert problem'' part of this derivation is often described as the ``simultaneous spectral analysis of the Lax pair'', to distinguish it from the purely spatial integration typical of inverse scattering methods.
        Indeed, the four particular integral solutions of the Lax system would be identified as Jost functions in the inverse scattering parlance, but are unusual in that they represent spacetime path integrals rather than just spatial integrals.

        This stage does not use the boundary conditions at all, and therefore need not be reimplemented for subsequent problems in which only the boundary conditions are changed.
    }
    \item[2.]{
        Retaining the assumption of existence, and observing that the Ehrenpreis form relies on boundary values which are not data of the problem, one derives from the global relation a system of $n$ equations in the $2n$ time transformed boundary values.
        Supplemented by time transforms of the $n$ boundary conditions~\eqref{eqn:IBVP.BC}, this provides a full rank system for all the unknown quantities in the Ehrenpreis form.
        There remains the issue that the expressions thus derived feature terms dependent on the Fourier transform of the solution at final time.
        Fortunately, it turns out that, once substituted into the Ehrenpreis form, the specific contour integrals of those problematic terms evaluate to zero or, more generally, may be replaced by other terms whose evaluation is possible using only data of the problem.
        The latter is justified using a Cauchy's theorem and Jordan's lemma argument to deform infinite contours and, if necessary, the global relation to validate a substitution of residues.

        For problems with monomial spatial differential operator or simple separated boundary conditions, it is known that the success of removing the problematic terms is equivalent to wellposedness of the IBVP~\cite{Pel2004a,Smi2012a,Smi2012b}.
        In general, this is open.

        At the conclusion of this stage, one has, under the assumption of existence, derived a representation of the solution.
        As this representation must be satisfied by all solutions of the problem, one has also proved unicity of the solution.
    }
    \item[3.]{
        Beginning with the ``solution representation'' obtained in the previous stage, but regarding it as an ansatz of irrelevant provenance, one directly proves that the function $q(x,t)$ defined by this formula satisfies IBVP~\eqref{eqn:IBVP}.
        This establishes existence of a solution and bootstraps the method.

        Because the solution representation is spectral, it has particularly simple $(x,t)$ dependence.
        Moreover, the contour integrals converge uniformly on arbitrarily large closed subsets of the spacetime domain (in some cases on the full domain), justifying spatial and temporal differentiation under the integrals.
        Readily,~\eqref{eqn:IBVP.PDE} follows.
        The Fourier inversion theorem may be used to justify~\eqref{eqn:IBVP.IC}.
        The boundary conditions are more complicated to justify, but essentially rely on the same kind of Jordan's lemma arguments that were used at the end of stage~2.
    }
\end{enumerate}

\subsubsection{Simplified version}
As hinted at above, the original version of the Fokas transform method is grounded in its origins as an inverse scattering transform method for completely integrable nonlinear evolution equations on domains with boundaries.
That the Riemann-Hilbert implementation of the first stage is nothing more than an artifact of this heritage was identified by Fokas in~\cite{Fok2002a} and used again in~\cite{FP2005a}.
The simplified approach is favoured in particular by Deconinck's group.
Among many generalisations of the method, to evolution equations with mixed spacetime derivatives~\cite{DV2013a}, to systems of evolution equations~\cite{DGSV2018a}, and especially to interface problems~\cite{DPS2014a,DS2014b,SS2015a,DSS2016a,DS2020a,STV2019a}, they provided the most accessible introductions to the method for half line and two point IBVP~\cite{DT2012a,DTV2014a}.

The simplified reimplementation of the Fokas transform method still follows the same three stage method but, in the first stage, the Ehrenpreis form is instead derived as follows.
Beginning with the global relation, the Fourier inversion theorem is applied to derive an equation similar to the Ehrenpreis form but with integrals over $(-\infty,\infty)$.
These real integrals are understood as contour integrals and, via Jordan's lemma, deformed away from the real line.

The major drawback of this approach is that it is rather less obvious than in the Riemann-Hilbert method where the complex contours should lie.
With the new implementation, one must simply deform ``as far Jordan's lemma will allow'' to derive the Ehrenpreis form; that this was the right deformation is only justified in retrospect, when it turns out that the removal of problem terms later in stage~2 requires that precisely that deformation had been performed in stage~1.
A secondary drawback is that the lack of Riemann-Hilbert formalism obfuscates the parallels with the integrable nonlinear version of the Fokas transform method, but that only matters for PDE which are linearizations of integrable nonlinear models.
However, the significant advantage that it requires of the reader much less advanced complex analysis has made the simplified version popular.
Indeed, it has been generalized by others to solve problems with multipoint conditions~\cite{PS2018a}, nonlocal conditions~\cite{MS2018a}, and conditions involving time dependent boundary forms~\cite{GPV2019a,ST2021a}.
It has also been used to analyse the effects of discontinuous data~\cite{BT2019a} and corner singularities~\cite{BT2019b}.

\subsubsection{True transform version}
Yet a third, and more substantially different, implementation of the Fokas transform method was introduced in~\cite{FS2016a} by Fokas and Smith for IBVP~\eqref{eqn:IBVP} with monomial $\omega$, and adapted to half line problems~\cite{Smi2015a,PS2016a}.
These papers attempted to answer the question ``in what sense does the Fokas transform method provide the spectral representation of the differential operator $L$?''.
The answer to this question was already well understood in the case that $L$ belongs to one of the classes of operators for which series spectral representations exist~\cite{Chi2006a,PS2013a}.
But, despite special cases such as~\eqref{eqn:IBVPJacksonExample} having received some attention~\cite{Pel2005a}, there was no general analysis applicable to the more interesting two point operators for which a series expansion in the eigenfunctions diverges.
It was determined that such operators have a spectral representation via a continuously parametrized family of spectral functions, up to a remainder functional, and that the remainder functional divided by the eigenvalue is in the kernel of the inverse Fokas transform.
That the forward Fokas transform coincides with this continuous spectral representation explains its success where spectral series methods failed.

On the way to answering the spectral representation question, it was convenient to view the Fokas transform method as something much more like the classical Fourier sine transform and Fourier cosine transform methods for solving the Dirichlet and Neumann heat problems on the half line.
It had been understood previously that, having gone to the trouble of implementing the Fokas transform method (via the original or simplified approach), one may extract an integral transform pair.
Indeed, although it was derived by very different means, one may, at least formally, interpret the solution representation as the compounded effect of applying the forward transform to the initial datum, evolving in time, and applying the inverse transform to the result.
However, with no way to derive this transform pair other than the full implementation of the Fokas transform method, this observation served little purpose.
In~\cite{FS2016a}, we were able to write down the definition of the transform pair in advance of its justification.
The method then takes a very different form.
With the transform pair already validated via the formulation of an inversion theorem, it only remains to follow the classical Fourier sine transform method as closely as possible (see~\S\ref{sec:Classical} for details).

The greatest advantage of the true transform version of the method is that it is fully algorithmic.
The formulae for the transform pair are complicated, requiring construction of the classical adjoint of $L$, but are explicit, so may easily be implemented programatically~\cite{Xia2019a}.
Although the application of the transform and its inverse is more complicated than are the equivalent steps for the classical Fourier sine transform, this process is still orders of magnitude quicker and easier than the implementation of stages~1 and~2 per the original or simplified versions of the Fokas transform method.
One disadvantage of this approach is that it is less instructive; there is no indication in~\cite{FS2016a} how the general transform pair was derived.
Indeed, at the time this work was done, the formulae for the transform pair with $\omega$ monomial were already known~\cite{Smi2012a}, so~\cite{FS2016a} needed but to reexpress the existing formulae in spectral language and offer new spectral theoretic proofs.
Significant as this disadvantage may seem to mathematicians, it will concern engineers and applied scientists less.
Most of those who would use the Fourier sine transform desire only to know the validity of its inversion theorem; the details of the proof of that validity, much more how one might derive the transform pair from first principles, are irrelevant.
The other problem with the true transform version of the Fokas transform method is that it has not yet been implemented for any IBVP in which $\omega$ is not monomial.

\subsection{Aims~\&~layout of this work} \label{ssec:Introduction.Aims}
The present work aims to solve the latter problem by implementing the true transform version of the Fokas transform method for IBVP~\eqref{eqn:IBVP} in general.
Moreover, the Fokas transform method is also implemented for a broad class of problems~\eqref{eqn:Interface.IBVP}, which includes all initial interface boundary value problems and initial multipoint boundary value problems.
In doing so, we aim also to hint at how the formulae for the transform pair might be arrived at without the benefit of having already completed the implementation of either earlier version of the Fokas transform method.
Indeed, still no such original or simplified formulation exists in general, so this is the first version of the Fokas transform method for general problems~\eqref{eqn:Interface.IBVP} and, for many such problems in which eigenfunction expansions diverge, the first spectral solution method.
In the process, we formulate and prove precise diagonalization spectral representation theorems for operators $L$ defined by equation~\eqref{eqn:L} and interface generalizations thereof.
For two point operators with monomial $\omega$, the diagonalization theorem reduces to that of~\cite{FS2016a}, but here it is presented in a way that more clearly highlights its utility for IBVP and transform methods.

We begin by reviewing in~\S\ref{sec:Classical} the classical Fourier series and transform methods for the heat equation, paying close attention to aspects of the argument that admit generalization and criteria that might be relaxed.
We also give a sketch of the desired new implementation of the Fokas transform method.
In~\S\ref{sec:TransformPair}, we define the Fokas transform pair in general, and prove its validity.
The diagonalization results are formulated and proved in~\S\ref{sec:Diagonalization}, realising the formal conjectures of~\S\ref{sec:Classical}.
Using the theorems of the previous sections, the Fokas transform method itself is implemented in~\S\ref{sec:Method}.
The necessary generalizations for the method to admit inhomogeneities in both the PDE and the boundary conditions are presented in~\S\ref{sec:Inhomogeneous}.
Each step is repeated in the more complicated setting of IIVP in~\S\ref{sec:Interface}.
The method relies on the construction of the classical adjoint of an ordinary differential operator with arbitrary interface and boundary conditions, which appears to be novel in general, so is presented in appendix~\S\ref{sec:AdjointOperator}.
Concluding remarks and some open problems appear in~\S\ref{sec:Conclusion}.

Throughout~\S\ref{sec:TransformPair}--\ref{sec:Method}, we provide a worked example to illustrate the general results in a particular case.
The various parts of the example appear as examples~\ref{eg:Separated3Ord.part1},~\ref{eg:Separated3Ord.part2},~\ref{eg:Separated3Ord.part3},~\ref{eg:Separated3Ord.part4},~\ref{eg:Separated3Ord.part5},~\ref{eg:Separated3Ord.part6}, and~\ref{eg:Separated3Ord.part7}.

\subsection{Informal summary of results} \label{ssec:Introduction.ResultsSummary}
The results presented in this paper are rigourously obtained under precise technical assumptions, but we present here an informal summary of the results without the technical requirements.
\begin{description}
    \item[Definition of transform pair.]{
        In~\S\ref{sec:TransformPair}, for each operator $L$, we define the Fokas transform pair for unit interval functions $\phi$ by
        \[
            \phi \mapsto \FokForTrans{\phi} \qquad\mbox{and}\qquad F \mapsto \FokInvTrans{F}.
        \]
        Under a (not very restrictive) assumption on $L$, we prove the validity of the Fokas transform pair theorem~\ref{thm:FokasTransformValid}:
        \[
            \FokInvTrans{\FokForTrans{\phi}}(x) = \phi(x).
        \]
    }
    \item[Diagonalization.]{
        In~\S\ref{sec:Diagonalization}, we prove a new kind of diagonalization result that describes how the Fokas transform pair interacts with the operator $L$.
        We show that
        \[
            \FokForTrans{L\phi}(\la) = \la^n \FokForTrans{\phi}(\la) + \FokRemTrans{\phi}(\la),
        \]
        and the remainder transform $\FokRemTransNoArg$ is, in a certain sense, in the kernel of the inverse Fokas transform.
        The precise statement is in theorem~\ref{thm:Diag}.
    }
    \item[Solution methods.]{
        In~\S\ref{sec:Method}, we use the diagonalization result to justify that the Fokas transform pair can be used to solve IBVP~\eqref{eqn:IBVP}.
        In the proof of theorem~\ref{thm:FokasMethodSolRep}, we implement a recognizable transform method to show that the solution $q$ of IBVP~\eqref{eqn:IBVP} satisfies
        \[
            q(x,t) = \FokInvTrans{\re^{-at\argdot^n}\FokForTrans{Q}}(x).
        \]
        In~\S\ref{sec:Inhomogeneous}, the extensions to the diagonalization and solution method results are implemented so that inhomogeneous problems may also be solved.
        The effect is to insert into the above solution equation certain additional terms which are defined by appropriate applications of the Fokas transform pair to the inhomogeneities.
        The main result for inhomogeneous IBVP is theorem~\ref{thm:Inhomogeneous.FokasMethodSolRep}.
    }
\end{description}
The results for IIVP appear formally similar to those described above but with appropriate vectorization of the operators, transforms, etc.

\subsubsection*{Notational conventions}
We use notation $X^\star$ to denote an object that is, in some sense made explicit at the point of definition, adjoint to the object $X$.
The complex conjugate of $X\in\CC$ is denoted $\overline{X}$.
The closure of a set $X$ is denoted $\clos(X)$.
Sign superscripts $X^+$ highlight contrast with $X^-$, and usually relate, in some sense, to the upper and lower complex half planes respectively, as in $\CC^\pm$.
The function $\sqrt[n]{\argdot}$ is the principle branch of the $n$\textsuperscript{th} root function, having branch cut along the negative real axis.
We denote the circle and open disc, centered at $\la\in\CC$ with radii $\epsilon>0$, by $C(\la,\epsilon)$ and $D(\la,\epsilon)$, respectively.
The open annulus centred at $\la\in\CC$ with inner radius $\epsilon>0$ and outer radius $\delta>\epsilon$ is denoted $A(\lambda,\epsilon,\delta)$.
We use $\cdot$ for the anonymous argument of a function, $\odot$ for the ordinary unweighted sesquilinear dot product on the vector space $\CC^n$, and $\circ$ for entrywise multiplication of finite lists, and entrywise action of lists of operators on lists of functions.

We use a broad array of Fourier type transforms in this work, summarized in table~\ref{tbl:Transforms}.
The notation $\GenericForTrans{\phi}(\la)$ represents the transform $\GenericForTransNoArg$ of the function $\phi$ evaluated at spectral point $\la$.
When applying one dimensional transforms in the first argument of functions of more than one variable, and provided the meaning is clear, we use the notational shorthand $\GenericForTrans{q}(\la;t)$ to represent $\GenericForTrans{q(\argdot,t)}(\la)$.

\begin{table}[ht]
    \centering
    \footnotesize
    \begin{tabular}{l|lll}
        Class & Notation & Description & Definition \\
        \hline
        \multirow{3}{*}{Generic} & $\GenericForTransNoArg$ & generic Fourier type transform & \S\ref{ssec:Classical.Sketch} \\
         & $\GenericInvTransNoArg$ & generic inverse Fourier type transform & \S\ref{ssec:Classical.Sketch} \\
         & $\GenericRemTransNoArg$ & generic remainder type transform & \S\ref{ssec:Classical.Sketch} \\
        \hline
        \multirow{4}{*}{Fourier} & $\FourRealTransNoArg$ & complex Fourier transform & \eqref{eqn:ClassicalFourierTransform.Defn} \\
         & $\FourHalfTransNoArg$ & complex Fourier transform of function supported on $[0,\infty)$ & \S\ref{ssec:Classical.FourierTransform} \\
         & $\FourSineTransNoArg$ & Fourier sine transform & \eqref{eqn:FourierSineTransformDefn} \\
         & $\FourUnitTransNoArg$ & complex Fourier transform of function supported on $[0,1]$ & \eqref{eqn:FourUnitTrans.Defn} \\
        \hline
        \multirow{6}{*}{Fokas $[0,1]$} & $\FormalForTransNoArg$ & formal Fokas transform & \eqref{eqn:TransformPair.Formal.Defn.Forward} \\
         & $\FormalInvTransNoArg$ & inverse formal Fokas transform & \eqref{eqn:TransformPair.Formal.Defn.Inverse} \\
         & $\FokForTransNoArg$ & Fokas transform & \eqref{eqn:FokasForwardTransformDefinition} \\
         & $\FokForTransPMNoArg$ & left/right Fokas transform & \eqref{eqn:FokasForwardTransformDefinition.LeftRight} \\
         & $\FokInvTransNoArg$ & inverse Fokas transform & \eqref{eqn:FokasInverseTransformDefinition} \\
         & $\FokRemTransNoArg$ & Fokas remainder transform & thm~\ref{thm:Diag} \\
        \hline
        \multirow{8}{*}{Fokas interface} & $\rFormalForTransNoArg{r}$ & formal Fokas transform for operator $\mathcal{L}_r$ & \eqref{eqn:Interface.TransformPair.Formal.Defn.Forward} \\
         & $\rFormalInvTransNoArg{r}$ & inverse formal Fokas transform for operator $\mathcal{L}_r$ & \eqref{eqn:Interface.TransformPair.Formal.Defn.Inverse} \\
         & $\FokForTransNoArg$ & Fokas transform & \eqref{eqn:Interface.FokasForwardTransformDefinition} \\
         & $\rFokForTransPMNoArg{r}$ & $r$\textsuperscript{th} component of the left/right Fokas transform & \eqref{eqn:Interface.FokasForwardTransformDefinition.LeftRight} \\
         & $\FokInvTransNoArg$ & inverse Fokas transform & \eqref{eqn:Interface.FokasInverseTransformDefinition} \\
         & $\rFokInvTransNoArg{r}$ & $r$\textsuperscript{th} component of the inverse Fokas transform & \eqref{eqn:Interface.FokasInverseTransformDefinition} \\
         & $\FokRemTransNoArg$ & Fokas remainder transform & thm~\ref{thm:Interface.Diag} \\
         & $\rFokRemTransNoArg{r}$ & $r$\textsuperscript{th} component of the Fokas remainder transform & proof of thm~\ref{thm:Interface.Diag}
    \end{tabular}
    \caption{
        Notation for Fourier type transforms
    } \label{tbl:Transforms}
\end{table}

\section{Inspiration from classical Fourier methods} \label{sec:Classical}

Problems similar to problem~\eqref{eqn:IBVP} have been studied using a variety of classical methods.
To introduce the new implementation of the Fokas transform method, it is valuable to review the classical methods to which it is most closely related: the Fourier transform and Fourier series methods.
In this section, we shall critically review those methods, drawing attention to both the desirable features and the inconvenient traits of each classical method.
To simplify the presentation, avoiding unnecessarily technical arguments and complicated expressions, we will focus on specific elementary examples involving smooth functions, rather than attempting to state any theorems on the general applicability of these methods.

\subsection{Fourier transform method} \label{ssec:Classical.FourierTransform}

Consider the full line heat problem
\begin{subequations} \label{eqn:ClassicalFourierTransform.IBVP}
\begin{align}
    \label{eqn:ClassicalFourierTransform.IBVP.PDE} \tag{\theparentequation.PDE}
    \partial_t q(x,t) - \partial_{xx} q(x,t) &= 0 & (x,t) &\in \RR \times (0,T), \\
    \label{eqn:ClassicalFourierTransform.IBVP.IC} \tag{\theparentequation.IC}
    q(x,0) &= Q(x) & x &\in\RR, \\
    \label{eqn:ClassicalFourierTransform.IBVP.BC} \tag{\theparentequation.BC}
    q(\argdot,t) &\in \mathcal{S}(\RR) & t &\in[0,T],
\end{align}
\end{subequations}
in which $\mathcal{S}(\RR)$ represents the Schwartz space of smooth rapidly decaying functions.
Taking a Fourier transform in space,
\begin{equation} \label{eqn:ClassicalFourierTransform.1}
    \FourRealTrans{\partial_tq}(\la;t) - \FourRealTrans{\partial_{xx}q}(\la;t) = 0,
\end{equation}
where $\FourRealTransNoArg$ is the usual complex Fourier transform
\begin{equation} \label{eqn:ClassicalFourierTransform.Defn}
    \FourRealTrans{\phi}(\la) = \int_{-\infty}^\infty \re^{-\ri\la x} \phi(x)\D x.
\end{equation}

As any student can recite, ``the Fourier transform turns differentiation into multiplication'' or, phrased in terms of linear algebra, ``the Fourier transform diagonalizes the derivative operator''.
Precisely, integration by parts twice yields
\begin{align*}
    \FourRealTrans{\frac{\D^2}{\D x^2}\phi}(\la) &= \int_{-\infty}^\infty \re^{-\ri\la x} \phi''(x) \D x \\
    &= \left[ \re^{-\ri\la x} \left( \phi'(x) + \ri\la\phi(x) \right) \right]_{x=-\infty}^{x=\infty} - \la^2 \int_{-\infty}^\infty \re^{-\ri\la x} \phi(x) \D x \\
    &= \lim_{x\to\infty}\left[ \re^{-\ri\la x}\left(\phi'(x)+\ri\la\phi(x)\right) \right] - \lim_{y\to-\infty}\left[ \re^{-\ri\la y}\left(\phi'(y)+\ri\la\phi(y)\right) \right] - \la^2 \FourRealTrans{\phi}(\la).
\end{align*}
The exponentials are oscillatory and, for $\phi\in\mathcal{S}(\RR)$, the functions $\phi$ and $\phi'$ are decaying, so the boundary terms all evaluate to $0$.
We conclude that
\[
    \FourRealTrans{\frac{\D^2}{\D x^2}\phi}(\la) = - \la^2 \FourRealTrans{\phi}(\la).
\]
The differential operator has been ``diagonalized'' in the sense that it has been reduced to a multiplication operator in the spectral space.

Applying this to equation~\eqref{eqn:ClassicalFourierTransform.1}, and assuming sufficient smoothness in $t$, we find
\[
    \frac{\D}{\D t} \FourRealTrans{q}(\la;t) + \la^2 \FourRealTrans{q}(\la;t) = 0.
\]
For each $\la\in\RR$, this is an ODE for $\FourRealTrans{q}(\la;\argdot)$.
Applying the Fourier transform to the initial condition~\eqref{eqn:ClassicalFourierTransform.IBVP.IC} of the IBVP provides, for each $\la\in\RR$, an initial condition for the corresponding ODE.
Solving the ODE initial value problem, we find that
\[
    \FourRealTrans{q}(\la;t) = \re^{-\la^2t} \FourRealTrans{Q}(\la).
\]
The Fourier inversion theorem now implies that
\[
    q(x,t) = \frac{1}{2\pi} \int_{-\infty}^\infty \re^{\ri\la x-\la^2t}\FourRealTrans{Q}(\la) \D\la.
\]

This method relied crucially on the diagonalization property of the Fourier transform.
Now consider a slightly different problem, the half line homogeneous Dirichlet heat problem
\begin{subequations} \label{eqn:ClassicalFourierTransform.IBVP2}
\begin{align}
    \label{eqn:ClassicalFourierTransform.IBVP2.PDE} \tag{\theparentequation.PDE}
    \partial_t q(x,t) - \partial_{xx} q(x,t) &= 0 & (x,t) &\in (0,\infty) \times (0,T), \\
    \label{eqn:ClassicalFourierTransform.IBVP2.IC} \tag{\theparentequation.IC}
    q(x,0) &= Q(x) & x &\in[0,\infty), \\
    \label{eqn:ClassicalFourierTransform.IBVP2.BC} \tag{\theparentequation.BC}
    q(0,t) &= 0, \qquad q(\argdot,t) \in \mathcal{S}[0,\infty) & t &\in[0,T],
\end{align}
\end{subequations}
where $Q$ is known and it is assumed that $Q\in\mathcal{S}[0,\infty)$, $Q(0)=0$.
Here,
\[
    \mathcal{S}[0,\infty) = \left\{ \phi = \psi\big\rvert_{[0,\infty)} : \psi\in\mathcal{S}(\RR) \right\}.
\]

Na\"{\i}vely attempting to apply the same method as above, one might begin by deriving the interaction of the half line Fourier transform with the half line homogeneous Dirichlet heat operator:
\begin{align*}
    \FourHalfTrans{\frac{\D^2}{\D x^2}\phi}(\la) &= \int_{0}^\infty \re^{-\ri\la x} \phi''(x) \D x \\
    &= \lim_{x\to\infty}\left[ \re^{-\ri\la x}\left(\phi'(x)+\ri\la\phi(x)\right) \right] - \left( \phi'(0) + \ri\la\phi(0) \right) - \la^2 \FourHalfTrans{\phi}(\la) \\
    &= 0 -\phi'(0)-\ri\la0 - \la^2 \FourHalfTrans{\phi}(\la);
\end{align*}
the rapid decay and homogeneous Dirichlet boundary condition $\phi(0)=0$ ensure that three boundary terms evaluate to $0$, but one boundary term remains.
So the ODE from applying the spatial Fourier transform to equation~\eqref{eqn:ClassicalFourierTransform.IBVP2.PDE} is not so simple:
\[
    \frac{\D}{\D t} \FourHalfTrans{\phi}(\la) + \la^2 \FourHalfTrans{\phi}(\la) + q_x(0,t) = 0.
\]

At this point, one reaches for a new approach.
The Fourier sine transform $\FourSineTransNoArg$, defined by
\begin{equation} \label{eqn:FourierSineTransformDefn}
    \FourSineTrans{\phi}(\la) = \int_0^\infty \sin(\la x) \phi(x) \D x,
\end{equation}
for $\phi\in \mathcal{S}([0,\infty))$ with $\phi(0)=0$,
has the property
\begin{align*}
    \FourSineTrans{\frac{\D^2}{\D x^2}\phi}(\la) &= \lim_{x\to\infty}\left[ \sin(\la x)\phi'(x)-\la\cos(\la x)\phi(x) \right] \\
    &\hspace{10em} - \left( \sin(0)\phi'(0) - \la\cos(0)\phi(0) \right) - \la^2 \FourSineTrans{\phi}(\la) \\
    &= - \la^2 \FourSineTrans{\phi}(\la),
\end{align*}
so we get a simple ODE in $t$ for $\FourSineTrans{\phi}(\la)$ and proceed as before.

In this case, the half line Fourier transform was the wrong tool to solve the problem, because it failed to diagonalize the differential operator; the Fourier sine transform successfully diagonalized the differential operator, so it was perfect for solving the given problem.
Similarly, for the half line homogeneous Neumann heat problem, we use the Fourier cosine transform, and $h$ transforms can be used for Robin problems~\cite{Pin2011b}.
There appears to be a unique transform (pair) suited to each half line homogeneous problem.
Choosing the right one relies on analysis of how the transform interacts with the boundary conditions to diagonalize the differential operator.
Moreover, for half line problems of third order, none of the classical half line Fourier transforms mentioned above have the desired property, and it is not easy, a priori, to derive the ``right'' transform.

\subsection{Fourier series method} \label{ssec:Classical.FourierSeries}

Consider the finite interval homogeneous Dirichlet heat problem
\begin{subequations} \label{eqn:ClassicalFourierTransform.IBVP3}
\begin{align}
    \label{eqn:ClassicalFourierTransform.IBVP3.PDE} \tag{\theparentequation.PDE}
    \partial_t q(x,t) - \partial_{xx} q(x,t) &= 0 & (x,t) &\in (0,1) \times (0,T), \\
    \label{eqn:ClassicalFourierTransform.IBVP3.IC} \tag{\theparentequation.IC}
    q(x,0) &= Q(x) & x &\in[0,1], \\
    \label{eqn:ClassicalFourierTransform.IBVP3.BC} \tag{\theparentequation.BC}
    q(0,t) &= 0, \qquad q(1,t) = 0 & t &\in[0,T],
\end{align}
\end{subequations}
where $Q\in\mathrm{C}^\infty[0,1]$ and $Q(0)=Q(1)=0$.
Using separation of variables, one begins by seeking solutions of equations~\eqref{eqn:ClassicalFourierTransform.IBVP3.PDE} and~\eqref{eqn:ClassicalFourierTransform.IBVP3.BC} of the form $q(x,t)=X(x)\tau(t)$.
One arrives at a temporal ODE and the Sturm Liouville problem
\begin{gather*}
    -X''(x) = \la^2 X(x), \\
    X(0) = 0, \qquad X(1) = 0,
\end{gather*}
which is the eigenvalue problem for the finite interval Dirichlet heat operator.
This problem is readily solved, with eigenvalues $\la_j^2=(j\pi)^2$, for positive integers $j$, and corresponding eigenfunctions $X_j(x) = \sin(\la_jx)$.
The temporal ODE has matching solutions $T_j(x) = \re^{-\ri\la_j^2t}$.
Hoping to find a solution of the full problem~\eqref{eqn:ClassicalFourierTransform.IBVP3}, one appeals to the principle of linear superposition and makes the ansatz
\[
    q(x,t) = \sum_{j=1}^\infty C_j X_j(x) T_j(x),
\]
for some coefficients $C_j$.

The validity of this series expansion ansatz (and the fact the $C_j$ may be calculated via inner products of the initial datum with the corresponding $X_j$) follows from the fact that the eigenfunctions of the finite interval Dirichlet heat operator form a basis.
The Sturm-Liouville theory guarantees this basis property for a wide variety of second order differential operators, and the theory of Birkhoff~\cite{Bir1908b} provides similar theorems about expansion in complete biorthogonal systems for certain classes of nonselfadjoint differential operators of higher order.
But completeness of eigenfunctions (and associated functions; see, for example,~\cite{Loc2000a}) is a crucial feature if one hopes to achieve a spectral series representation.
One need not reach for a complicated example to construct an operator whose eigenfunctions do not form a complete system, the heat operator with boundary conditions $X(0)=X'(0)=0$ will do, but such second order operators correspond to illposed IBVP.
However, it has been known since 1915~\cite{Jac1915a,Hop1919a} and rediscovered several times~\cite{Pel2005a,Pap2011a} (see~\cite{Smi2012a,Smi2012b} for some other examples), that there exist well posed IBVP of the form~\eqref{eqn:IBVP} whose eigenfunctions and associated functions do not form a complete system.
Although such degeneracies occur only for odd order, they need not be especially obscure examples; the third order IBVP
\begin{subequations} \label{eqn:3ordIBVP}
\begin{align}
    \tag{\theparentequation.PDE} \partial_t q(x,t) + \partial_{xxx} q(x,t) &= 0 & (x,t) &\in (0,1) \times (0,T), \\
    \tag{\theparentequation.IC} q(x,0) &= Q(x) & x &\in[0,1], \\
    \tag{\theparentequation.BC} q(0,t) &= 0, \qquad q(1,t) = 0, \qquad q_x(1,t) = 0 & t &\in[0,T]
\end{align}
\end{subequations}
has this property.

\subsection{Fourier series method as a transform method} \label{ssec:Classical.FourierSeriesAsTransform}

We can view the Fourier series method as a transform method too.
From this perspective, for IBVP~\eqref{eqn:ClassicalFourierTransform.IBVP3} the forward transform is the map
\[
    \phi \mapsto (C_j)_{j\in\NN}, \qquad C_j = 2\int_0^1 \sin(j\pi x)\phi(x) \D x,
\]
and the inverse transform is the map
\[
    (C_j)_{j\in\NN} \mapsto \phi, \qquad \phi(x) = \sum_{j=1}^\infty C_j \sin(j\pi x).
\]
In more ``transform'' notation, the forward transform $\FourSineSeriesNoArg$ and inverse transform $\FourInvSineSeriesNoArg$ are
\begin{equation*}
    \FourSineSeries{\phi}(j) = 2\int_0^1 \sin(j\pi x)\phi(x) \D x, \qquad\qquad
    \FourInvSineSeries{(C_j)_{j\in\NN}}(x) = \sum_{j=1}^\infty C_j \sin(j\pi x).
\end{equation*}
Because, for smooth $\phi$ with $\phi(0)=0=\phi(1)$,
\begin{equation} \label{eqn:ClassicalFourierTransform.FourSineSeriesDiagonalization}
    \FourSineSeries{\frac{\D^2}{\D x^2}\phi} = -(j\pi)^2 \FourSineSeries{\phi},
\end{equation}
applying the Fourier sine series transform to~\eqref{eqn:ClassicalFourierTransform.IBVP3.PDE} we find that
\[
    \frac{\D}{\D t} \FourSineSeries{q}(\la;t) + (j\pi)^2 \FourSineSeries{q}(\la;t)=0.
\]
Solving the ODE and applying~\eqref{eqn:ClassicalFourierTransform.IBVP3.IC}, we find
\(
    \FourSineSeries{q}(\la;t) = \re^{-(j\pi)^2t}\FourSineSeries{Q}(\la;t),
\)
and an application of the inverse transform $\FourInvSineSeriesNoArg$ completes the solution of IBVP~\eqref{eqn:ClassicalFourierTransform.IBVP3}.

With the transform method view of the Fourier series method, separation of variables becomes a convenient technique for the derivation of the transform pair.
Validity of the inverse transform relies on the basis properties of the eigenfunctions.
As discussed above, this is a strong criterion for the inversion theorem.
Other than the invertibility of the transform, the essential prerequisite for the above method was diagonalization equation~\eqref{eqn:ClassicalFourierTransform.FourSineSeriesDiagonalization}.
Let us attempt to rederive the forward transform using this as the starting point, but without (or without such obvious a priori expectation of) the strong ``completeness of eigenfunctions'' ansatz implicit in separation of variables.

Assuming the forward transform represents a sesquilinear inner product with some function $X(\argdot,j)$, equation~\eqref{eqn:ClassicalFourierTransform.FourSineSeriesDiagonalization} implies
\begin{equation} \label{eqn:ClassicalFourierTransform.SeriesAsTransform.AdjointEquation}
    -(j\pi)^2 \left\langle \phi , X(\argdot,j) \right\rangle = \left\langle \mathcal{L}\phi , X(\argdot,j) \right\rangle = \left[ \text{boundary terms} \right] + \left\langle \phi , \mathcal{L}^\star X(\argdot,j) \right\rangle,
\end{equation}
in which $\mathcal{L}$ is the formal differential operator $\D^2/\D x^2$.
(Writing the $\mathrm{L}^2$ inner products explicitly, the same equation and an expression for boundary terms may be derived via integration by parts, but we choose this presentation to emphasize that it is the adjoint formal differential operator appearing on the right.)
We begin by ignoring the boundary terms or, equivalently, pretending $\phi(0)=\phi(1)=\phi'(0)=\phi'
(1)=0$, so that the boundary terms evaluate to $0$.
This suggests $-\overline{(j\pi)^2}X(x,j) = \mathcal{L}^\star X(x,j)$; each $X$ is an eigenfunction of the adjoint formal differential operator $\mathcal{L}^\star$.
For any nonzero eigenvalue $-\overline{(j\pi)^2}\in\CC$, there is an eigenspace of $\mathcal{L}^\star$
\begin{equation} \label{eqn:ClassicalFourierTransform.SeriesAsTransform.Eigenspace}
    \operatorname{span}\left\{\overline{\re^{-\ri j\pi x}}, \; \overline{\re^{+\ri j\pi x}}\right\}.
\end{equation}
Using integration by parts (in general this can be done using Green's formula~\cite[corollary to theorem~3.6.3]{CL1955a}), the boundary terms in equation~\eqref{eqn:ClassicalFourierTransform.SeriesAsTransform.AdjointEquation} are
\[
    \phi'(1)X(1,j) - \phi(1)\partial_xX(1,j) - \phi'(0)X(0,j) + \phi(0)\partial_xX(0,j).
\]
Applying only the true boundary conditions $\phi(1)=0=\phi(0)$, suggested by~\eqref{eqn:ClassicalFourierTransform.IBVP3.BC}, we can find the specific eigenfunction $X(\argdot,j)$ in eigenspace~\eqref{eqn:ClassicalFourierTransform.SeriesAsTransform.Eigenspace}, and specify to only certain values: $X(x,j)=\overline{\sin(j\pi x)}$ and $j\in\NN$.

The above approach to deriving the forward transform suggests that the forward transform should be expressible as the inner product with functions $X(\argdot,j)$ having the properties:
\begin{enumerate}
    \item{
        $X(\argdot,j)$ is an eigenfunction of the adjoint formal differential operator $\mathcal{L}^\star$, belonging to eigenspace~\eqref{eqn:ClassicalFourierTransform.SeriesAsTransform.Eigenspace}; the eigenspace is fixed by the diagonalization ansatz~\eqref{eqn:ClassicalFourierTransform.FourSineSeriesDiagonalization}.
        Note that the boundary conditions do not affect this, as $\mathcal{L}^\star$ is a formal differential operator.
    }
    \item{
        The specific choice of eigenfunctions $X(\argdot,j)$ within eigenspace~\eqref{eqn:ClassicalFourierTransform.SeriesAsTransform.Eigenspace} is the one for which the boundary terms in equation~\eqref{eqn:ClassicalFourierTransform.SeriesAsTransform.AdjointEquation} evaluate to zero.
    }
\end{enumerate}
This heuristic view of ``finite interval Fourier transform derivation'' will guide~\S\ref{sec:TransformPair}.

\subsection{Desirable features in a transform method} \label{ssec:Classical.DesirableFeatures}

It is not really fair to directly contrast the Fourier transform and series methods, as they are applied to solve different classes of problems: the Fourier transform method for problems on infinite and semiinfinite domains, and the Fourier series method for problems on finite domains.
But it is instructive to study their distinguishing features, so that one might design a better method incorporating the advantageous characteristics of each.

The Fourier transform method requires, as a first step, knowledge or derivation of the ``right'' transform pair, one which diagonalizes the spatial differential operator.
When faced with an IBVP not previously encountered, it is not obvious how one might invent such a transform pair.
This is a distinct disadvantage of the Fourier transform method contrasted with the Fourier series method which is fully algorithmic in its derivation of the right transform pair.

However, the Fourier series method requires an additional ansatz to be effective: completeness of the eigenfunctions (and, where appropriate, associated functions) of the spatial differential operator.
As discussed in~\S\ref{ssec:Classical.FourierSeries}, this ansatz is false even for some wellposed IBVP, but the method permits no escape from this requirement without imposing unnatural restrictions on the space of admissible initial data.
This represents a drawback of the Fourier series method compared with the Fourier transform method.

In this paper, we seek a method which synthesizes the attractive features of both the classical Fourier methods.
\begin{enumerate}
    \item{
        It should be fully algorithmic, requiring no further invention once the transform pair is known.
    }
    \item{
        The new method should also work without need for completeness of the eigenfunctions and associated functions.
    }
\end{enumerate}

It would also be desirable that our method be applicable equally to problems posed on the finite interval and the semiinfinite interval.
We omit the latter class of problems from consideration in this work only for reason of space.
The Fokas transform method is applicable to such problems, and~\cite{PS2016a} describes a true transform version of the method, albeit only for $\omega$ monomial.
The extension to general $\omega$ is left for future work.

\subsection{Sketch of a transform method} \label{ssec:Classical.Sketch}

The various Fourier transforms studied in~\S\ref{ssec:Classical.FourierTransform} and~\S\ref{ssec:Classical.FourierSeriesAsTransform} have the property that, for corresponding differential operators $L$ (or, rather, by choosing the appropriate transform for a given operator $L$), the operator is diagonalized by the transform in the sense that
\begin{equation} \label{eqn:Classical.Sketch.Diagonalization}
    \GenericForTrans{L\phi}(\la) = z(\la)\GenericForTrans{\phi}(\la),
\end{equation}
for some polynomial $z$.
Suppose that, for a given two point differential operator $L:\Phi\to \mathrm{L}^1[0,1]$ as described in~\S\ref{ssec:Introduction.Operator}, we already know a linear transform $\GenericForTransNoArg$ with this property, and we also know its inverse transform $\GenericInvTransNoArg$ has that for all $\phi\in\Phi$ and all $x\in(0,1)$
\begin{equation} \label{eqn:Classical.Sketch.InverseTransform}
    \GenericInvTrans{\GenericForTrans{\phi}}(x) = \phi(x).
\end{equation}
Then the transform method can proceed as follows.
First, the transform is applied to equation~\eqref{eqn:IBVP.PDE} in the spatial variable, and the diagonalization property is used to reduce the result to an ODE in $t$,
\[
    \frac{\D}{\D t}\GenericForTrans{q}(\la;t) + a z(\la) \GenericForTrans{q}(\la;t) = 0.
\]
This ODE is easily solved as $\GenericForTrans{q}(\la;t) = \re^{-az(\la)t}\GenericForTrans{q}(\la;0)$ and the same transform is applied to the initial condition~\eqref{eqn:IBVP.IC} to obtain an initial value.
Therefore
\[
    \GenericForTrans{q}(\la;t) = \re^{-az(\la)t}\GenericForTrans{Q}(\la),
\]
and application of the inverse transform yields
\[
    q(x,t) = \GenericInvTrans{ \re^{-az(\argdot)t}\GenericForTrans{Q} }(x),
\]
the solution of the IBVP.

This approach is fine as long as one can find, \emph{as long as there exists}, a transform pair that truly diagonalizes $L$ in the sense of equation~\eqref{eqn:Classical.Sketch.Diagonalization}.
But it is possible to specify a weaker ``diagonalization'' type criterion and still have a valid method.
To wit, suppose that instead of diagonalization equation~\eqref{eqn:Classical.Sketch.Diagonalization} it holds that
\begin{subequations} \label{eqn:Classical.Sketch.WeakDiagonalization}
\begin{equation} \label{eqn:Classical.Sketch.WeakDiagonalization.Diagonalization}
    \GenericForTrans{L\phi}(\la) = z(\la)\GenericForTrans{\phi}(\la) + \GenericRemTrans{\phi}(\la),
\end{equation}
where $\GenericRemTransNoArg$ has the property that, for appropriate functions $q(x,t)$ such that $q(\argdot,t)\in\Phi$, and for all $x\in(0,1)$,
\begin{equation} \label{eqn:Classical.Sketch.WeakDiagonalization.RemainderControl}
    \GenericInvTrans{\int_0^t \re^{az(\argdot)(s-t)} \GenericRemTrans{q}(\argdot;s) \D s}(x) = 0,
\end{equation}
\end{subequations}
and still $\GenericInvTransNoArg$ is linear and the inverse of $\GenericForTransNoArg$ in the sense of equation~\eqref{eqn:Classical.Sketch.InverseTransform}.
Attempting to apply this transform to~\eqref{eqn:IBVP.PDE}, one arrives at the only slightly more complicated ODE
\[
    \frac{\D}{\D t}\GenericForTrans{q}(\la;t) + a z(\la) \GenericForTrans{q}(\la;t) + \GenericRemTrans{q}(\la;t) = 0.
\]
The solution of its initial value problem is
\[
    \GenericForTrans{q}(\la;t) = \re^{-az(\la)t}\GenericForTrans{Q}(\la) - \re^{-az(\la)t} \int_0^t \re^{az(\la)s} \GenericRemTrans{q}(\la;s) \D s.
\]
Applying the inverse transform and equation~\eqref{eqn:Classical.Sketch.WeakDiagonalization.RemainderControl} yields solution representation
\[
    q(x,t) = \GenericInvTrans{ \re^{-az(\argdot)t}\GenericForTrans{Q} }(x).
\]

That the type of diagonalization formally conjectured in equations~\eqref{eqn:Classical.Sketch.WeakDiagonalization} is weaker than the classical sense of equation~\eqref{eqn:Classical.Sketch.Diagonalization} follows from the restriction $\GenericRemTransNoArg=0$.
It is less obvious that it is strictly weaker in a useful way.
Nevertheless, in theorem~\ref{thm:Diag}, we will show that this is truly the sense in which the Fokas transform diagonalizes $L$; the nonzero remainder transform $\GenericRemTransNoArg$ is explicitly constructed in~\S\ref{sec:Diagonalization}.
Because we have chosen to specialise the arguments of this work to operators on finite domains, the classical Fourier transform methods are not truly included in the above class of transforms.
However, it is clear that they are closely related to the case $\GenericRemTransNoArg=0$.
When the Fourier series is viewed as a transform as in~\S\ref{ssec:Classical.FourierSeriesAsTransform}, the Fourier series fits diagonalization~\eqref{eqn:Classical.Sketch.WeakDiagonalization} with $\GenericRemTransNoArg=0$.

\section{Fokas transform pair} \label{sec:TransformPair}

As discussed in~\S\ref{ssec:Classical.FourierSeriesAsTransform}, the forward transform in the Fourier series is an inner product with certain eigenfunctions of the adjoint formal differential operator $\mathcal{L}^\star$.
Inspired by this observation, we will construct in~\S\ref{ssec:TransformPair.Formal} a ``formal transform'' from eigenfunctions of $\mathcal{L}^\star$.
Specifically, we will derive a formal transform pair tailored to the overspecified formal IBVP~\eqref{eqn:IBVP} but with boundary conditions~\eqref{eqn:IBVP.BC} replaced by
\[
    \partial_x^jq(0,t) = 0 = \partial_x^jq(1,t) \qquad \qquad \mbox{for all } j\in\{0,1,\ldots,n-1\}.
\]
As we are aiming for weaker diagonalization~\eqref{eqn:Classical.Sketch.WeakDiagonalization} instead of~\eqref{eqn:ClassicalFourierTransform.FourSineSeriesDiagonalization}, the definition of the full forward Fokas transform in~\S\ref{ssec:TransformPair.BC} accounts for the boundary conditions in a different way from what was done in~\S\ref{ssec:Classical.FourierSeriesAsTransform}.
Before we can begin, because we wish to treat differential operators with nonzero lower order terms, we require the results detailed in~\S\ref{ssec:TransformPair.nu} on a class of biholomorphic functions.

\subsection{A biholomorphic function} \label{ssec:TransformPair.nu}

Let $\nu(\la)$ be such that
\begin{equation} \label{eqn:defnnu}
    \omega(\nu(\la)) = \la^n, \qquad \lim_{\la\to\infty}\nu(\la)/\la=1.
\end{equation}
The effect of the limit is to specify a particular branch of this inverse polynomial type function, at least for $\la$ sufficiently large.
The following proposition establishes that such a function exists, and describes some properties of $\nu$.

\begin{prop} \label{prop:propertiesnu}
    For sufficiently large $\propertiesnuradius>0$, there is a function $\nu$ which satisfies equations~\eqref{eqn:defnnu}, is biholomorphic outside the disc centered at zero with radius $\propertiesnuradius$, and whose derivative approaches $1$ as $\la\to\infty$, uniformly in the argument of $\la$.
\end{prop}

\begin{proof}
    For $\nu\in\CC$, consider the function $\la:\CC\to\CC$ defined by
    \begin{equation} \label{eqn:propertiesnu.proof.0}
        \la(\nu) = \nu \sqrt[n]{1+\sum_{j=0}^{n-2}c_j\left(\frac{1}{\nu}\right)^{n-j}}.
    \end{equation}
    Raising both sides to the power $n$, we see that
    \begin{equation} \label{eqn:propertiesnu.proof.1}
        \la(\nu)^n = \omega(\nu).
    \end{equation}
    We aim to define $\nu(\la)$ as an appropriate inverse of $\la(\nu)$.

    The function $1+\sum_{j=0}^{n-2}c_j\left(\frac{1}{\nu}\right)^{n-j}$ is a polynomial in $\left(\frac{1}{\nu}\right)$, so it is holomorphic except at zero.
    If
    \[
        \abs{\nu} > \max\left\{1,\sum_{j=0}^{n-2}\abs{c_j}\right\}, \qquad \mbox{then} \qquad \abs{\sum_{j=0}^{n-2}c_j\left(\frac{1}{\nu}\right)^{n-j}} < 1,
    \]
    so, for all $\nu$ sufficiently large,
    \[
        \Re\left(1+\sum_{j=0}^{n-2}c_j\left(\frac{1}{\nu}\right)^{n-j}\right) > 0.
    \]
    On the right half plane, $\sqrt[n]{\argdot}$ is holomorphic.
    Therefore, outside a sufficiently large disc, $\la$ is a holomorphic function of $\nu$.
    Moreover, implicit differentiation applied to equation~\eqref{eqn:propertiesnu.proof.1} implies that, for any choice of $\epsilon\in(0,1)$, outside a (possibly larger) disc, $\la'(\nu)$ lies in a disc of radius $\epsilon$ centred at $1$.
    In particular, we choose $\epsilon=1/\pi$, and let $\propertiesnuradius'$ be the radius of a disc sufficiently large that on its complement $\la$ is holomorphic and its derivative lies within $\epsilon$ of $1$.

    Next, we show that, on the domain $\abs{\nu}>\propertiesnuradius'$, $\la$ is univalent.
    Suppose there exist $\nu_1,\nu_2$ in the domain such that $\la(\nu_1)=\la(\nu_2)$.
    Then, for any simple contour $\gamma$ extending from $\nu_1$ to $\nu_2$ without exiting the domain,
    \[
        0 = \la(\nu_2)-\la(\nu_1) = \int_\gamma \la'(\nu) \D\nu.
    \]
    Defining for notational convenience $\mu(\nu) = \la'(\nu)-1$, it follows that $\abs{\mu(\nu)}\leq1/\pi$ and
    \[
        0 = \int_\gamma 1+\mu(\nu) \D\nu = \nu_2-\nu_1 + \int_\gamma \mu(\nu) \D\nu.
    \]
    It is not necessarily possible to select $\gamma$ to be the path following the straight line segment connecting $\nu_1$ to $\nu_2$, as this contour may pass outside the selected domain of $\la$.
    However, there are two semicircles having that straight line segment as their diameter, and (at least) one of them must lie wholly within the domain of $\la$. Selecting $\gamma$ to follow that semicircle from $\nu_1$ to $\nu_2$, we observe that $\mbox{length}(\gamma) = \pi\abs{\nu_2-\nu_1}/2$.
    Then
    \[
        \abs{\int_\gamma \mu(\nu) \D\nu} \leq \frac{\mbox{length}(\gamma)}{\pi} = \frac{\abs{\nu_2-\nu_1}}{2}.
    \]
    But then $\abs{\nu_2-\nu_1}\leq\frac{\abs{\nu_2-\nu_1}}{2}$, so $\nu_2=\nu_1$.
    We have shown that $\la$ is a univalent function.

    Because $\la$ is a univalent holomorphic function outside the disc of radius $\propertiesnuradius'$, it is biholomorphic.
    We define $\nu$ to be the inverse of $\la$ on this domain.
    Then $\nu$ is also a biholomorphic function, on the image of the complement of the disc of radius $\propertiesnuradius'$, a subset of which is the complex plane outside another sufficiently large disc.
    It follows from equation~\eqref{eqn:propertiesnu.proof.1} that the first of equations~\eqref{eqn:defnnu} holds.
    The limit in equation~\eqref{eqn:defnnu} follows from equation~\eqref{eqn:propertiesnu.proof.0}.
    Differentiating implicitly the first of equations~\eqref{eqn:defnnu} implies the claimed limit of $\nu'$.
\end{proof}

This choice of domain of $\nu$ is unnecessarily restrictive; one could expand the domain of $\nu$ up to its branch cuts.
However, the choice of the exterior of a disc as the domain simplifies certain statements in the proceeding, because it ensures that $\nu$ is a biholomorphism and the domains of $\nu$ and its rotations $\nu(\re^{\ri\theta}\argdot)$ are the same.

If $\omega$ is the particularly simple polynomial $\omega(\nu) = \nu^n + c_0$, which is always the case when $n=2$, then an explicit formula for $\nu$ is readily derived.
Indeed, for such $\omega$, the function
\[
    \nu(\la) = \lambda \sqrt[n]{1-\frac{c_0}{\la^n}}
\]
agrees with the branch of $\nu$ defined in proposition~\ref{prop:propertiesnu} and is analytically continued to the whole complex plane except for the star of branch cuts composed of $n$ straight line segments
\[
    \la\in\CC\mbox{ such that }\abs{\la}<\sqrt[n]{\abs{c_0}}\mbox{ and }\arg(\la)=(\arg(c_0)+2k\pi)/n\mbox{ for some }k\in\ZZ.
\]
Note that every zero of $\omega$ corresponds to a branch point of $\nu$ so, unless $\omega$ is monomial, $\nu$ must have some branch cuts; it is impossible to extend $\nu$ to an entire function except when $\omega$ is monomial, in which case $\nu$ is the identity function.

\begin{eg} \label{eg:Separated3Ord.part1}
    Consider the problem
    \begin{subequations} \label{eqn:eg.IBVP}
    \begin{align}
        \label{eqn:eg.IBVP.PDE} \tag{\theparentequation.PDE}
        \partial_t q(x,t) + (\partial_x^3 - \partial_x)q(x,t) &= 0 & (x,t) &\in (0,1) \times (0,T), \\
        \label{eqn:eg.IBVP.IC} \tag{\theparentequation.IC}
        q(x,0) &= Q(x) & x &\in[0,1], \\
        \label{eqn:eg.IBVP.BC} \tag{\theparentequation.BC}
        q(0,t)=0, \quad q(1,t)=0, \quad q_x(1,t) &= 0 & t &\in[0,T].
    \end{align}
    \end{subequations}
    This is IBVP~\eqref{eqn:IBVP} in which
    \begin{gather*}
        \mathcal{L}\phi = \omega(-\ri\partial_x)\phi, \qquad \omega(k)=k^3+k, \qquad a=-\ri, \\
        B_1\phi = \phi(0), \qquad B_2\phi = \phi(1), \qquad B_3\phi=\phi'(1).
    \end{gather*}
    Therefore the biholomorphic function of proposition~\ref{prop:propertiesnu} satisifes
    \(
        \nu(\la)^3+\nu(\la)=\la^3
    \)
    and is the inverse of
    \(
        \la(\nu) = \nu\sqrt[3]{1+1/\nu^2}
    \)
    outside an a disc of radius $\propertiesnuradius$.
    Following the unoptimized arguments presented in the proof of the proposition, one may select $\propertiesnuradius'=11\pi/3$, which guarantees that $3.5$ is a sufficently large value for $\propertiesnuradius$.
    Numerical evaluation suggests that $\nu$ is biholomorphic outside a disc of radius $0.8$.
    
    This example is continued in example~\ref{eg:Separated3Ord.part2}.
\end{eg}

\subsection{A formal transform pair} \label{ssec:TransformPair.Formal}

The ordinary finite interval Fourier transform $\FourUnitTransNoArg$, defined by
\begin{equation} \label{eqn:FourUnitTrans.Defn}
    \FourUnitTrans{\phi}(\la):= \int_0^1 \re^{-\ri\la x} \phi(x) \D x,
\end{equation}
may be understood as a complex inner product with the function $\re^{\ri\bar\la x}$.
For each $\la\in\CC$,
\begin{equation} \label{eqn:FourierEigenfunctions}
    \mathcal{L}^\star_0 \, \re^{\ri\bar\la x} = (-\ri)^n\frac{\D^n}{\D x^n} \re^{\ri\bar\la x} = \bar\la^n \re^{\ri\bar\la x};
\end{equation}
the exponential Fourier kernel is an eigenfunction of the monomial $n$\textsuperscript{th} order formal differential operator $\mathcal{L}^\star_0$, and the conjugate of every complex $\la$ is the $n$\textsuperscript{th} root of an eigenvalue.

Aiming to construct the Fokas transform as a generalization and continuous counterpart of classical nonselfadjoint Fourier series, we must incorporate the eigenfunctions of adjoint formal differential operator $\mathcal{L}^\star$.
In analogy with equation~\eqref{eqn:FourierEigenfunctions}, the eigenfunctions of $\mathcal{L}^\star$ are $\re^{\ri\overline{\nu(\la)} x}$, with eigenvalues $\bar{\la}^n$.
We use this observation to construct a transform pair of relevance to the formal differential operator $\mathcal{L}$.
The forward transform will be a finite interval Fourier transform, except with the exponential Fourier kernel replaced by a continuum of eigenfunctions of $\mathcal{L}^\star$.

Let $\propertiesnuradius$ be as defined in proposition~\ref{prop:propertiesnu}, and pick any $\propertiesnuradius_1,\propertiesnuradius_2>\propertiesnuradius$.
Let $\gamma$ be the contour that follows $\nu^{-1}((-\infty,-\propertiesnuradius_1])$, then a simple path wholly in $\nu^{-1}(\CC^+)$ or wholly in $\nu^{-1}(\CC^-)$ from $\nu^{-1}(-\propertiesnuradius_1)$ to $\nu^{-1}(\propertiesnuradius_2)$, then follows $\nu^{-1}([\propertiesnuradius_2,\infty))$.
It is possible to define a contour this way because $\nu$ is a bijection.
The properties of $\nu$ imply that $\gamma$ approaches the negative and positive real axes at the beginning and end.
If the coefficients of $\omega$ are real, then $\gamma$ is a finite deformation of the real line.
For $\rho>\propertiesnuradius$, define also $\gamma_\rho$ to be the restriction of $\gamma$ to $\nu(D(0,\rho))$.
We define the forward formal transform $\FormalForTransNoArg$ and inverse formal transform $\FormalInvTransNoArg$ by
\begin{subequations} \label{eqn:TransformPair.Formal.Defn}
\begin{align}
    \label{eqn:TransformPair.Formal.Defn.Forward}
    \FormalForTrans{\phi}(\la) &:=
    \left\langle \phi , \re^{\ri\overline{\nu(\la)}\argdot} \right\rangle =
    \int_0^1 \phi(x) \re^{-\ri\nu(\la)x} \D x, & \abs{\la} > \propertiesnuradius, \\
    \label{eqn:TransformPair.Formal.Defn.Inverse}
    \FormalInvTrans{F}(x) &:= \frac{1}{2\pi} \lim_{\rho\to\infty}\int_{\gamma_\rho} \re^{\ri\nu(\la)x} \nu'(\la)F(\la) \D\la.
\end{align}
\end{subequations}

The formal transform pair is valid in the sense of theorem~\ref{thm:FormalTransformValid}.

\begin{thm} \label{thm:FormalTransformValid}
    Suppose that $\phi \in \mathrm{C}^1[0,1]$.
    Then, for all $x\in(0,1)$,
    \[
        \FormalInvTrans{\FormalForTrans{\phi}}(x) = \phi(x).
    \]
\end{thm}

\begin{proof}
    Changing variables $\mu=\nu(\la)$, we observe
    \begin{align*}
        \FormalInvTrans{\FormalForTrans{\phi}}(x) &= \frac{1}{2\pi} \lim_{\rho\to\infty}\int_{\gamma_\rho} \re^{\ri\nu(\la)x} \nu'(\la)
        \int_0^1 \phi(y) \re^{-\ri\nu(\la)y} \D y
        \D\la \\
        &= \frac{1}{2\pi} \lim_{\rho\to\infty}\int_{\gamma'_\rho} \re^{\ri\mu x}
        \int_0^1 \phi(y) \re^{-\ri\mu y} \D y
        \D\mu \\
        &= \frac{1}{2\pi} \lim_{\rho\to\infty}\int_{\gamma'_\rho} \re^{\ri\mu x}
        \FourUnitTrans{\phi}(\mu)
        \D\mu,
    \end{align*}
    in which $\gamma'_\rho=\nu^{-1}(\gamma_\rho)$ is the restriction to $D(0,\rho)$ of a contour that follows $(-\infty,-\propertiesnuradius_1]$, then a simple path in $\CC^+$ or $\CC^-$ from $-\propertiesnuradius_1$ to $\propertiesnuradius_2$, then follows $[\propertiesnuradius_2,\infty)$.
    The integrand of the outer integral is entire.
    Therefore, assuming $\rho$ is sufficiently large as to ensure $D(0,\rho)$ contains the contour linking $-\propertiesnuradius_1$ to $\propertiesnuradius_2$, and using Cauchy's theorem, the contour $\gamma'_\rho$ may be replaced by the contour $[-\rho,\rho]$, so that
    \[
        \FormalInvTrans{\FormalForTrans{\phi}}(x)
        = \frac{1}{2\pi} \lim_{\rho\to\infty}\int_{-\rho}^\rho \re^{\ri\mu x} \FourUnitTrans{\phi}(\mu) \D\mu,
    \]
    which is an ordinary real Cauchy principal value integral.

    Denoting by $\phi_0$ the zero extension of $\phi$ to the full real line, the integral
    \[
        \FourUnitTrans{\phi}(\mu) = \int_0^1 \phi(y) \re^{-\ri\mu y} \D y = \int_{-\infty}^\infty \phi_0(y) \re^{-\ri\mu y} \D y = \FourRealTrans{\phi_0}(\mu)
    \]
    is an ordinary full line Fourier transform of the function $\phi_0$, which is absolutely integrable and piecewise $\mathrm{C}^1$.
    The result follows by the usual Fourier inversion theorem.
\end{proof}

The above proof demonstrates significant freedom to choose the integration contour for inverse Fourier transforms, when the function originally transformed had compact support.
This same same freedom will also be exploited in the validity theorem~\ref{thm:FokasTransformValid} for the transform pair respecting the boundary conditions.

Some kind of principal value is necessary in the Fourier inversion theorem for absolutely integrable piecewise $\mathrm{C}^1$ functions $\phi$.
However, because $\FourUnitTrans{\phi}(\la)$ is bounded for $\la\in\RR$, it need not be exactly the Cauchy principal value.
The limit in equation~\eqref{eqn:TransformPair.Formal.Defn.Inverse} represents a contour integral version of a Cauchy principal value, in which the bounding region of the contour expands so that its image in $\nu$ is a disc.
Despite being the most natural extension of the Cauchy principal value in the ordinary Fourier inversion theorem, hence a convenient setting for the proof of theorem~\ref{thm:FormalTransformValid}, this is not always the most conducive framing for the application of the formal transform pair.
Therefore, in corollary~\ref{cor:AlternativeFormalTransformValid}, we present a slightly adjusted inverse transform and inversion theorem, wherein the Cauchy principal value bounding region of the contour is a disc in the $\la$ plane, as described in definition~\ref{defn:ComplexPrincipleValue}.

\begin{defn} \label{defn:ComplexPrincipleValue}
    Suppose $C$ is a contour in $\CC$, which extends to infinity, and $f$ is locally absolutely integrable on $C$.
    Then, wherever the limit exists, the \emph{principal value contour integral} of $f$ along $C$ is
    \[
        \CPV_{C} f(\la) := \lim_{\rho\to\infty} \int_C f(\la) \chi_{D(0,\rho)}(\la) \D\la,
    \]
    in which $\chi_{D(0,\rho)}$ represents the indicator function of the open disc centred at $0$ with radius $\rho$.

    By
    \[
        \CPVwithoutintegral \left( \int_{C_1}f_1(\la)\D\la + \int_{C_2}f_2(\la)\D\la \right),
    \]
    we shall mean the joint principal value
    \[
        \lim_{\rho\to\infty} \left( \int_{C_1} f_1(\la) \chi_{D(0,\rho)}(\la) \D\la + \int_{C_2} f_2(\la) \chi_{D(0,\rho)}(\la) \D\la \right).
    \]
\end{defn}

The above definition agrees with the Cauchy principal value for real Lebesgue integrals with singularities at infinity.

\begin{cor} \label{cor:AlternativeFormalTransformValid}
    Suppose $\phi \in \mathrm{C}^1[0,1]$.
    Then, for all $x\in(0,1)$,
    \[
        \frac{1}{2\pi} \;\CPV_{\gamma} \re^{\ri\nu(\la)x} \nu'(\la)F(\la) \D\la = \phi(x).
    \]
\end{cor}

\begin{proof}
    Let $\hat\gamma_\rho$ be the restriction of $\gamma$ to $D(0,\rho)$, and assume $\rho$ is sufficiently large that $\hat\gamma_\rho$ has only one connected component.
    Define $\sigma_\rho^+$ to be the startpoint and $\sigma_\rho^-$ to be the endpoint of the contour $\hat\gamma_\rho$.
    Then $\sigma_\rho^\pm$ are the only elements of the singletons
    \[
        C(0,\rho) \cap \gamma \cap \{ \la\in\CC: \pm\Re(\la)>0 \}.
    \]
    Therefore, exploiting the definition of $\gamma$ to justify the equality, $\nu(\sigma_\rho^\pm)$ are the only elements of the singletons
    \[
        \nu(C(0,\rho)) \cap \nu(\gamma) \cap \{ \nu(\la)\in\CC: \pm\Re(\la)>0 \}
        = \nu(C(0,\rho)) \cap \RR^\pm.
    \]
    By the limit in equations~\eqref{eqn:defnnu}, $\lim_{\rho\to\infty}\left(\nu(\sigma_\rho^\pm)\mp\rho\right)=0$.

    Using the same arguments as in the proof of theorem~\ref{thm:FormalTransformValid},
    \begin{multline*}
        \frac{1}{2\pi} \;\CPV_{\gamma} \re^{\ri\nu(\la)x} \nu'(\la)F(\la) \D\la
        = \frac{1}{2\pi} \lim_{\rho\to\infty}\int_{\hat\gamma_\rho} \re^{\ri\nu(\la)x} \nu'(\la)F(\la) \D\la \\
        = \frac{1}{2\pi} \;\CPV_{\RR} \re^{\ri\mu x} \int_0^1 \phi(y) \re^{-\ri\mu y} \D y \D\mu
        + \lim_{\rho\to\infty} \left(\left\{ \int_{\nu(\sigma_\rho^-)}^{-\rho} + \int_{\rho}^{\nu(\sigma_\rho^+)} \right\} \re^{\ri\mu x} \int_0^1 \phi(y) \re^{-\ri\mu y} \D y \D\mu\right),
    \end{multline*}
    and the first term on the right evaluates to $\phi(x)$.
    But
    \begin{align*}
        \abs{ \int_{\rho}^{\nu(\sigma_\rho^+)} \re^{\ri\mu x} \int_0^1 \phi(y) \re^{-\ri\mu y} \D y \D\mu }
        &\leq \abs{\nu(\sigma_\rho^+)-\rho} \max_{\substack{\mu\in[\nu(\sigma_\rho^+),\rho]\\\text{or}\\\mu\in[\rho,\nu(\sigma_\rho^+)]}}\abs{ \int_0^1 \phi(y) \re^{-\ri\mu y} \D y } \\
        &\leq \abs{\nu(\sigma_\rho^+)-\rho} \max_{y\in[0,1]} \abs{\phi(y)},
    \end{align*}
    and a similar estimate holds for the other integral.
    The result follows by proposition~\ref{prop:propertiesnu}.
\end{proof}

\subsection{A transform pair respecting the boundary conditions} \label{ssec:TransformPair.BC}

We denote by $L^\star$ the adjoint of $L$, as constructed in~\cite[chapter~11]{CL1955a}.
Precisely,
\begin{equation*}
    \mathcal{L}^\star \psi = \overline\omega(-\ri\partial_x)\psi = (-\ri)^n\phi^{(n)} + \sum_{j=0}^{n-2} (-\ri)^j \overline{a_j} \psi^{(j)},
\end{equation*}
for $\overline\omega$ the Schwarz conjugate of $\omega$, and, if $\BVec{B}^\star$ is a vector of boundary forms adjoint to $\BVec{B}$, with boundary coefficients
\begin{equation}
    B_k^\star \psi = \sum_{j=1}^n \left( \Msup{b}{k}{j}{\star}\psi^{(j-1)}(0) + \Msup{\beta}{k}{j}{\star} \psi^{(j-1)}(1) \right), \qquad k\in\{1,2,\ldots,n\},
\end{equation}
then $L^\star:\Phi_{\BVec{B}^\star}\to\mathrm{L}^1[0,1]$ is defined by $L^\star\psi = \mathcal{L}^\star\psi$.
For convenience, we separate the effects of the adjoint boundary forms at each boundary as
\begin{align*}
    \Msups{B}{k}{}{\star}{+} \psi &= \sum_{j=1}^n \Msup{b}{k}{j}{\star} \psi^{(j-1)}(0), & k &\in \{1,2,\ldots,n\}, \\
    \Msups{B}{k}{}{\star}{-} \psi &= \sum_{j=1}^n\Msup{\beta}{k}{j}{\star} \psi^{(j-1)}(1), & k &\in \{1,2,\ldots,n\}.
\end{align*}

\begin{eg} \label{eg:Separated3Ord.part2}
    We continue the analysis of example~\ref{eg:Separated3Ord.part1}.
    
    Because all coefficients of $\omega$ are real, $\omega$ is its own Schwarz conjugate.
    Therefore $\mathcal{L}^\star=\mathcal{L}$; the operator $L$ is formally selfadjoint.
    To determine adjoint boundary conditions, we integrate by parts in the inner product:
    \[
        \int_0^1\mathcal{L}\phi\bar \psi \D x = \int_0^1 \ri \phi'''\bar \psi \D x = \int_0^1 \phi \overline{\ri \psi'''} \D x + \ri \bigg[\phi'' \bar \psi - \phi'\overline{\psi'} + \phi\overline{\psi''}\bigg]_0^1 = \int_0^1\phi\overline{\mathcal{L}^\star\psi} \D x,
    \]
    with the last equality holding if and only if
    \[
        0 = \phi''(1) \overline{\psi(1)} - \phi'(1)\overline{\psi'(1)} + \phi(1)\overline{\psi''(1)} - \phi''(0) \overline{\psi(0)} + \phi'(0)\overline{\psi'(0)} - \phi(0)\overline{\psi''(0)}
    \]
    If $\phi$ is free in $\Phi_\BVec{B}$, then this is equivalent to $\psi\in\Phi_\BVec{B^\star}$ for
    \[
        B_1^\star\psi = \psi(0), \qquad B_2^\star\psi = \psi(1), \qquad B_3^\star\psi = \psi'(0).
    \]
    Therefore, $\Phi_\BVec{B^\star}\neq\Phi_\BVec{B}$; $L$ is nonselfadjoint despite being formally selfadjoint.
    
    This example is continued in example~\ref{eg:Separated3Ord.part3}.
\end{eg}

\subsubsection{The forward transform}

Let $\alpha=\re^{2\pi\ri/n}$, a primitive $n$\textsuperscript{th} root of unity.
For $j\in\{1,2,\ldots,n\}$, we define
\begin{equation} \label{eqn:defn.yj}
    y_j(x;\la) := \re^{\ri\overline{\nu(\alpha^{j-1}\la)}x}.
\end{equation}
Fixing $\la\in\CC\setminus D(0,\propertiesnuradius)$ and leaving $j$ free, each $y_j(\argdot;\la)$ is an eigenfunction of the formal differential operator $\mathcal{L}^\star$, with eigenvalue $\bar{\la}^n$, and together they span the eigenspace for that eigenvalue.
We use these solutions to construct the \emph{characteristic matrix} of the adjoint operator $M(\la)$ having entries
\begin{equation} \label{eqn:defn.Mjk}
    \M{M}{j}{k}(\la) := \overline{B_k^\star y_j(\argdot;\la)},
\end{equation}
and the \emph{characteristic determinant} of the adjoint operator
\begin{equation} \label{eqn:defn.Delta}
    \Delta(\la) := \det M(\la).
\end{equation}
Note the double complex conjugate of $\la$ in the definitions of $M$ and $\Delta$, implying that $M$ and $\Delta$ are both analytic functions of $\la$ for $\abs{\la}>\propertiesnuradius$.
Although these are objects defined in terms of the adjoint operator $L^\star$, we omit the superscript $^\star$ from the notation, because we have no need for the analogous objects one might construct from the original operator $L$.
For notational convenience, we define $\M{\MMinors}{j}{k}$ as the $(n-1)\times(n-1)$ submatrix of $\begin{bmatrix}M&M\\M&M\end{bmatrix}$ whose $(1,1)$ entry is the $(j+1,k+1)$ entry of its parent, and we also define the \emph{left characteristic matrix} $M^+$ and \emph{right characteristic matrix} $M^-$ to be those with entries
\begin{equation} \label{eqn:defn.Mpmjk}
    \Msup{M}{j}{k}{\pm}(\la) := \overline{\Msups{B}{k}{}{\star}{\pm} y_j(\argdot;\la)}
\end{equation}

Using the characteristic matrix, and the left and right characteristic matrices, we define
\begin{subequations} \label{eqn:FokasForwardTransformDefinition.LeftRight}
\begin{align}
    \FokForTransP{\phi}(\la) &= \frac{\nu'(\la)}{2\pi\Delta(\la)} \sum_{j=1}^n \sum_{k=1}^n (-1)^{(n-1)(j+k)} \det \M{\MMinors}{j}{k}(\la) \Msup{M}{1}{k}{+}(\la) \int_0^1 \re^{-\ri\nu(\alpha^{j-1}\la) x} \phi(x) \D x, \\
    \FokForTransM{\phi}(\la) &= \frac{-\nu'(\la)}{2\pi\Delta(\la)} \sum_{j=1}^n \sum_{k=1}^n (-1)^{(n-1)(j+k)} \det \M{\MMinors}{j}{k}(\la) \Msup{M}{1}{k}{-}(\la) \int_0^1 \re^{-\ri\nu(\alpha^{j-1}\la) x} \phi(x) \D x.
\end{align}
\end{subequations}
The transform $\FokForTransPMNoArg$ appears to be two forward transforms.
In practical application of the map, one always selects just one of $\FokForTransPNoArg$ or $\FokForTransMNoArg$ for each particular $\la$, so $\phi\mapsto \FokForTransPMNoArg$ may be considered as a single forward transform;
\begin{equation} \label{eqn:FokasForwardTransformDefinition}
    \phi(x) \mapsto \FokForTrans{\phi}(\la) := \begin{cases} \FokForTransP{\phi}(\la) & \mbox{certain } \la, \\ \FokForTransM{\phi}(\la) & \mbox{certain other } \la. \end{cases}
\end{equation}
Owing essentially to Cauchy's theorem and Jordan's lemma, there is significant freedom to choose which $\la$ correspond to which case, so we do not provide a fixed definition here, preferring to defer the choice until the definition of the inverse transform below.
Moreover, in establishing the validity of the transform pair, it will be useful to consider both formulae at certain values of $\la$, so it is important to note that both $\FokForTransP{\phi}(\la)$ and $\FokForTransM{\phi}(\la)$ are defined wherever $\abs{\la}>\propertiesnuradius$.

The integrals in equations~\eqref{eqn:FokasForwardTransformDefinition.LeftRight} are sesquilinear inner products $\langle \phi , y_j(\argdot; \la) \rangle$, so $\FokForTransPMNoArg$ are inner products of $\phi$ with particular eigenfunctions of $\mathcal{L}^\star$.
This is comparable with the forward transform for the Fourier sine series discussed in~\S\ref{ssec:Classical.FourierSeriesAsTransform}.
In contrast, because we target a diagonalization result of the form~\eqref{eqn:Classical.Sketch.WeakDiagonalization} instead of~\eqref{eqn:Classical.Sketch.Diagonalization}, we have taken a different particular member of each eigenspace and used a different, broader selection of eigenvalues.
The fact that this is the right definition of the forward transform to produce diagonalization~\eqref{eqn:Classical.Sketch.WeakDiagonalization} will be established in~\S\ref{sec:Diagonalization}.

\begin{eg} \label{eg:Separated3Ord.part3}
    We continue the analysis of example~\ref{eg:Separated3Ord.part2}.
    
    Following the definitions, $\alpha = \re^{2\pi\ri/3}=\frac{1}{2}(-1+\ri\sqrt{3})$.
    \[
        M^+(\la) = \begin{bmatrix} 1 & 0 & -\ri\nu(\la) \\ 1 & 0 & -\ri\nu(\alpha\la) \\ 1 & 0 & -\ri\nu(\alpha^2\la) \end{bmatrix},
        \qquad\qquad
        M^-(\la) = \begin{bmatrix} 0 & \re^{-\ri\nu(\la)} & 0 \\ 0 & \re^{-\ri\nu(\alpha\la)} & 0 \\ 0 & \re^{-\ri\nu(\alpha^2\la)} & 0 \end{bmatrix}.
    \]
    Therefore,
    \begin{equation} \label{eqn:Separated3Ord.Delta}
        \Delta(\la) = -\ri \sum_{j=1}^3 \re^{-\ri\nu(\alpha^{j-1}\la)} \left(\nu(\alpha^j\la)-\nu(\alpha^{j+1}\la)\right)
    \end{equation}
    and
    \begin{subequations} \label{eqn:Separated3Ord.FokForTransPM}
    \begin{align}
        \FokForTransP{\phi}(\la) &= \frac{-\ri\nu'(\la)}{2\pi\Delta(\la)} \sum_{j=1}^3 \left[ \left( \nu(\alpha^{j+1}\la)\re^{-\ri\nu(\alpha^j\la)} - \nu(\alpha^{j}\la)\re^{-\ri\nu(\alpha^{j+1}\la)} \right) \right. \notag \\ &\hspace{12em} + \left.\left( \re^{-\ri\nu(\alpha^{j+1}\la)} - \re^{-\ri\nu(\alpha^j\la)} \right) \nu(\la) \right] \FormalForTrans{\phi}(\alpha^{j-1}\la), \\
        \FokForTransM{\phi}(\la) &= \frac{\ri\nu'(\la)}{2\pi\Delta(\la)} \,\re^{-\ri\nu(\la)} \sum_{j=1}^3 \left( \nu(\alpha^{j}\la) - \nu(\alpha^{j+1}\la) \right) \FormalForTrans{\phi}(\alpha^{j-1}\la).
    \end{align}
    \end{subequations}
    
    This example is continued in example~\ref{eg:Separated3Ord.part4}.
\end{eg}

\subsubsection{The inverse transform}

The inverse transform must be an integral that can be applied to the result of the forward transform.
But the forward transform has denominator $\Delta$, whose zeros must be navigated by any prospective integration contour.
Let $y_j^\principalpart(x;\la) = \re^{\ri\overline{\alpha^{j-1}\la} x}$ and define $\Delta^\principalpart$ in terms of $y_j^\principalpart$ as was $\Delta$ defined in terms of $y_j$.
Then $\Delta^\principalpart$ is the chracteristic determinant of the principal part of the adjoint operator, and it follows from the definition that $\Delta^\principalpart$ is an exponential polynomial.
It may be that $\Delta^\principalpart$ is identically $0$ or another constant or polynomial multiplied by an explonential, but those cases correspond to illposed IBVP, so we shall assume $\Delta^\principalpart$ is an exponential polynomial with at least two terms, having different complex exponents.
It follows from the results of~\cite{Lan1931a} that the zeros of $\Delta$ have finite infimal separation (though any may be of finite order higher than $1$) and are asymptotically distributed along certain rays or logarithmic strips.
Moreover, in successively larger annuli whose inner and outer radii differ by the constant $\delta>0$, the number of zeros of $\Delta^\principalpart$ approaches some number $k(\delta)$.
Therefore, there is some $\epsilon>0$ such that $5\epsilon$ is less than the infimal separation of the zeros of $\Delta^\principalpart$, ensuring that the discs $D(\sigma,2\epsilon)$ for $\Delta^\principalpart(\sigma)=0$ have pairwise separation of at least $\epsilon$.

Select any $R$ at least $\epsilon$ greater than the $\propertiesnuradius$ of proposition~\ref{prop:propertiesnu}.
The arguments of~\cite{Lan1931a} may be adapted to show that all sufficiently large zeros of $\Delta$ can be made arbitrarily close to those of $\Delta^\principalpart$.
We assume that a sufficiently large $R$ is chosen that
\[
    D(0,R)\cup\bigcup_{\substack{\sigma\in\CC\setminus D(0,R):\\\Delta^\principalpart(\sigma)=0}} D(\sigma,\epsilon)
\]
contains all zeros of $\Delta$.
It may be necessary to further increase $R$ in order to ensure that $C(0,R)$ does not intersect any circles of radius $3\epsilon$ about zeros of $\Delta^\principalpart$.
We shall assume this is done, so that the contours defined below do not self intersect, and do not intersect one another.

We define the sets
\begin{subequations}
\begin{align}
    \CC_\nu^\pm &= \{\la\in\CC:\abs\la>\propertiesnuradius,\pm\Im(\nu(\la))>0\}, \\
    Z^+ &= \{\sigma\in\clos\CC_\nu^+:\Delta^\principalpart(\sigma)=0,\abs\sigma>R\}, \\
    Z^- &= \{\sigma\in\CC_\nu^-:\Delta^\principalpart(\sigma)=0,\abs\sigma>R\}, \\
    Z_\cuts &= \{\sigma\in\CC:\Delta^\principalpart(\sigma)=0,\abs\sigma<R\}.
\end{align}
We define the contour $\Gamma$ as follows.
\begin{align}
    \Gamma &= \Gamma_0 \cup \Gamma_a \cup \Gamma_\cuts, \\
    \Gamma_0 &= \Gamma_0^+ \cup \Gamma_0^-, \\
    \Gamma_0^\pm &= \bigcup_{\sigma\in Z^\pm} C(\sigma,\epsilon), \\
    \Gamma_a &= \Gamma_a^+ \cup \Gamma_a^-, \\
    \Gamma_a^\pm &= \partial \left( \{\la\in\CC_\nu^\pm:\Re(a\la^n)>0,\,\abs\la>R\} \setminus \bigcup_{\sigma\in Z^+\cup Z^-} D(\sigma,2\epsilon) \right), \\
    \Gamma_\cuts &= C(0,R) \\
    \Gamma^\pm &= \Gamma_0^\pm \cup \Gamma_a^\pm.
\end{align}
\end{subequations}
The contour $\Gamma$ is displayed in figure~\ref{fig:Gamma}.
\begin{figure}
    \centering
    \includegraphics{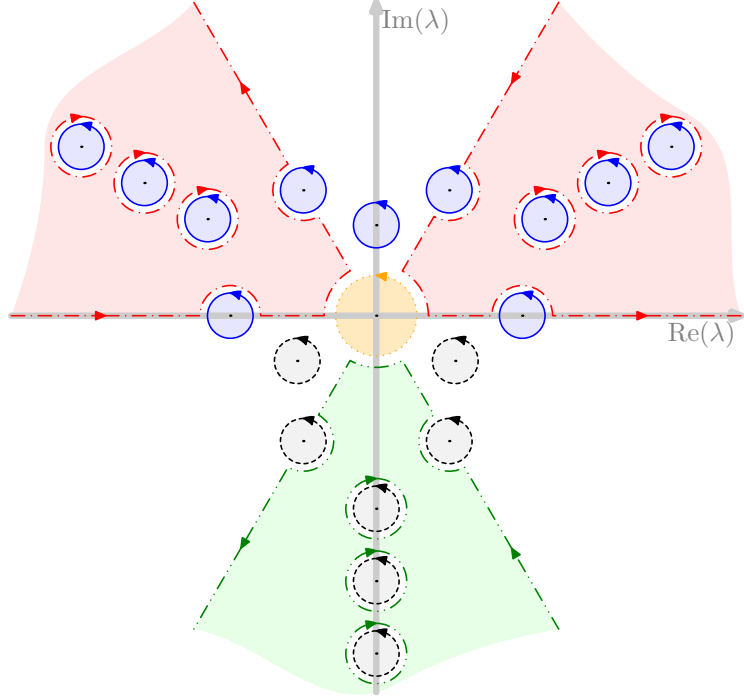}
    \caption{
        The components of contour $\Gamma$ are displayed:
            \boxGAMMAc{$\Gamma_\cuts$ dotted orange};
            \boxGAMMAaP{$\Gamma_a^+$ dot-dashed red};
            \boxGAMMAaM{$\Gamma_a^-$ dot-dot-dashed green};
            \boxGAMMAzP{$\Gamma_0^+$ solid blue};
            \boxGAMMAzM{$\Gamma_0^-$ dashed black}.
        The regions to the left of each component are shaded.
        The zeros of $\Delta^\principalpart$ are shown as black dots.
        In this example, $n=3$ and $a=-\ri$.
    }
    \label{fig:Gamma}
\end{figure}
If $\arg(a)$ is close to $\pm\pi/2$, then $\Gamma_a^\pm$ will include rays making small angles with the positive and negative real axes, as well as components lying on the boundaries of $\nu^{-1}(\RR)$, which approach $\RR$ itself.
If there are zeros of $\Delta^\principalpart$ closeby, then the removed discs may separate these components of $\Gamma_a^\pm$ into a number of simple loops.
Provided $\arg(a)\neq\pm\pi/2$, there are only finitely many such separated components, before the separation between the boundaries of $\nu^{-1}(\RR)$ and the rays exceeds $4\epsilon$; past that point $\Gamma_a^\pm$ has infinite components.
If $\arg(a)=\pm\pi/2$, then we have required the coefficients of $\omega$ be real, so $\nu^{-1}(\RR)\subseteq\RR$, and this difficulty is avoided.

For $F$, a bivalued function $F^\pm$ of a single complex variable, the inverse Fokas transform $\FokInvTransNoArg$ is defined by
\begin{equation} \label{eqn:FokasInverseTransformDefinition}
    F(\la) \mapsto \FokInvTrans{F}(x) := \CPVwithoutintegral \left( \int_{\Gamma^+\cup\Gamma_\cuts} \re^{\ri\nu(\la)x}F^+(\la) \D\la + \int_{\Gamma^-} \re^{\ri\nu(\la)x}F^-(\la) \D\la \right).
\end{equation}

\begin{eg} \label{eg:Separated3Ord.part4}
    We continue the analysis of example~\ref{eg:Separated3Ord.part3}.
    
    The characteristic determinant of the principal part, $\Delta^\principalpart$, is defined by the same formula as $\Delta$,~\eqref{eqn:Separated3Ord.Delta}, but with $\nu$ replaced by the identity function so that $\Delta^\principalpart$ is an exponential sum.
    By~\cite[proposition~A.1]{Pel2005a}, apart from the third order zero at $0$, $\Delta^\principalpart$ has simple zeros on the rays $-\ri\alpha^j\RR^+$, for $j\in\{0,1,2\}$, and no other zeros.
    Applying the arguments of~\cite{Lan1931a}, we find that the zeros of $\Delta^\principalpart$ are asymptotically at $-\ri\alpha^j\rho_k$, for $j\in\{0,1,2\}$ and positive integers $k$, where
    \[
        \rho_k = \frac{\pi}{\sqrt{3}}\left(\frac{1}{3}+2k\right).
    \]
    A numerical root finder based on the argument principle suggests that the zeros of $\Delta$ (and those of $\Delta^\principalpart$) all lie within discs of radius $0.1$ about the points $-\ri\alpha^j\rho_k$.
    Each circular component of $\Gamma_0^\pm$ corresponds to exactly one concentric circular component of $\Gamma_a^\pm$ of double the radius.
    Therefore, we cancel all the circular integrals (except that around $\Gamma_\cuts$) to simplify the expression.
    The fact that $\omega$ is Schwartz symmetric implies that $\CC^\pm_\nu=\CC^\pm$.
    
    From the above analysis, the contours can be simplified to
    \begin{subequations} \label{eqn:Separated3Ord.FokInvTransContours}
    \begin{align}
        \Gamma_\cuts &= C(0,0.8), \\
        \Gamma^+ &= \partial \left\{\la\in\CC:0<\arg(\la)<\frac\pi3 \mbox{ or } \frac\pi3<\arg(\la)<\frac{2\pi}3, \mbox{ and } \abs\la>1\right\}, \\
        \Gamma^- &= \partial \left\{\la\in\CC: \frac{-2\pi}3<\arg(\la)<\frac{-\pi}3, \mbox{ and } \abs\la>1\right\}.
    \end{align}
    \end{subequations}
    
    This example is continued in example~\ref{eg:Separated3Ord.part5}.
\end{eg}

\subsubsection{Validity theorem} \label{sssec:Validity}

We also define
\begin{equation} \label{eqn:defn.Gammaapmhat}
    \hat\Gamma_a^\pm = \partial \left( \{\la\in\CC_\nu^\pm:\Re(a\la^n)<0,\,\abs\la>R\} \setminus \bigcup_{\sigma\in Z^+\cup Z^-} D(\sigma,2\epsilon) \right).
\end{equation}

We claim that, under certain conditions, the transform~\eqref{eqn:FokasInverseTransformDefinition} is truly the inverse of transform~\eqref{eqn:FokasForwardTransformDefinition}.

\begin{asm} \label{asm:MainAssumption}
    The operator $L$, temporal coefficient $a$ and function $\phi$ are such that, for all $x\in(0,1)$,
    \begin{equation*}
        \CPVwithoutintegral\left( \int_{\hat\Gamma_a^+} \re^{\ri\nu(\la)x} \FokForTransP{\phi}(\la)\D\la + \int_{\hat\Gamma_a^-} \re^{\ri\nu(\la)x} \FokForTransM{\phi}(\la)\D\la \right) = 0.
    \end{equation*}
\end{asm}

\begin{thm} \label{thm:FokasTransformValid}
    Suppose $\phi\in \mathrm{C}^1[0,1]$ and assumption~\ref{asm:MainAssumption} holds.
    Then, for all $x\in(0,1)$,
    \begin{equation*}
        \FokInvTrans{\FokForTrans{\phi}}(x) = \phi(x).
    \end{equation*}
\end{thm}

The restriction to continuously differentiable functions is not essential, but is included to simplify the statement of the theorem.
Indeed, as with the usual Fourier inversion theorem, theorem~\ref{thm:FokasTransformValid} holds for $\phi$ piecewise continuously differentiable but with an average of the left and right limits at any discontinuities.
Similarly, if a smooth cutoff function is used in the principal value, then the inversion holds for $\phi$ piecewise continuous.
Extensions to spaces of integrable or square integrable functions, Sobolev spaces and distribution spaces are also possible with appropriate care in the definitions.
All such generalizations are inherited from the Fourier transform via corollary~\ref{cor:AlternativeFormalTransformValid}.
The restrictiveness of assumption~\ref{asm:MainAssumption} is discussed in~\S\ref{sssec:ValidityBirkhoff}.

\begin{lem} \label{lem:FokasTransformCancellation}
    For all $\phi\in \mathrm{L}^1[0,1]$ and for all $\la\in\CC$ with $\abs\la>\propertiesnuradius$,
    \[
        \FokForTransP{\phi}(\la) - \FokForTransM{\phi}(\la) = \frac{\nu'(\la)}{2\pi}\FormalForTrans{\phi}(\la).
    \]
\end{lem}

\begin{proof}
    This follows directly from the definitions of the transforms, and an application of the cyclic cofactor expansion of determinants to observe
    \[
        \sum_{k=1}^n(-1)^{(n-1)(j+k)}\det\M{\MMinors}{j}{k}(\la)\M{M}{1}{k}(\la) = \delta_{1,j}\Delta(\la),
    \]
    in which $\delta_{1,j}$ is the Kronecker delta.
\end{proof}

\begin{proof}[Proof of theorem~\ref{thm:FokasTransformValid}]
    Applying the hypothesis, we have
    \begin{equation*}
        \FokInvTrans{\FokForTrans{\phi}}(x)
        = \CPVwithoutintegral \left(
            \int_{\Gamma^+\cup\Gamma_\cuts\cup\hat\Gamma_a^+} \re^{\ri\nu(\la)x} \FokForTransP{\phi}(\la)\D\la
            + \int_{\Gamma^-\cup\hat\Gamma_a^-} \re^{\ri\nu(\la)x} \FokForTransM{\phi}(\la)\D\la
        \right).
    \end{equation*}
    Much of the contours $\hat\Gamma_a^\pm$ lie along contours $\Gamma_a^\pm$ but in the opposite direction, so their contributions to these integrals almost cancel.
    The remainders are integrals along the contours
    \[
        \check\Gamma^\pm = \partial \left( \{\la\in\CC_\nu^\pm:\abs\la>R\} \setminus \bigcup_{\sigma\in Z^+\cup Z^-} D(\sigma,2\epsilon) \right).
    \]
    Because $\phi\in \mathrm{L}^1[0,1]$ and $\epsilon,R$ have been chosen as described above, the integrands are holomorphic (within the radius of the principal value) on
    \[
        \CC \setminus \left( D(0,R) \cup \bigcup_{\sigma\in Z^+\cup Z^-} D(\sigma,\epsilon) \right),
    \]
    and continuous onto the boundary.
    Therefore, by Cauchy's theorem, we can make any finite contour deformation within this region without affecting the integrals.
    In particular, the contours $\check\Gamma^\pm$ can be replaced with the contours
    \[
        \tilde\Gamma^\pm = \partial \left( \{\la\in\CC_\nu^\pm:\abs\la>R\} \setminus \bigcup_{\sigma\in Z^+\cup Z^-} D(\sigma,\epsilon) \right),
    \]
    in which the radii of the excised discs have been reduced.
    This may appear to be an infinite deformation (as $Z^\pm$ are countably infinite sets) but, working within the principal value, it is a finite deformation.
    Using these new contours,
    \begin{equation*}
        \FokInvTrans{\FokForTrans{\phi}}(x)
        = \CPVwithoutintegral \left(
            \int_{\Gamma_0^+\cup\Gamma_\cuts\cup\tilde\Gamma^+} \re^{\ri\nu(\la)x} \FokForTransP{\phi}(\la)\D\la
            + \int_{\Gamma_0^-\cup\tilde\Gamma^-} \re^{\ri\nu(\la)x} \FokForTransM{\phi}(\la)\D\la
        \right).
    \end{equation*}

    For each $\sigma\in Z^\pm$ with $\abs{\Im(\nu(\sigma))}\geq\epsilon$, the contours $\Gamma_0^\pm\cup\tilde\Gamma^\pm$ include one circle $C(\sigma,\epsilon)$ oriented in each direction, so these parts of the contour integrals cancel out, and the contours may effectively be restricted to the region
    \begin{equation} \label{eqn:ValidityTheorem.RegionOfInterest}
        \{\la\in\CC:\abs\la>R,\,\abs{\Im(\nu(\la))}\leq\epsilon\}\cup\{\la\in\CC:\abs\la\leq R\}.
    \end{equation}
    It is here that we concentrate the remainder of the argument.
    The remaining integration contours are displayed in figure~\ref{fig:ValidityTheorem.1}.
    \begin{figure}
        \centering
        \includegraphics{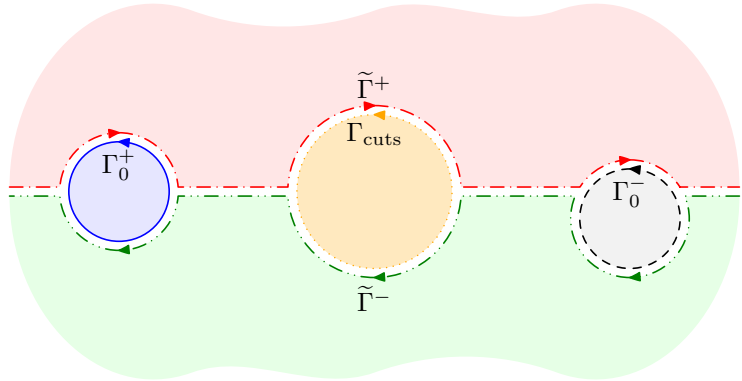}
        \caption{
            Remaining contours in the proof of theorem~\ref{thm:FokasTransformValid} after the first reduction:
                \boxGAMMAc{$\Gamma_\cuts$ dotted orange};
                \boxGAMMAaP{$\tilde\Gamma^+$ dot-dashed red};
                \boxGAMMAaM{$\tilde\Gamma^-$ dot-dot-dashed green};
                \boxGAMMAzP{$\Gamma_0^+$ solid blue};
                \boxGAMMAzM{$\Gamma_0^-$ dashed black}.
            The regions to the left of each contour are shaded.
            For simplicity of presentation, coefficients of $\omega$ are assumed real so that $\omega(\RR)\subseteq\RR$, and only parts of the contours lying within region of interest~\eqref{eqn:ValidityTheorem.RegionOfInterest} are displayed.
        }
        \label{fig:ValidityTheorem.1}
    \end{figure}
    \begin{figure}
        \centering
        \includegraphics{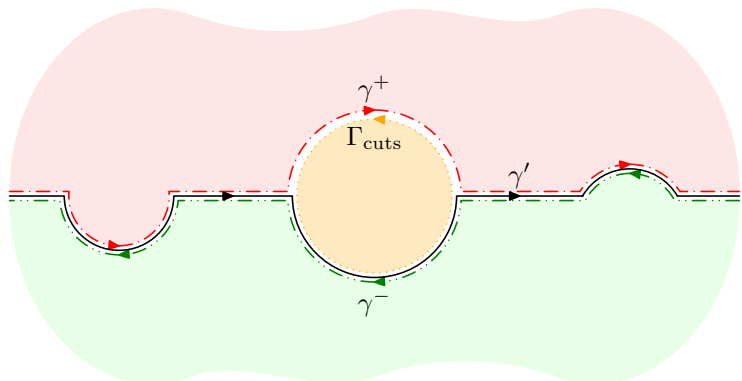}
        \caption{
            Remaining contours in the proof of theorem~\ref{thm:FokasTransformValid} after the second reduction:
                \boxGAMMAc{$\Gamma_\cuts$ dotted orange};
                \boxGAMMAaP{$\gamma^+$ dot-dashed red};
                \boxGAMMAaM{$\gamma^-$ dot-dot-dashed green}.
            The regions to the left of the first three contours are shaded.
            Also shown is the contour remaining after the third reduction: $\gamma'$ solid black.
        }
        \label{fig:ValidityTheorem.2}
    \end{figure}

    For each $\sigma\in Z^-$ with $\Im(\nu(\sigma))>-\epsilon$,
    \begin{align*}
        \CC_\nu^-\cap C(\sigma,\epsilon) &\mbox{ oriented in the positive sense, is part of } \Gamma_0^-, \\
        \CC_\nu^+\cap C(\sigma,\epsilon) &\mbox{ oriented in the positive sense, is also part of } \Gamma_0^-, \\
        \CC_\nu^-\cap C(\sigma,\epsilon) &\mbox{ oriented in the negative sense, is part of }  \tilde\Gamma^-, \\
        \CC_\nu^+\cap C(\sigma,\epsilon) &\mbox{ oriented in the negative sense, is part of }  \tilde\Gamma^+.
    \end{align*}
    The contributions of the first and third contours cancel, but the second and fourth do not because the integrands are different.
    A similar argument for each $\sigma\in Z^+$ with $\Im(\nu(\sigma))<\epsilon$ yields
    \begin{equation*}
        \FokInvTrans{\FokForTrans{\phi}}(x)
        = \CPVwithoutintegral \left(
            \int_{\gamma^+\cup\Gamma_\cuts} \re^{\ri\nu(\la)x} \FokForTransP{\phi}(\la)\D\la
            + \int_{\gamma^-} \re^{\ri\nu(\la)x} \FokForTransM{\phi}(\la)\D\la
        \right),
    \end{equation*}
    in which
    \begin{subequations} \label{eqn:FokasTransformValidProof.1}
    \begin{align}
        \gamma^+ &= \partial \left(\left(\CC_\nu^+ \cup \bigcup_{\substack{\sigma\in Z^+:\\\Im(\nu(\sigma))<\epsilon}}D(\sigma,\epsilon) \right) \setminus \bigcup_{\substack{\sigma\in Z^-:\\\Im(\nu(\sigma))>-\epsilon}}D(\sigma,\epsilon) \right), \\
        \gamma^- &= \partial \left(\left(\CC_\nu^- \cup \bigcup_{\substack{\sigma\in Z^-:\\\Im(\nu(\sigma))>-\epsilon}}D(\sigma,\epsilon) \right) \setminus \bigcup_{\substack{\sigma\in Z^+:\\\Im(\nu(\sigma))<\epsilon}}D(\sigma,\epsilon) \right).
    \end{align}
    \end{subequations}
    These contours are the same as one another except that they trace opposing negatively oriented arcs of $C(0,R)$, and they are oriented in opposite directions to one another.
    These integration contours are displayed in figure~\ref{fig:ValidityTheorem.2}.

    Observe that, on the contour $C(0,R)\cap\CC_\nu^+$, the contributions from integrals along contours $\gamma^+$ and $\Gamma_\cuts$ cancel exactly.
    Elsewhere, the integrals share a contour but have different integrands.
    Therefore, using the common contour
    \[
        \gamma' = \partial \left(\left(\CC_\nu^+ \cup D(0,R) \cup \bigcup_{\substack{\sigma\in Z^+:\\\Im(\nu(\sigma))<\epsilon}}D(\sigma,\epsilon) \right) \setminus \bigcup_{\substack{\sigma\in Z^-:\\\Im(\nu(\sigma))>-\epsilon}}D(\sigma,\epsilon) \right)
    \]
    indicated in figure~\ref{fig:ValidityTheorem.2} to combine the three integrals into one, we have
    \[
        \FokInvTrans{\FokForTrans{\phi}}(x) = \CPV_{\gamma'} \re^{\ri\nu(\la)x} \left( \FokForTransP{\phi}(\la) - \FokForTransM{\phi}(\la) \right) \D\la.
    \]
    Applying lemma~\ref{lem:FokasTransformCancellation},
    \[
        \FokInvTrans{\FokForTrans{\phi}}(x) = \frac{1}{2\pi}\, \CPV_{\gamma'} \re^{\ri\nu(\la)x} \nu'(\la)\FormalForTrans{\phi}(\la) \D\la.
    \]
    The integrand is holomorphic for $\abs\la>\propertiesnuradius$, so we can make the finite (or, rather, finite within the principal value) contour deformation
    \[
        \FokInvTrans{\FokForTrans{\phi}}(x) = \frac{1}{2\pi}\, \CPV_{\partial \left(\CC_\nu^+\setminus D(0,R)\right)} \re^{\ri\nu(\la)x} \nu'(\la)\FormalForTrans{\phi}(\la) \D\la.
    \]
    The result follows by corollary~\ref{cor:AlternativeFormalTransformValid}.
\end{proof}

\subsubsection{Validity criteria \& Birkhoff regularity} \label{sssec:ValidityBirkhoff}

A full interrogation of assumption~\ref{asm:MainAssumption} is beyond the scope of the present work, but we note that this assumption is not new.
An explicit statement of validity of this assumption in a related problem appears in~\cite[lemmata 2.1~\&~2.2]{MS2018a}, but equivalents are proved or assumed in every work on the Fokas transform method for evolution equations on the finite interval.
Nor is assumption~\ref{asm:MainAssumption} unreasonably restrictive.
We sketch below, in terms of Birkhoff's criteria for regularity, how one may derive sufficient but not necessary criteria for assumption~\ref{asm:MainAssumption} to hold.
As is appropriate for the solution of IBVP, we ensure that the assumption holds for all smooth enough $\phi$; assumption~\ref{asm:MainAssumption} should be understood as an assumption on $(L,a)$ valid for all $\phi$, not as an assumption on $(L,a,\phi)$.

For simplicity, we restrict to the case in which $\omega(\la)=\la^n$; $\nu(\la)=\la$.
In this case $\Delta=\Delta^\principalpart$ is an exponential polynomial, and its dominant term, as $\la\to\infty$ within the $J$\textsuperscript{th} sector (counting clockwise from an appropriate angle which depends on $n$) of $\{\la\in\CC:\Re(a\la^n)<0,\abs\la>R\}$, is given by
\[
    \det\begin{bmatrix}
        \Msup{M}{1}{1}{-} & \Msup{M}{1}{2}{-} & \cdots & \Msup{M}{1}{n}{-} \\
        \Msup{M}{2}{1}{-} & \Msup{M}{2}{2}{-} & \cdots & \Msup{M}{2}{n}{-} \\
        \vdots & \vdots &  & \vdots \\
        \Msup{M}{N}{1}{-} & \Msup{M}{N}{2}{-} & \cdots & \Msup{M}{N}{n}{-} \\
        \Msup{M}{N+1}{1}{+} & \Msup{M}{N+1}{2}{+} & \cdots & \Msup{M}{N+1}{n}{+} \\
        \Msup{M}{N+2}{1}{+} & \Msup{M}{N+2}{2}{+} & \cdots & \Msup{M}{N+2}{n}{+} \\
        \vdots & \vdots &  & \vdots \\
        \Msup{M}{n}{1}{+} & \Msup{M}{n}{2}{+} & \cdots & \Msup{M}{n}{n}{+}
    \end{bmatrix}(\alpha^J\la)
    = \mathcal{P}_N(\alpha^J\la)\exp\left(-\ri\la \alpha^J\sum_{j=1}^N\alpha^{j-1}\right),
\]
for
\[
    N = \begin{cases}
        \frac{n}{2} & \mbox{if } n \mbox{ even} \\
        \frac{n+1}{2} & \mbox{if } n \mbox{ odd and } a=\ri \\
        \frac{n-1}{2} & \mbox{if } n \mbox{ odd and } a = -\ri,
    \end{cases}
\]
in which the polynomial $\mathcal{P}_N$ is independent of $J$.
Birkhoff's regularity criteria are defined in terms of these same polynomials~\cite{Bir1908b}.
Indeed, a boundary form $B_k$ is said to have \emph{degree} $j-1$ if $j$ is maximal such that at least one of $\Msup{b}{k}{j}{\star},\Msup{\beta}{k}{j}{\star}$ is nonzero, and an even order operator $L^\star$ is \emph{Birkhoff regular} if $\mathcal{P}_{\frac{n}{2}}$ has degree equal to the sum of the degrees of all boundary forms.
An odd order operator must have both $\mathcal{P}_{\frac{n+1}{2}}$ and $\mathcal{P}_{\frac{n-1}{2}}$ of maximal degree to be called Birkhoff regular.

Integration by parts in the numerators of the expressions on the right of equations~\eqref{eqn:FokasForwardTransformDefinition.LeftRight} yields that the numerators have dominant terms
\[
    \frac{\mathscr{P}(\la)}{\la}\exp\left(-\ri\la \alpha^J\sum_{j=1}^N\alpha^{j-1}\right),
\]
in which the polynomials $\mathscr{P}$ have degree no greater than the sum of the degrees of all the boundary forms.
Therefore, at least away from the zeros of $\Delta$, Birkhoff regularity of $L^\star$ is sufficient to yield $\FokForTransPM{\phi}(\la)=\bigoh{\abs\la^{-1}}$ as $\la\to\infty$ in the sectors $\{\la\in\CC:\Re(a\la^n)<0,\abs\la>R\}$, for all $\phi\in\Phi$.

Note that, for odd order operators, one does not need the full strength of regularity to conclude this decay; only one of the two polynomials $\mathcal{P}_{\frac{n+1}{2}}$ and $\mathcal{P}_{\frac{n-1}{2}}$ need be of maximal degree to ensure that $\Delta$ dominates the numerator in the appropriate sectors, and which one is required depends on $a$.
An example of this phenomenon appears in IBVP~\eqref{eqn:3ordIBVP}, which is well posed but would be illposed if instead $a=\ri$.

With such a decay result, and assuming the zeros of $\Delta$ do not get in the way, one can apply Jordan's lemma to show that assumption~\ref{asm:MainAssumption} holds.
The locus of the zeros of $\Delta$ must be established separately.
Generally, it is acceptable for the zeros to lie within and on the boundaries of sectors $\{\la\in\CC:\Re(a\la^n)<0,\abs\la>R\}$, but there should not be infinitely many zeros interior to those sectors.
As the zeros of $\Delta$ are $n$\textsuperscript{th} roots of the eigenvalues of $L$, it is unsurprising that a criterion on their locus appears in the validity of any proposed spectral method for the solution of IBVP~\eqref{eqn:IBVP}.
For example, the instantaneous blowup of solutions of the heat equation with time running in reverse arises from exactly the issue of infinitely many zeros of $\Delta$ being aligned on the centre rays of sectors $\{\la\in\CC:\Re(a\la^n)<0,\abs\la>R\}$.
The restrictions on $a$ in~\S\ref{ssec:Introduction.IBVP} are designed to avoid some such cases, including the reverse time heat equation, but there are classes of boundary conditions for the zero potential linear Schr\"{o}dinger equation, wherein $a=\pm\ri$, featuring the same instantaneous blowup, so the analysis is a little more delicate than just the restrictions on $a$.

Although the above argument applies only to monomial $\omega$, the large $\la$ limit $\nu(\la)/\la\to1$ from proposition~\ref{prop:propertiesnu} indicates that the same kind of analysis may be valid for arbitrary $\omega$.
A full generalization is postponed for future work, but example~\ref{eg:Separated3Ord.part5} provides such an argument for one problem.

\begin{eg} \label{eg:Separated3Ord.part5}
    We continue the analysis of example~\ref{eg:Separated3Ord.part4}.
    
    Because of the locus of the zeros of $\Delta^\principalpart$ and using $\nu(\RR)\subset\RR$, the contours $\hat\Gamma_{-\ri}^\pm$ have a simpler expression:
    \[
        \hat\Gamma_{-\ri}^\pm = \partial \left\{\la\in\CC^\pm:\Re(-\ri\la^3)<0,\;\abs\la>1\right\}.
    \]
    Therefore, to justify assumption~\ref{asm:MainAssumption}, it is sufficient to show that, for all $\phi\in\mathrm{C}^1[0,1]$,
    \begin{equation} \label{eqn:Separated3Ord.Assumption}
        \CPV_{\hat\Gamma_{-\ri}^\pm} \re^{\ri\nu(\la)x}\FokForTransPM{\phi}(\la)\D\la = 0.
    \end{equation}
    
    In the ``$-$ version'' of equation~\eqref{eqn:Separated3Ord.Assumption}, the integrand is
    \[
        \re^{-\ri\nu(\la)(1-x)} \frac{\ri\nu'(\la)}{2\pi\Delta(\la)} \sum_{j=1}^3 \left( \nu(\alpha^{j}\la) - \nu(\alpha^{j+1}\la) \right) \FormalForTrans{\phi}(\alpha^{j-1}\la).
    \]
    By proposition~\ref{prop:propertiesnu}, $\nu'(\la)=\bigoh1$ as $\abs\la\to\infty$, uniformly in $\arg(\la)$.
    Integrating by parts in the definition of $\FormalForTransNoArg$, we find that the terms in the sum are exactly terms in $\Delta$ divided by $-\nu(\alpha^{j-1}\la)$.
    Therefore, the integrand is
    \[
        \re^{-\ri\nu(\la)(1-x)} \bigoh{\abs{\la}^{-1}}, \qquad \mbox{as }\la\to\infty \mbox{ within }\clos\CC^-,
    \]
    uniformly in $\arg(\la)$.
    Using again that $\nu(\RR)\subset\RR$ and biholomorphicity of $\nu$, we have that $\nu(\clos\CC^-)\subset\clos(\CC^-)$; $\re^{-\ri\nu(\la)(1-x)}$ is a valid kernel for Jordan's lemma.
    Hence, by Jordan's lemma, the ``$-$ version'' of equation~\eqref{eqn:Separated3Ord.Assumption} holds.
    
    The integrand in the ``$+$ version'' of equation~\eqref{eqn:Separated3Ord.Assumption} is
    \[
        \re^{\ri\nu(\la)x}\frac{\nu'(\la)}{2\pi\Delta(\la)} X(\la)\left[ 1+\bigoh{\abs\la^{-1}} \right],
    \]
    in which $X$ has terms of the form
    \begin{align*}
        &\frac{\nu(\alpha^{j+1}\la)-\nu(\la)}{\nu(\alpha^{j-1}\la)} \re^{-\ri\nu(\alpha^j\la)}, &
        &\frac{\nu(\alpha^{j+1}\la)-\nu(\la)}{\nu(\alpha^{j-1}\la)} \re^{-\ri[\nu(\alpha^j\la)+\nu(\alpha^{j-1}\la)]}, \\
        &\frac{\nu(\la)-\nu(\alpha^{j}\la)}{\nu(\alpha^{j-1}\la)} \re^{-\ri\nu(\alpha^{j+1}\la)}, &
        &\frac{\nu(\la)-\nu(\alpha^{j}\la)}{\nu(\alpha^{j-1}\la)} \re^{-\ri[\nu(\alpha^{j+1}\la)+\nu(\alpha^{j-1}\la)]},
    \end{align*}
    for $j\in\{0,1,2\}$.
    The terms on the left are each dominated by terms in $\Delta$.
    Therefore, the above argument applies to the terms on the left, except that we are now working in $\clos(\CC^+)$, so $\Im\nu(\la)\geq0$ and $\re^{\ri\nu(\la)x}$ is the right kernel for Jordan's lemma here.
    Of the terms on the right, the dominant ones are
    \begin{align*}
        &\begin{matrix} \mbox{top right,} \\ j=0, \end{matrix} &
        &\frac{\nu(\alpha\la)-\nu(\la)}{\nu(\alpha^{2}\la)} \re^{-\ri[\nu(\la)+\nu(\alpha^{2}\la)]}, &
        &\mbox{for} &
        \pi/3&\ll\arg(\la)\leq2\pi/3, \\
        &\begin{matrix} \mbox{bottom right,} \\ j=2, \end{matrix} &
        &\frac{\nu(\la)-\nu(\alpha^{2}\la)}{\nu(\alpha\la)} \re^{-\ri[\nu(\la)+\nu(\alpha\la)]}, &
        &\mbox{for} &
        \pi/3&\leq\arg(\la)\ll2\pi/3.
    \end{align*}
    Because $\nu(\clos\CC^-)\subset\clos(\CC^-)$, if $\pi/3\ll\arg(\la)\leq2\pi/3$, then $\Re(-\ri\nu(\alpha^{2}\la))\leq0$, so the top term is dominated by a term in $\Delta$.
    Similarly, if $\pi/3\leq\arg(\la)\ll2\pi/3$, then $\Re(-\ri\nu(\alpha\la))\leq0$, so the bottom term is dominated by a term in $\Delta$.
    Therefore, the ``$+$ version'' of equation~\eqref{eqn:Separated3Ord.Assumption} holds.
    
    We have justified that, for this example, assumption~\ref{asm:MainAssumption} holds for all $\phi\in\mathrm{C}^1[0,1]$.
    The validity theorem~\ref{thm:FokasTransformValid} follows.
    
    This example is continued in example~\ref{eg:Separated3Ord.part6}.
\end{eg}

\section{Diagonalization} \label{sec:Diagonalization}


In this section, we prove that the Fokas transform pair diagonalizes the differential operator $L$ in the sense of equations~\eqref{eqn:Classical.Sketch.WeakDiagonalization}.
Precisely, we prove the following theorem.

\begin{thm} \label{thm:Diag}
    There exists a \emph{remainder transform} $\FokRemTransNoArg$ such that, for all $\phi\in\Phi_{\BVec{B}}$,
    \begin{subequations} \label{eqn:DiagThm}
    \begin{equation} \label{eqn:DiagThm.1}
        \FokForTrans{L\phi}(\la) = \la^n \FokForTrans{\phi}(\la) + \FokRemTrans{\phi}(\la)
    \end{equation}
    and,
    if
    $q:[0,1]\times[0,T]\to\CC$ is such that, for all $t\in[0,T]$, $q(\argdot,t)\in\Phi$ and, for all $j\in\{0,1,\ldots,n-1\}$, uniformly for all $x\in[0,1]$, $\partial_x^jq(x,\argdot)$ is a function of bounded variation,
    then,
    for all $t\in[0,T]$ and all $x\in(0,1)$,
    \begin{equation} \label{eqn:DiagThm.2}
        \FokInvTrans{\int_0^t \re^{a\argdot^n(s-t)} \FokRemTrans{q}(\argdot;s) \D s}(x) = 0.
    \end{equation}
    \end{subequations}
\end{thm}

The proof of equation~\eqref{eqn:DiagThm.1} is essentially lemma~\ref{lem:DiagMainLemma}, where we also provide an explicit formula for the Fokas remainder transform $\FokRemTransNoArg$.
But first we need a polynomial growth bound on the entries of the left and right characteristic matries $M^\pm$.

\begin{lem} \label{lem:PolynomialGrowth}
    The function $\Msup{M}{j}{k}{+}(\la)$ is a polynomial in $\nu(\alpha^{j-1}\la)$ (so a holomorphic function of $\la$ outside a finite disc) and has polynomial growth rate $\bigoh{\abs{\la}^{n-1}}$, uniformly in $\arg(\la)$, as $\la\to\infty$.
    The function $\Msup{M}{j}{k}{-}(\la)$ is the product of $\re^{-\ri\nu(\alpha^{j-1}\la)}$ with another function; the latter function is a polynomial in $\nu(\alpha^{j-1}\la)$ (so a holomorphic function of $\la$ outside a finite disc) and has polynomial growth rate $\bigoh{\abs{\la}^{n-1}}$, uniformly in $\arg(\la)$, as $\la\to\infty$.
\end{lem}

\begin{proof}
    By definition,
    \[
        \Msup{M}{j}{k}{-}(\la) = \re^{-\ri\nu(\alpha^{j-1}\la)} \sum_{\ell=1}^n \overline{\Msup{\beta}{k}{\ell}{\star}}\left(-\ri\nu\left(\alpha^{j-1}\la\right)\right)^{\ell-1}.
    \]
    The result for $\Msup{M}{j}{k}{-}$ follows from proposition~\ref{prop:propertiesnu}.
    The argument is the same for $\Msup{M}{j}{k}{+}$.
\end{proof}

\begin{lem} \label{lem:DiagMainLemma}
    If $\phi\in\Phi_{\BVec{B}}$, then there exist functions $P_\phi^\pm$, analytic on the domain of $\nu$, with
    $P_\phi^\pm(\la) = \bigoh{\abs{\la}^{n-1}}$,
    uniformly in $\arg(\la)$, as $\la\to\infty$
    and
    \begin{subequations} \label{eqn:DiagMainLemma.ApproxEigs}
    \begin{align}
        \FokForTransP{L\phi} (\la) &= \la^n \FokForTransP{\phi}(\la) + P_\phi^+(\la), \\
        \FokForTransM{L\phi} (\la) &= \la^n \FokForTransM{\phi}(\la) + P_\phi^-(\la)\re^{-\ri\nu(\la)}.
    \end{align}
    \end{subequations}
    Moreover, the maps $\phi \mapsto P_\phi^\pm$ are linear.
    Indeed, such functions are given by
    \begin{subequations} \label{eqn:DiagMainLemma.PpmFormulae}
    \begin{align}
        \label{eqn:DiagMainLemma.PpmFormulae.+}
        P_\phi^+(\la) &= \frac{\nu'(\la)}{2\pi} \BVec{B}_\complementary\phi \odot \left( \overline{\Msup{M}{1}{1}{+}(\la)} , \overline{\Msup{M}{1}{2}{+}(\la)} , \ldots , \overline{\Msup{M}{1}{n}{+}(\la)} \right), \\
        \label{eqn:DiagMainLemma.PpmFormulae.-}
        P_\phi^-(\la) \re^{-\ri\nu(\la)} &= \frac{-\nu'(\la)}{2\pi} \BVec{B}_\complementary\phi \odot \left( \overline{\Msup{M}{1}{1}{-}(\la)} , \overline{\Msup{M}{1}{2}{-}(\la)} , \ldots , \overline{\Msup{M}{1}{n}{-}(\la)} \right),
    \end{align}
    \end{subequations}
    in which $\BVec{B}_\complementary$ represents the vector of boundary forms which is complementary (in the sense of~\cite[chapter~11]{CL1955a}) to the vector of boundary forms $\BVec{B}$ compatible with $\BVec{B}^\star$, and where $\odot$ is the usual sesquilinear dot product on $\CC^{n}$.
\end{lem}

\begin{proof}
    For notational convenience, we define functions $\psi^\pm$ via their complex conjugates
    \begin{equation} \label{eqn:DiagMainLemma.Proof.psipm}
        \overline{\psi^\pm(x;\la)} := \frac{\pm\nu'(\la)}{2\pi\Delta(\la)} \sum_{j=1}^n \sum_{\ell=1}^n (-1)^{(n-1)(j+\ell)} \det \M{\MMinors}{j}{\ell}(\la) \Msup{M}{1}{\ell}{\pm}(\la) \re^{-\ri\nu(\alpha^{j-1}\la) x},
    \end{equation}
    so that
    \begin{equation} \label{eqn:DiagMainLemma.Proof.InnerProduct}
        \FokForTransPM{\phi}(\la) = \left\langle \phi , \psi^\pm(\argdot;\la) \right\rangle,
    \end{equation}
    for $\langle\argdot,\argdot\rangle$ the usual sesquilinear inner product on $\mathrm{L}^2[0,1]$.
    Note that, for fixed $\la$, $\psi^\pm(\argdot;\la)$ are linear combinations $y_j(\argdot;\la)$ for $j\in\{1,2,\ldots,n\}$, so $\psi^\pm(\argdot;\la)$ are themselves eigenfunctions of the formal differential operator $\mathcal{L}^\star$ with eigenvalue $\bar\la^n$.
    Note also that
    \[
        \overline{B_k^\star\psi^\pm(\argdot;\la)\vphantom{()}} = \overline{B_k^\star\vphantom{()}} \; \overline{\psi^\pm(\argdot;\la)} = \frac{\pm\nu'(\la)}{2\pi\Delta(\la)} \sum_{j=1}^n \sum_{\ell=1}^n (-1)^{(n-1)(j+\ell)} \det \M{\MMinors}{j}{\ell}(\la) \Msup{M}{1}{\ell}{\pm}(\la) \overline{B_k^\star\vphantom{()}} \re^{-\ri\nu(\alpha^{j-1}\la) \argdot},
    \]
    but
    \[
        \overline{B_k^\star\vphantom{()}} \re^{-\ri\nu(\alpha^{j-1}\la) \argdot} = \overline{B_k^\star\vphantom{()}} \; \overline{y_j(\argdot;\la)} = \M{M}{j}{k}(\la)
    \]
    and
    \[
        \sum_{j=1}^n (-1)^{(n-1)(j+\ell)} \det \M{\MMinors}{j}{\ell}(\la) \M{M}{j}{k} = \delta_{\ell,k}\Delta(\la),
    \]
    for $\delta_{\ell,k}$ the Kronecker delta.
    Hence
    \begin{equation} \label{eqn:DiagMainLemma.Proof.AdjBFpsi}
        \overline{B_k^\star\psi^\pm(\argdot;\la)} = \frac{\pm\nu'(\la)}{2\pi} \Msup{M}{1}{k}{\pm}(\la)
    \end{equation}

    Using equation~\eqref{eqn:DiagMainLemma.Proof.InnerProduct}, and the results, terminology and notation of~\cite[chapter~11]{CL1955a}, we find
    \[
        \FokForTransPM{L\phi}(\la) = \left\langle L\phi , \psi^\pm(\argdot;\la) \right\rangle = \left\langle \phi , \mathcal{L}^\star\psi^\pm(\argdot;\la) \right\rangle + \BVec{B}\phi \odot \BVec{B}_\complementary^\star\psi^\pm(\argdot;\la) + \BVec{B}_\complementary\phi \odot \BVec{B}^\star\psi^\pm(\argdot;\la),
    \]
    for $\BVec{B}_\complementary^\star$ the vector of boundary forms complementary to $\BVec{B}^\star$ compatible with $\BVec{B}$ and $\BVec{B}_\complementary$.
    Because $\phi\in\Phi_{\BVec{B}}$, using the earlier observation the $\psi^\pm$ are eigenfunctions of $\mathcal{L}^\star$, and applying equation~\eqref{eqn:DiagMainLemma.Proof.AdjBFpsi} we find
    \begin{equation} \label{eqn:DiagMainLemma.Proof.PpmExplicitFormulae}
        \FokForTransPM{L\phi}(\la) = \la^n \left\langle \phi , \psi^\pm(\argdot;\la) \right\rangle \pm\frac{\nu'(\la)}{2\pi} \BVec{B}_{\complementary}\phi \odot \left( \overline{\Msup{M}{1}{1}{\pm}(\la)} , \overline{\Msup{M}{1}{2}{\pm}(\la)} , \ldots , \overline{\Msup{M}{1}{n}{\pm}(\la)} \right).
    \end{equation}
    The approximate eigenfunction equations~\eqref{eqn:DiagMainLemma.ApproxEigs} follow, for $P^\pm_\phi$ defined as in equations~\eqref{eqn:DiagMainLemma.PpmFormulae}.
    In these formulae, each sesquilinear dot product $\odot$ evaluates to a linear combination of $\Msup{M}{1}{k}{\pm}(\la)$ for $k\in\{1,2,\ldots,n\}$.
    Recall from proposition~\ref{prop:propertiesnu} that $\nu'(\la)=\bigoh{1}$ as $\la\to\infty$.
    The analyticity and asymptotic results follow by lemma~\ref{lem:PolynomialGrowth}.

    Observe from equations~\eqref{eqn:DiagMainLemma.Proof.PpmExplicitFormulae} that the functions $P_\phi^\pm$ depend on $\phi$ only through $\BVec{B}_\complementary\phi$, in which the vector of complementary boundary forms $\BVec{B}_\complementary:\AC^{n-1}\to\CC^n$ is a linear map.
    Therefore, $P_\phi^\pm$ also depend linearly on $\phi$.
\end{proof}

\begin{proof}[Proof of theorem~\ref{thm:Diag}]
    By lemma~\ref{lem:DiagMainLemma}, defining $\FokRemTransNoArg$ by
    \begin{equation} \label{eqn:Diag.RemTransDefn}
        \FokRemTrans{\phi}(\la) =
        \begin{cases}
            P^+_\phi(\la) & \la \in\Gamma^+\cup\Gamma_\cuts, \\
            P^-_\phi(\la) \re^{-\ri\nu(\la)} & \la\in\Gamma^-,
        \end{cases}
    \end{equation}
    equation~\eqref{eqn:DiagThm.1} follows immediately.

    To establish equation~\eqref{eqn:DiagThm.2}, we will show that each term on the right of
    \begin{multline*}
        \FokInvTrans{\int_0^t \re^{a\argdot^n(s-t)} \FokRemTrans{q}(\argdot;s) \D s}(x) \\
        = \CPV_{\Gamma_a^+} \re^{\ri\nu(\la)x-a\la^nt} \int_0^t \re^{a\la^ns}  P^+_{q(\argdot,s)}(\la) \D s \D\la
        + \CPV_{\Gamma_a^-} \re^{\ri\nu(\la)(x-1)-a\la^nt} \int_0^t \re^{a\la^ns}  P^-_{q(\argdot,s)}(\la) \D s \D\la \\
        + \CPV_{\Gamma_0^+} \re^{\ri\nu(\la)x-a\la^nt} \int_0^t \re^{a\la^ns} P^+_{q(\argdot,s)}(\la) \D s \D\la
        + \CPV_{\Gamma_0^-} \re^{\ri\nu(\la)(x-1)-a\la^nt} \int_0^t \re^{a\la^ns} P^-_{q(\argdot,s)}(\la) \D s \D\la \\
        + \int_{\Gamma_\cuts} \re^{\ri\nu(\la)x-a\la^nt} \int_0^t \re^{a\la^ns} P^+_{q(\argdot,s)}(\la) \D s \D\la
    \end{multline*}
    evaluates to $0$.
    The termwise evaluation will also justify the above distribution of the principal value.
    We treat these terms in the order they appear above, beginning with the integrals along $\Gamma_a^\pm$, which require the full criteria of the theorem.

    By definition, each component of $\BVec{B}_\complementary q(\argdot,t)$ is a linear combination of the $2n$ terms $\partial_x^jq(0,t)$ and $\partial_x^jq(1,t)$, which are all functions of bounded variation.
    Hence each component of $\BVec{B}_\complementary q(\argdot,t)$ is itself a function of bounded variation.
    It follows from the formulae for $P^\pm_{q(\argdot,s)}(\la)$ given in lemma~\ref{lem:DiagMainLemma} that, for each $\la$ in the domain of $\nu$, $P^\pm_{q(\argdot,s)}(\la)$ is a function of bounded variation in $s$.
    Moreover, by the asymptotic properties of $\Msup{M}{j}{k}{\pm}$ given in lemma~\ref{lem:PolynomialGrowth}, the total variation in $s$
    \[
        V_0^T\left( P^\pm_{q(\argdot,s)}(\la) \right) = \bigoh{\abs{\la}^{n-1}},
    \]
    uniformly in $\arg(\la)$, as $\la\to\infty$.
    Therefore
    \begin{align*}
        \abs{\frac{1}{\la^n}\int_0^t \re^{-a\la^n(t-s)} \D P^\pm_{q(\argdot,s)}(\la)}
        &\leq \frac{\sup_{s\in[0,t]}\abs{\re^{-a\la^n(t-s)}}}{\abs{\la^n}} V_0^t\left( P^\pm_{q(\argdot,s)}(\la) \right) \\
        &\leq \frac{V_0^T\left( P^\pm_{q(\argdot,s)}(\la) \right)}{\abs{\la^n}}
        = \bigoh{\abs{\la}^{-1}},
    \end{align*}
    uniformly in $\arg(\la)$, as $\la\to\infty$ from within the closure of the region enclosed by $\Gamma_a^\pm$.
    Therefore, applying lemma~\ref{lem:DiagMainLemma},
    \begin{align}
        \notag
        \re^{-a\la^nt} \int_0^t \re^{a\la^ns}  P^\pm_{q(\argdot,s)}(\la) \D s
        &= \frac{1}{a\la^n}\left( P^\pm_{q(\argdot,t)}(\la) - \re^{-a\la^nt} P^\pm_{q(\argdot,0)}(\la) - \int_0^t \re^{-a\la^n(t-s)} \D
        P^\pm_{q(\argdot,s)}(\la) \right) \\
        \label{eqn:DiagThmProof.2}
        &= \bigoh{\abs{\la}^{-1}},
    \end{align}
    uniformly in $\arg(\la)$, as $\la\to\infty$ within the closure of the region to the left of $\Gamma_a^\pm$.
    By Morera's theorem and lemma~\ref{lem:DiagMainLemma}, the integral on the left of equations~\eqref{eqn:DiagThmProof.2} is holomorphic on the domain of $\nu$.
    Therefore, by Jordan's lemma,
    \begin{align*}
        \CPV_{\Gamma_a^+} \re^{\ri\nu(\la)x-a\la^nt} \int_0^t \re^{a\la^ns}  P^+_{q(\argdot,s)}(\la) \D s \D\la &= 0, \\
        \CPV_{\Gamma_a^-} \re^{\ri\nu(\la)(x-1)-a\la^nt} \int_0^t \re^{a\la^ns}  P^-_{q(\argdot,s)}(\la) \D s \D\la &= 0,
    \end{align*}
    with the first equation holding for all $x>0$ and the second for all $x<1$.

    The fact that the integral on the left of equations~\eqref{eqn:DiagThmProof.2} is holomorphic on the domain of $\nu$ also implies, via Cauchy's theorem, that
    \begin{align*}
        \CPV_{\Gamma_0^+} \re^{\ri\nu(\la)x-a\la^nt} \int_0^t \re^{a\la^ns} P^+_{q(\argdot,s)}(\la) \D s \D\la &= 0, \\
        \CPV_{\Gamma_0^-} \re^{\ri\nu(\la)(x-1)-a\la^nt} \int_0^t \re^{a\la^ns} P^-_{q(\argdot,s)}(\la) \D s \D\la &= 0.
    \end{align*}

    Substituting formula~\eqref{eqn:DiagMainLemma.PpmFormulae.+} for $P^+_{q(\argdot,s)}(\la)$ and making the change of variables $\mu=\nu(\la)$, we obtain
    \begin{align*}
        \int_{\Gamma_\cuts} &\re^{\ri\nu(\la)x-a\la^nt} \int_0^t \re^{a\la^ns} P^+_{q(\argdot,s)}(\la) \D s \D \la \\
        = &\int_{\Gamma_\cuts} \re^{\ri\nu(\la)x-a\la^nt} \int_0^t \re^{a\la^ns} \BVec{B}_\complementary q(\argdot,s) \odot \left( \overline{\Msup{M}{1}{1}{+}(\la)} , \overline{\Msup{M}{1}{2}{+}(\la)} , \ldots , \overline{\Msup{M}{1}{n}{+}(\la)} \right) \D s \frac{\nu'(\la)}{2\pi} \D\la \\
        = &\int_{\gamma_\cuts} \re^{\ri\mu x-a\omega(\mu)t} \int_0^t \re^{a\omega(\mu)s} \\
        &\hspace{6em} \BVec{B}_\complementary q(\argdot,s) \odot \left( \overline{\Msup{M}{1}{1}{+}(\nu^{-1}(\mu))} , \overline{\Msup{M}{1}{2}{+}(\nu^{-1}(\mu))} , \ldots , \overline{\Msup{M}{1}{n}{+}(\nu^{-1}(\mu))} \right) \D s \frac{1}{2\pi} \D\mu,
    \end{align*}
    in which $\nu^{-1}$ is the inverse of $\nu$ and $\gamma_\cuts$ is the simple closed loop enclosing the origin which is the image of $\Gamma_\cuts$ in $\nu$.
    By lemma~\ref{lem:PolynomialGrowth}, for each $k\in\{1,2,\ldots,n\}$, $\Msup{M}{1}{k}{+}(\nu^{-1}(\mu))$ is an entire function of $\mu$, and the apparent complex conjugates are artifacts of presentation via the sesquilinear dot product $\odot$.
    So the integrand is an entire function of $\mu$ and, by Cauchy's theorem,
    \[
        \int_{\Gamma_\cuts} \re^{\ri\nu(\la)x-a\la^nt} \int_0^t \re^{a\la^ns} P^+_{q(\argdot,s)}(\la) \D s \D\la = 0.
        \qedhere
    \]
\end{proof}

\begin{eg} \label{eg:Separated3Ord.part6}
    We continue the analysis of example~\ref{eg:Separated3Ord.part5}.
    
    The crucial objects in the diagonalization theorem are the functions $P_\phi^\pm$ whose polynomial growth bound is asserted in lemma~\ref{lem:DiagMainLemma}.
    In the proof of that lemma, the formulae~\eqref{eqn:DiagMainLemma.PpmFormulae} for $P_\phi^\pm$ were derived using Green's formula (corollary to theorem~3.6.3) and the boundary form formula (theorem~11.2.1), both from~\cite{CL1955a}.
    But for particular examples, those theorems reduce to integration by parts and some algebraic bookkeeping, so we shall proceed directly.
    
    We aim to rewrite $\FokForTransPM{L\phi}(\la)$ as $\la^3\FokForTransPM{\phi}$ plus a remainder term.
    To do so, we express $\FokForTransPM{L\phi}(\la)$ as the inner product of $L\phi$ with a particular $\psi^\pm(\argdot;\la)$ that represents the kernel of $\FokForTransPMNoArg$.
    By the calculation in example~\ref{eg:Separated3Ord.part2},
    \begin{equation} \label{eqn:Separated3Ord.Remainders1}
        \int_0^1L\phi\overline{\psi^\pm} \D x
        = \int_0^1\phi\overline{\mathcal{L}^\star\psi^\pm}\D x
        + \ri\left[ \phi''(1)\overline{\psi^\pm(1)} - \phi''(0)\overline{\psi^\pm(0)} + \phi'(0)\overline{\psi^{\pm\prime}(0)} \right].
    \end{equation}
    The result of example~\ref{eg:Separated3Ord.part3} tells us that, for the left side to represent $\FokForTransPM{L\phi}(\la)$, we have to define
    \begin{align*}
        \overline{\psi^+(x;\la)} &= \frac{-\ri\nu'(\la)}{2\pi\Delta(\la)} \sum_{j=1}^3 \left[ \left( \nu(\alpha^{j+1}\la)\re^{-\ri\nu(\alpha^j\la)} - \nu(\alpha^{j}\la)\re^{-\ri\nu(\alpha^{j+1}\la)} \right) \right. \\ &\hspace{12em} + \left.\left( \re^{-\ri\nu(\alpha^{j+1}\la)} - \re^{-\ri\nu(\alpha^j\la)} \right) \nu(\la) \right] \re^{-\ri\nu(\alpha^{j-1}\la) x}, \\
        \overline{\psi^-(x;\la)} &= \frac{\ri\nu'(\la)}{2\pi\Delta(\la)} \,\re^{-\ri\nu(\la)} \sum_{j=1}^3 \left( \nu(\alpha^{j}\la) - \nu(\alpha^{j+1}\la) \right) \re^{-\ri\nu(\alpha^{j-1}\la) x}.
    \end{align*}
    So $\psi^\pm(x;\la)$ are each linear combinations of functions $\overline{\re^{-\ri\nu(\alpha^{j-1}\la) x}}$, which are each eigenfunctions of $\mathcal{L}^\star$ with common eigenvalue $\bar{\la}^n$, so the inner product on the right of equation~\eqref{eqn:Separated3Ord.Remainders1} simplifies to
    \[
        \la^n\int_0^1\phi\overline{\psi^\pm} \D x = \la^n\FokForTransPM{\phi}(\la).
    \]
    Evaluating $\overline{\psi^\pm(x;\la)}$ and its derivative at $x=0,1$ and reindexing the sums appropriately, we find that
    \begin{align*}
        \overline{\psi^+(1;\la)} &= 0, & \overline{\psi^+(0;\la)} &= \frac{\nu'(\la)}{2\pi}, & \overline{\psi^{+\prime}(0;\la)} &= -\ri\nu(\la)\frac{\nu'(\la)}{2\pi}, \\
        \overline{\psi^-(1;\la)} &= \re^{-\ri\nu(\la)}\frac{\nu'(\la)}{2\pi}, & \overline{\psi^-(0;\la)} &= 0, & \overline{\psi^{-\prime}(0;\la)} &= 0.
    \end{align*}
    Therefore, equation~\eqref{eqn:Separated3Ord.Remainders1} simplifies to
    \begin{subequations} \label{eqn:Separated3Ord.Remainders2}
    \begin{align}
        \FokForTransP{L\phi}(\la) &= \la^n\FokForTransP{\phi}(\la) + \frac{\nu'(\la)}{2\pi} \left[ \M{B}{\complementary}{1}\phi \times 1 + \M{B}{\complementary}{2}\phi \times 0 + \M{B}{\complementary}{3}\phi \times (-\ri\nu(\la)) \right], \\
        \FokForTransM{L\phi}(\la) &= \la^n\FokForTransM{\phi}(\la) - \re^{-\ri\nu(\la)} \frac{\nu'(\la)}{2\pi} \left[ \M{B}{\complementary}{1}\phi \times 0 + \M{B}{\complementary}{2}\phi \times 1 + \M{B}{\complementary}{3}\phi \times 0 \right],
    \end{align}
    \end{subequations}
    for complimentary boundary conditions
    \begin{equation} \label{eqn:Separated3Ord.BFcomplementary}
        \M{B}{\complementary}{1}\phi = -\ri\phi''(0), \qquad\qquad \M{B}{\complementary}{2}\phi = -\ri\phi''(1), \qquad\qquad \M{B}{\complementary}{3}\phi = \ri\phi'(0).
    \end{equation}
    
    Equations~\eqref{eqn:Separated3Ord.Remainders2} are equations~\eqref{eqn:DiagMainLemma.Proof.PpmExplicitFormulae}, and they yield
    \[
        P_\phi^+(\la) = \frac{\nu'(\la)}{2\pi}\left( \M{B}{\complementary}{1}\phi - \ri\nu(\la)\M{B}{\complementary}{3}\phi \right)
        \qquad \mbox{and} \qquad
        P_\phi^-(\la) = - \frac{\nu'(\la)}{2\pi} \M{B}{\complementary}{2}\phi
    \]
    for this particular IBVP.
    The $\bigoh{\abs\la^{3-1}}$ growth bound of lemma~\ref{lem:DiagMainLemma} is immediate from proposition~\ref{prop:propertiesnu}.
    In this example, a stronger $\bigoh{\abs\la}$ growth bound applies, but that is not usually true.
    
    The analysis is continued in example~\ref{eg:Separated3Ord.part7}.
\end{eg}

\begin{rmk}
    The statement of lemma~\ref{lem:DiagMainLemma} requires that the vector of complementary boundary forms $\BVec{B}_\complementary$ be compatible with $\BVec{B}$ and $\BVec{B}^\star$.
    This precise compatibility condition is why (see equation~\eqref{eqn:Separated3Ord.BFcomplementary}) we have coefficients other than one in the complementary boundary forms.
    It might appear more natural to have the alternate complementary boundary forms $\M{\tilde B}{\complementary}{1}\phi = \phi''(0)$, $\M{\tilde B}{\complementary}{2}\phi = \phi''(1)$, $\M{\tilde B}{\complementary}{3}\phi = \phi'(0)$, which are just a linear transformation of the complementary boundary forms we selected.
    But then the adjoint boundary forms would undergo an according linear transformation, changing the definition of $\FokForTransPMNoArg$.
    
    In~\cite{CL1955a}, it is the complementary boundary forms that may be chosen, thereby specifying precisely the adjoint boundary forms and their complements.
    But it is clear from the integration by parts argument that one can equivalently choose the adjoint boundary forms, which imposes the complements of both.
    In this work, we chose the latter approach because, while the adjoint boundary forms are necessary in constructing the Fokas transforms, our approach obviates explicit expressions for the complimentary boundary forms: although the concept was required in lemma~\ref{lem:DiagMainLemma}, the diagonalization theorem~\ref{thm:Diag} does not need the complimentary boundary forms.
    The last claim is true for homogeneous IBVP only; solving an inhomogeneous IBVP does require adjoint complementary boundary forms, as shown in the inhomogeneous diagonalization theorem~\ref{thm:Inhomogeneous.Diag}.
\end{rmk}

\section{Fokas transform method} \label{sec:Method}

\subsection{True transform version} \label{ssec:Method.TrueTransform}

Theorem~\ref{thm:FokasMethodSolRep} gives the solution of IBVP~\eqref{eqn:IBVP} using the Fokas transform pair.
Its proof reads like an ordinary transform method, but it relies on the diagonalization theorem we proved above.

\begin{thm} \label{thm:FokasMethodSolRep}
    Suppose the criteria of theorem~\ref{thm:FokasTransformValid} hold and there exists a solution $q(x,t)$ of problem~\eqref{eqn:IBVP}, which is absolutely continuous in $t$ and satisfying the criteria of theorem~\ref{thm:Diag}.
    Then
    \[
        q(x,t) = \FokInvTrans{\re^{-at\argdot^n}\FokForTrans{Q}}(x).
    \]
\end{thm}

\begin{proof}
    We apply the forward Fokas transform to the partial differential equation~\eqref{eqn:IBVP.PDE}.
    Then, by linearity of the forward Fokas transform, almost everywhere in $t$,
    \[
        0 = \FokForTrans{\partial_tq}(\la;t) + a \FokForTrans{Lq}(\la;t) = \frac{\D}{\D t} \FokForTrans{q}(\la;t) + a \FokForTrans{Lq}(\la;t).
    \]
    Hence, by theorem~\ref{thm:Diag},
    \begin{equation*}
        0 = \frac{\D}{\D t} \FokForTrans{q}(\la;t) + a \la^n \FokForTrans{q}(\la;t) + a\FokRemTrans{q}(\la;t).
    \end{equation*}
    This class of first order linear ordinary differential equations, parametrized by $\la$, can be equipped with initial conditions by applying the forward Fokas transforms to the initial condition~\eqref{eqn:IBVP.IC}; the initial value problems dervied thus have solutions
    \begin{equation} \label{eqn:FokasMethodSolRep.Proof.SolvedODE}
        \FokForTrans{q}(\la;t) = \re^{-a\la^nt}\FokForTrans{Q}(\la) - a \re^{-a\la^nt} \int_0^t \re^{a\la^ns} \FokRemTrans{q}(\la;s) \D s,
    \end{equation}
    for all $\la$ in the domain of $\nu$.

    We apply the inverse Fokas transform to equations~\eqref{eqn:FokasMethodSolRep.Proof.SolvedODE}.
    By theorem~\ref{thm:FokasTransformValid} and linearity of the inverse transform, it follows that
    \begin{equation*}
        q(x,t) = \FokInvTrans{\re^{-at\argdot^n}\FokForTrans{Q}}(x)
        - a\FokInvTrans{ \int_0^t \re^{a\argdot^n(s-t)} \FokRemTrans{q}(\argdot;s) \D s }(x).
    \end{equation*}
    By theorem~\ref{thm:Diag}, the latter term evaluates to $0$.
\end{proof}

\begin{eg} \label{eg:Separated3Ord.part7}
    We continue the analysis of example~\ref{eg:Separated3Ord.part6}.
    
    Applying theorem~\ref{thm:FokasMethodSolRep}, the problem studied in example~\ref{eg:Separated3Ord.part1} has solution
    \[
        q(x,t) = \FokInvTrans{\re^{-at\argdot^n}\FokForTrans{Q}}(x),
    \]
    in which the forward transform is defined by
    \[
        \FokForTrans{\phi}(\la) = \begin{cases} \FokForTransP{\phi}(\la) & \mbox{if } \la\in\Gamma^+\cup\Gamma_\cuts, \\ \FokForTransM{\phi}(\la) & \mbox{if } \la\in\Gamma^-, \end{cases}
    \]
    where $\FokForTransPMNoArg$ are given by equations~\eqref{eqn:Separated3Ord.FokForTransPM},
    the inverse transform is defined by
    \[
        \FokInvTrans{F}(x) = \CPV_{\Gamma^+\cup\Gamma^-\cup\Gamma_\cuts} \re^{\ri\nu(\la) x}F(\la)\D\la,
    \]
    and the contours are given in equations~\eqref{eqn:Separated3Ord.FokInvTransContours}.
    
    This concludes the analysis of this extended example.
\end{eg}

\subsection{Simplified version} \label{ssec:Method.Simplified}

We have claimed throughout that the method described in this work is ``the Fokas transform method'' in the sense that the transform pair one might have derived through the original or simplified versions of the Fokas transform method is the same as the transform pair defined in~\S\ref{ssec:TransformPair.BC}.
But the method shown in the proof of theorem~\ref{thm:FokasMethodSolRep} is very different from the original or simplified versions of the Fokas transform method, so this claim has yet to be justified.
Restricting to the case $\omega$ monomial, it can be checked that the transforms match, by comparison with the transform pair defined in~\cite{Smi2012a}.
This is a long but not difficult calculation.
For separated boundary conditions with arbitrary $\omega$, one may compare with~\cite{FP2001a,FP2005a}.
But there has not yet been a full general implementation of the simplified or original version of the Fokas transform method, so it is not possible to compare in general with existing work.
Nevertheless, we claim:

\begin{prop} \label{prop:SimplifiedFokasMethod}
    The Fokas transform pair defined in~\S\ref{ssec:TransformPair.BC} and the solution representation of theorem~\ref{thm:FokasMethodSolRep} are the same as the corresponding transform pair and solution representation derived via the simplified version of the Fokas transform method, up to an application of lemma~\ref{lem:FokasTransformCancellation} and finite contour deformations.
\end{prop}

\begin{lem} \label{lem:SimplifiedFokasMethod.SplitBoundaryTerms}
    For $\phi,\psi\in\Phi$ and $x\in[0,1]$, let $\beta(\phi,\psi)(x)$ be the sesquilinear form such that
    \[
        \left\langle\mathcal{L}\phi,\psi\right\rangle = \left\langle\phi,\mathcal{L}^\star\psi\right\rangle + \beta(\phi,\psi)(1) - \beta(\phi,\psi)(0).
    \]
    Then, for all $\phi\in\Phi_{\BVec{B}}$,
    \[
        \beta(\phi,\psi)(1) = \BVec{B}_\complementary \phi \odot \Msups{\BVec{B}}{}{}{\star}{-} \psi,
        \qquad\qquad
        -\beta(\phi,\psi)(0) = \BVec{B}_\complementary \phi \odot \Msups{\BVec{B}}{}{}{\star}{+} \psi,
    \]
    in which $\Msups{\BVec{B}}{}{}{\star}{\pm}$ are the vectors of left and right adjoint boundary forms $\left( \Msups{B}{1}{}{\star}{\pm} , \Msups{B}{2}{}{\star}{\pm} , \ldots , \Msups{B}{n}{}{\star}{\pm} \right)$.
\end{lem}

\begin{proof}
    By the boundary form formula~\cite[theorem~11.2.1]{CL1955a},
    \begin{equation} \label{eqn:SimplifiedFokasMethod.SplitBoundaryTerms.2}
        \beta(\phi,\psi)(1) - \beta(\phi,\psi)(0) = \BVec{B} \phi \odot \BVec{B}_\complementary^\star \psi + \BVec{B}_\complementary \phi \odot \BVec{B}^\star \psi,
    \end{equation}
    and $\phi\in\Phi_{\BVec{B}}$ implies the the first term on the right evaluates to $0$.
    Therefore, for all $\psi\in\Phi$,
    \begin{equation} \label{eqn:SimplifiedFokasMethod.SplitBoundaryTerms.1}
        \beta(\phi,\psi)(1) - \beta(\phi,\psi)(0) = \BVec{B}_\complementary \phi \odot \BVec{B}^\star \psi.
    \end{equation}
    Let $\chi\in \mathrm{C}^\infty[0,1]$ be a smooth cutoff function such that $\chi(x)=0$ for $x<1/3$ and $\chi(x)=1$ for $x>2/3$.
    If $\psi\in\Phi$, then $\psi\chi\in\Phi$ too.
    Therefore, using equation~\eqref{eqn:SimplifiedFokasMethod.SplitBoundaryTerms.1} for the second equality,
    \[
        \beta(\phi,\psi)(1) = \beta(\phi,\psi\chi)(1) - \beta(\phi,\psi\chi)(0) = \BVec{B}_\complementary \phi \odot \BVec{B}^\star [\psi\chi] = \BVec{B}_\complementary \phi \odot \Msups{\BVec{B}}{}{}{\star}{-} \psi.
    \]
    Similarly,
    \[
        -\beta(\phi,\psi)(0) = \beta(\phi,\psi(1-\chi))(1) - \beta(\phi,\psi(1-\chi))(0) = \BVec{B}_\complementary \phi \odot \BVec{B}^\star [\psi(1-\chi)] = \BVec{B}_\complementary \phi \odot \Msups{\BVec{B}}{}{}{\star}{+} \psi.
        \qedhere
    \]
\end{proof}

\begin{proof}[Proof of proposition~\ref{prop:SimplifiedFokasMethod}]
    To show that the results of our true transform method are generally the same as the simplified method, we present the general implementation of the simplified version of the Fokas transform method using the notation developed in~\S\ref{sec:TransformPair}--\ref{sec:Diagonalization}.

    Applying the unit interval Fourier transform to~\eqref{eqn:IBVP.PDE}, one obtains, for all $\la\in\CC$,
    \[
        \FourUnitTrans{\left[\partial_t + a L\right]q(\argdot,t)\vphantom{G^G}}(\la) = 0.
    \]
    Assuming sufficient smoothness in time, the time derivative may be exchanged with the integral, and the rest of the Fourier transform may be reexpressed using the sesquilinear inner product:
    \[
        \frac{\D}{\D t} \FourUnitTrans{q}(\la;t) + a \left\langle Lq,\overline{\re^{-\ri\la\argdot}}\right\rangle = 0.
    \]
    Integrating by parts and adopting the boundary form notation of lemma~\ref{lem:SimplifiedFokasMethod.SplitBoundaryTerms},
    \[
        \frac{\D}{\D t} \FourUnitTrans{q}(\la;t) + a \left\langle q,\mathcal{L}^\star\overline{\re^{-\ri\la\argdot}}\right\rangle + a \left[\beta(q(\argdot,t),\overline{\re^{-\ri\la\argdot}})(x)\right]_{x=0}^{x=1} = 0,
    \]
    and a change of variables $\la\mapsto\nu(\la)$ establishes
    \begin{equation} \label{eqn:Methd.Simplified.1}
        \frac{\D}{\D t} \FormalForTrans{q}(\la;t) + a \left\langle q,\mathcal{L}^\star y_1(\argdot;\la) \right\rangle + a \Big[\beta(q(\argdot,t),y_1(\argdot;\la))(x)\Big]_{x=0}^{x=1} = 0,
    \end{equation}
    for all $\la$ outside a disc of radius $\propertiesnuradius$, where $y_1$ is as defined in equation~\eqref{eqn:defn.yj}.
    By proposition~\ref{prop:propertiesnu}, $y_1(\argdot;\la)$ is an eigenfunction of $\mathcal{L}^\star$:
    \[
        \mathcal{L}^\star y_1(\argdot;\la) = \overline{\omega}(-\ri\partial_x)\overline{\re^{-\ri\nu(\la)\argdot}} = \overline{\la^n}y_1(\argdot;\la).
    \]
    Therefore equation~\eqref{eqn:Methd.Simplified.1} simplifies to
    \[
        \left[\frac{\D}{\D t} + a \la^n\right] \FormalForTrans{q}(\la;t) + a \Big[\beta(q(\argdot,t),y_1(\argdot;\la))(x)\Big]_{x=0}^{x=1} = 0.
    \]
    Solving the ordinary differential equation, we obtain the global relation
    \begin{equation} \label{eqn:Method.Simplified.GR}
        \FormalForTrans{q}(\la;t) = \re^{-a\la^nt} \FormalForTrans{q_0}(\la) - a\re^{-a\la^nt}\int_0^t \re^{a\la^ns} \Big[\beta(q(\argdot,t),y_1(\argdot;\la))(x)\Big]_{x=0}^{x=1} \D s.
    \end{equation}

    Applying corollary~\ref{cor:AlternativeFormalTransformValid},
    \[
        q(x,t) = \frac{1}{2\pi} \;\CPV_{\gamma} \re^{\ri\nu(\la)x-a\la^nt} \nu'(\la) \left[\FormalForTrans{q_0}(\la) - a \int_0^t \re^{a\la^ns} \Big[\beta(q(\argdot,t),y_1(\argdot;\la))(x)\Big]_{x=0}^{x=1} \D s \right] \D \la.
    \]
    Under the principal value, this integral can be split into three parts, two along the existing contour $\gamma$ passing above the disc $D(0,\propertiesnuradius)$ and one along a new contour $\gamma^-$ with the opposite orientation and passing below the disc:
    \begin{multline*}
        q(x,t)
        = \frac{1}{2\pi} \; \CPVwithoutintegral \bigg(
              \int_{\gamma  } \re^{\ri\nu(\la)x-a\la^nt}\nu'(\la) \FormalForTrans{q_0}(\la) \D\la \\
            + \int_{\gamma  } \re^{\ri\nu(\la)x-a\la^nt}\nu'(\la) a \int_0^t \re^{a\la^ns} \beta(q(\argdot,t),y_1(\argdot;\la))(0) \D s \D\la \\
            + \int_{\gamma^-} \re^{\ri\nu(\la)x-a\la^nt}\nu'(\la) a \int_0^t \re^{a\la^ns} \beta(q(\argdot,t),y_1(\argdot;\la))(1) \D s \D\la
        \bigg).
    \end{multline*}
    By Jordan's lemma,
    \begin{multline} \label{eqn:Method.Simplified.EhrenpreisForm}
        q(x,t)
        = \frac{1}{2\pi} \; \CPVwithoutintegral \bigg(
              \int_{\gamma        } \re^{\ri\nu(\la)x-a\la^nt}\nu'(\la) \FormalForTrans{q_0}(\la) \D\la \\
            + \int_{\hat\Gamma_a^+} \re^{\ri\nu(\la)x-a\la^nt}\nu'(\la) a \int_0^t \re^{a\la^ns} \beta(q(\argdot,t),y_1(\argdot;\la))(0) \D s \D\la \\
            + \int_{\hat\Gamma_a^-} \re^{\ri\nu(\la)x-a\la^nt}\nu'(\la) a \int_0^t \re^{a\la^ns} \beta(q(\argdot,t),y_1(\argdot;\la))(1) \D s \D\la
        \bigg),
    \end{multline}
    where the contours $\hat\Gamma_a^\pm$ are defined in equation~\eqref{eqn:defn.Gammaapmhat}.
    Equation~\eqref{eqn:Method.Simplified.EhrenpreisForm} is a contour integral representation of the solution of IBVP~\eqref{eqn:IBVP} which is known in the Fokas transform literature as the Ehrenpreis form.

    By lemma~\ref{lem:SimplifiedFokasMethod.SplitBoundaryTerms},
    \begin{align*}
        -\beta(q(\argdot,t),y_1(\argdot;\la))(0)
        &= \BVec{B}_\complementary q(\argdot,s) \odot \overline{M_1^+(\la)}, \\
        \beta(q(\argdot,t),y_1(\argdot;\la))(1)
        &= \BVec{B}_\complementary q(\argdot,s) \odot \overline{M_1^-(\la)},
    \end{align*}
    in which $M_1^\pm$ are the first rows of left and right characteristic matrices $M^\pm$ defined in equation~\eqref{eqn:defn.Mpmjk}.
    Substituting into the Ehrenpreis form, we obtain
    \begin{multline} \label{eqn:Method.Simplified.EhrenpreisForm2}
        q(x,t)
        = \frac{1}{2\pi} \; \CPVwithoutintegral \bigg(
              \int_{\gamma        } \re^{\ri\nu(\la)x-a\la^nt}\nu'(\la) \FormalForTrans{q_0}(\la) \D\la \\
            - \int_{\hat\Gamma_a^+} \re^{\ri\nu(\la)x-a\la^nt}\nu'(\la) \left[a \int_0^t \re^{a\la^ns} \BVec{B}_\complementary q(\argdot,s) \D s\right] \odot \overline{M_1^+(\la) \D\la} \\
            + \int_{\hat\Gamma_a^-} \re^{\ri\nu(\la)x-a\la^nt}\nu'(\la) \left[a \int_0^t \re^{a\la^ns} \BVec{B}_\complementary q(\argdot,s) \D s\right] \odot \overline{M_1^-(\la)} \D\la
        \bigg),
    \end{multline}
    an equation which depends on the initial datum and the $n$ unknown spectral functions that make up the vector
    \begin{equation} \label{eqn:Method.Simplified.UnknownSpectral}
        a \int_0^t \re^{a\la^ns} \BVec{B}_\complementary q(\argdot,s) \D s.
    \end{equation}
    Therefore, to obtain an effective solution representation, it remains to find suitable expressions for the unknown spectral functions.

    Substituting equation~\eqref{eqn:SimplifiedFokasMethod.SplitBoundaryTerms.2} into global relation~\eqref{eqn:Method.Simplified.GR} and rearranging, we find
    \[
        \left[a \int_0^t \re^{a\la^ns} \BVec{B}_\complementary q(\argdot,s) \D s\right] \odot \BVec{B}^\star y_1(\argdot;\la) = \FormalForTrans{q_0}(\la) - \re^{a\la^nt} \FormalForTrans{q}(\la;t).
    \]
    For all $j\in\{1,2,\ldots,n\}$, $\la^n$ is invariant under the map $\la\mapsto\alpha^j\la$.
    Therefore, for all $j$,
    \begin{equation} \label{eqn:Method.Simplified.System}
        \left[a \int_0^t \re^{a\la^ns} \BVec{B}_\complementary q(\argdot,s) \D s\right] \odot \BVec{B}^\star y_j(\argdot;\la) = \FormalForTrans{q_0}(\alpha^{j-1}\la) - \re^{a\la^nt} \FormalForTrans{q}(\alpha^{j-1}\la;t).
    \end{equation}
    This is a system of $n$ linear equations in the $n$ unknown spectral functions that appear in the vector~\eqref{eqn:Method.Simplified.UnknownSpectral}.
    By equations~\eqref{eqn:defn.Mjk}--\eqref{eqn:defn.Delta}, the determinant of system~\eqref{eqn:Method.Simplified.System} is $\Delta(\la)$.
    Moreover, solving the system and evaluating the specific linear combinations of solutions required for substitution into equation~\eqref{eqn:Method.Simplified.EhrenpreisForm2}, we find
    \begin{align*}
        \frac{\nu'(\la)}{2\pi} \left[a \int_0^t \re^{a\la^ns} \BVec{B}_\complementary q(\argdot,s) \D s\right] \odot \overline{M_1^+(\la)} &= \FokForTransP{q_0}(\la) - \re^{a\la^nt} \FokForTransP{q}(\la;t), \\
        -\frac{\nu'(\la)}{2\pi} \left[a \int_0^t \re^{a\la^ns} \BVec{B}_\complementary q(\argdot,s) \D s\right] \odot \overline{M_1^-(\la)} &= \FokForTransM{q_0}(\la) - \re^{a\la^nt} \FokForTransM{q}(\la;t).
    \end{align*}
    By assumption~\ref{asm:MainAssumption}, once substituted into equation~\eqref{eqn:Method.Simplified.EhrenpreisForm2}, the latter terms on the right of the above equations yield zero contribution.
    Hence
    \begin{multline} \label{eqn:Method.Simplified.SolRep}
        q(x,t)
        = \CPVwithoutintegral \bigg(
              \int_{\gamma        } \re^{\ri\nu(\la)x-a\la^nt}\frac{\nu'(\la)}{2\pi} \FormalForTrans{q_0}(\la) \D\la
            - \int_{\hat\Gamma_a^+} \re^{\ri\nu(\la)x-a\la^nt} \FokForTransP{q_0}(\la) \D\la \\
            - \int_{\hat\Gamma_a^-} \re^{\ri\nu(\la)x-a\la^nt} \FokForTransM{q_0}(\la) \D\la
        \bigg).
    \end{multline}

    Equation~\eqref{eqn:Method.Simplified.SolRep} is an effective solution representation because no unknowns appear on the right.
    It is (up to change of variables $k=\nu(\la)$) the usual integral representation obtained via the original or simplified Fokas transform method; the infinite components of contours $\hat\Gamma_a^+$ are the boundaries the regions
    \[
        \{\la\in\CC^\pm: \Re(a\la^n)<0\},
    \]
    which are usually referred to as $D^\pm$ in the Fokas transform literature.
    But it is not exactly the solution representation obtained in theorem~\ref{thm:FokasMethodSolRep}.
    Indeed, in this work we have followed the approach of~\cite{FS2016a,Smi2012a,KPPS2018a} in rewriting the three integrals of equation~\eqref{eqn:Method.Simplified.SolRep} into two integrals that are, ignoring some small perturbations, about the boundaries of regions
    \[
        \{\la\in\CC^\pm: \Re(a\la^n)>0\}
    \]
    and usually denoted $E^\pm$ in the Fokas transform literature.
    Precisely, using lemma~\ref{lem:FokasTransformCancellation}, the first integral on the right of equation~\eqref{eqn:Method.Simplified.SolRep} can be rewritten as
    \[
        \int_{\gamma  } \re^{\ri\nu(\la)x-a\la^nt} \FokForTransP{q}(\la;t) \D\la
        +
        \int_{\gamma^-} \re^{\ri\nu(\la)x-a\la^nt} \FokForTransM{q}(\la;t) \D\la,
    \]
    with $\gamma^-$ as above.
    The contours have been defined such that, after deformations over finite regions, this yields the solution representation of theorem~\ref{thm:FokasMethodSolRep}.
\end{proof}

\section{Inhomogeneous problems} \label{sec:Inhomogeneous}

For any $\BVec{h}\in\CC^n$, let $\M{\Phi}{\BVec{B}}{\BVec{h}} = \{\phi\in\Phi: \BVec{B}\phi=\BVec{h}\}$, a set of functions which is not itself a linear space, but whose elements differ from one another by elements of the space $\Phi_{\BVec{B}}$.
Consider the inhomogeneous IBVP
\begin{subequations} \label{eqn:Inhomogeneous.IBVP}
\begin{align}
    \label{eqn:Inhomogeneous.IBVP.PDE} \tag{\theparentequation.PDE}
    \partial_t q(x,t) + a \mathcal{L}q(\argdot,t) &= \mathcal{Q}(x,t) & (x,t) &\in (0,1) \times (0,T), \\
    \label{eqn:Inhomogeneous.IBVP.IC} \tag{\theparentequation.IC}
    q(x,0) &= Q(x) & x &\in[0,1], \\
    \label{eqn:Inhomogeneous.IBVP.BC} \tag{\theparentequation.BC}
    \BVec{B}q(\argdot,t) &= \BVec{h}(t) & t &\in[0,T],
\end{align}
\end{subequations}
in which $\mathcal{Q}$ and $\BVec{h}=(h_1,h_2,\ldots,h_n)$ are appropriately smooth functions, $Q\in\M{\Phi}{\BVec{B}}{\BVec{h}(0)}$, and coefficient $a$ obeys the same conditions as those for IBVP~\eqref{eqn:IBVP}.
In this section, we describe the extensions of the above results necessary to study problem~\eqref{eqn:Inhomogeneous.IBVP} in which one or both of $\mathcal{Q}$ and $\BVec{h}$ are nonzero.

The argument in the proof of theorem~\ref{thm:FokasMethodSolRep} relies on theorems~\ref{thm:FokasTransformValid} and~\ref{thm:Diag}.
But the Fokas transform validity theorem~\ref{thm:FokasTransformValid} does not require $\phi\in\Phi_{\BVec{B}}$, only $\phi\in\mathrm{C}^1[0,1]$.
Therefore, no adaptation of the validity theorem is required for inhomogeneous problems.
If it happens that $\BVec{h}=\BVec{0}$, so that the only inhomogeneities are in~\eqref{eqn:Inhomogeneous.IBVP.PDE} and~\eqref{eqn:Inhomogeneous.IBVP.IC}, then~\eqref{eqn:Inhomogeneous.IBVP.PDE} may be rewritten with the operator $L$ in place of $\mathcal{L}$, and diagonalization theorem~\ref{thm:Diag} may also be applied unchanged.
However, if $\BVec{h}\neq\BVec{0}$, then we must adapt both the diagonalization theorem and its main lemma~\ref{lem:DiagMainLemma}.
Their generalizations are presented below, followed by the appropriate restatement and proof of the Fokas transform method theorem~\ref{thm:FokasMethodSolRep} for inhomogeneous IBVP~\eqref{eqn:Inhomogeneous.IBVP}.

\begin{lem} \label{lem:Inhomogeneous.DiagMainLemma}
    Let $\BVec{B}_\complementary^\star$ be the vector of complementary adjoint boundary forms associated with $\BVec{B}$ and $\BVec{B}^\star$, as defined in~\cite[theorem~11.2.1]{CL1955a}, and $\psi^\pm$ be the complex conjugates of eigenfunctions of $\mathcal{L}^\star$ defined in equation~\eqref{eqn:DiagMainLemma.Proof.psipm}.
    If $\phi\in\M{\Phi}{\BVec{B}}{\BVec{h}}$, then there exist functions $P_\phi^\pm$, analytic on the domain of $\nu$, with
    $P_\phi^\pm(\la) = \bigoh{\abs{\la}^{n-1}}$,
    uniformly in $\arg(\la)$, as $\la\to\infty$
    and
    \begin{subequations} \label{eqn:Inhomogeneous.DiagMainLemma.ApproxEigs}
    \begin{align}
        \FokForTransP{\mathcal{L}\phi} (\la) &= \la^n \FokForTransP{\phi}(\la) + P_\phi^+(\la) + \BVec{h} \odot \BVec{B}_\complementary^\star \psi^+(\argdot;\la), \\
        \FokForTransM{\mathcal{L}\phi} (\la) &= \la^n \FokForTransM{\phi}(\la) + P_\phi^-(\la)\re^{-\ri\nu(\la)} + \BVec{h} \odot \BVec{B}_\complementary^\star \psi^-(\argdot;\la).
    \end{align}
    \end{subequations}
    Moreover, the maps $\phi \mapsto P_\phi^\pm$ are linear.
    Indeed, such functions are given by equations~\eqref{eqn:DiagMainLemma.PpmFormulae}.
\end{lem}

\begin{proof}
    The proof of lemma~\ref{lem:Inhomogeneous.DiagMainLemma} follows almost exactly the proof of lemma~\ref{lem:DiagMainLemma}.
    The only difference is that the simplification $\BVec{B}\phi=\BVec{0}$ that is used to derive equation~\eqref{eqn:DiagMainLemma.Proof.PpmExplicitFormulae} is replaced with $\BVec{B}\phi=\BVec{h}$.
    This results in the extra terms appearing in equations~\eqref{eqn:Inhomogeneous.DiagMainLemma.ApproxEigs} but not in equations~\eqref{eqn:DiagMainLemma.ApproxEigs}.
\end{proof}

\begin{thm} \label{thm:Inhomogeneous.Diag}
    There exists an \emph{inhomogeneous boundary term} $\mathcal{H}[\BVec{h}]$ defined by
    \[
        \mathcal{H}[\BVec{h}](\la) = \begin{cases}
            \BVec{h} \odot \BVec{B}_\complementary^\star \psi^+(\argdot;\la) & \mbox \la\in\Gamma^+\cup\Gamma_\cuts, \\
            \BVec{h} \odot \BVec{B}_\complementary^\star \psi^-(\argdot;\la) & \mbox \la\in\Gamma^-,
        \end{cases}
    \]
    for $\psi^\pm$ as in the proof of lemmata~\ref{lem:DiagMainLemma} and~\ref{lem:Inhomogeneous.DiagMainLemma},
    and a \emph{remainder transform} $\FokRemTransNoArg$, such that, for all $\phi\in\M{\Phi}{\BVec{B}}{\BVec{h}}$,
    \begin{subequations} \label{eqn:Inhomogeneous.DiagThm}
    \begin{equation} \label{eqn:Inhomogeneous.DiagThm.1}
        \FokForTrans{\mathcal{L}\phi}(\la) = \la^n \FokForTrans{\phi}(\la) + \FokRemTrans{\phi}(\la) + \mathcal{H}[\BVec{h}](\la)
    \end{equation}
    and,
    if
    $q:[0,1]\times[0,T]\to\CC$ is such that, for all $t\in[0,T]$, $q(\argdot,t)\in\Phi$ and, for all $j\in\{0,1,\ldots,n-1\}$, uniformly for all $x\in[0,1]$, $\partial_x^jq(x,\argdot)$ is a function of bounded variation,
    then,
    for all $t\in[0,T]$ and all $x\in(0,1)$,
    \begin{equation} \label{eqn:Inhomogeneous.DiagThm.2}
        \FokInvTrans{\int_0^t \re^{a\argdot^n(s-t)} \FokRemTrans{q}(\argdot;s) \D s}(x) = 0.
    \end{equation}
    \end{subequations}
\end{thm}

\begin{proof}
    If the remainder transform is defined by equation~\eqref{eqn:Diag.RemTransDefn}, then lemma~\ref{lem:Inhomogeneous.DiagMainLemma} implies equation~\eqref{eqn:Inhomogeneous.DiagThm.1}.
    The proof of equation~\eqref{eqn:Inhomogeneous.DiagThm.2} is identical to the proof of equation~\eqref{eqn:DiagThm.2}.
\end{proof}

\begin{thm} \label{thm:Inhomogeneous.FokasMethodSolRep}
    Suppose the criteria of theorem~\ref{thm:FokasTransformValid} hold, and there exists a solution $q(x,t)$ of problem~\eqref{eqn:Inhomogeneous.IBVP}, absolutely continuous in $t$ and satisfying the criteria of theorem~\ref{thm:Inhomogeneous.Diag}.
    Then, provided the latter two integrals converge,
    \begin{multline*}
        q(x,t)
        = \FokInvTrans{\re^{-at\argdot^n}\FokForTrans{Q}}(x)
        + \FokInvTrans{ \int_0^t \re^{a\argdot^n(s-t)} \FokForTrans{\mathcal{Q}}(\argdot;s) \D s }(x) \\
        - a\FokInvTrans{ \int_0^t \re^{a\argdot^n(s-t)} \mathcal{H}[\BVec{h}](\argdot;s) \D s }(x).
    \end{multline*}
\end{thm}

\begin{proof}
    Applying the forward Fokas transform to~\eqref{eqn:Inhomogeneous.IBVP.PDE}, and exploiting linearity of the forward Fokas transform, almost everywhere in $t$,
    \begin{align*}
        \FokForTrans{\mathcal{Q}}(\la;t) &= \frac{\D}{\D t} \FokForTrans{q}(\la;t) + a \FokForTrans{\mathcal{L}q}(\la;t).
    \intertext{By theorem~\ref{thm:Inhomogeneous.Diag},}
        \FokForTrans{\mathcal{Q}}(\la;t) &= \frac{\D}{\D t} \FokForTrans{q}(\la;t) + a \la^n \FokForTrans{q}(\la;t) + a\FokRemTrans{q}(\la;t) + a\mathcal{H}[\BVec{h}](\la;t).
    \end{align*}
    Using~\eqref{eqn:Inhomogeneous.IBVP.IC} to solve the initial value problem for this ODE, we find
    \begin{multline} \label{eqn:Inhomogeneous.FokasMethodSolRep.Proof.SolvedODE}
        \FokForTrans{q}(\la;t)
        = \re^{-a\la^nt}\FokForTrans{Q}(\la)
        + \re^{-a\la^nt} \int_0^t \re^{a\la^ns} \FokForTrans{\mathcal{Q}}(\la;s) \D s \\
        - a \re^{-a\la^nt} \int_0^t \re^{a\la^ns} \mathcal{H}[\BVec{h}](\la;s) \D s
        - a \re^{-a\la^nt} \int_0^t \re^{a\la^ns} \FokRemTrans{q}(\la;s) \D s.
    \end{multline}
    By theorem~\ref{thm:Inhomogeneous.Diag}, the inverse Fokas transform of the last term of equation~\eqref{eqn:Inhomogeneous.FokasMethodSolRep.Proof.SolvedODE} evaluates to $0$.
    Therefore, applying the inverse Fokas transform to equation~\eqref{eqn:Inhomogeneous.FokasMethodSolRep.Proof.SolvedODE} yields the claimed solution representation.
\end{proof}

\section{Multipoint and interface problems} \label{sec:Interface}

\subsection{Interface differential operators} \label{ssec:Interface.Operator}

Let $m,n\in\NN$ and, for each $r\in\{1,\ldots,m\}$, choose $\omega_r$ a monic polynomial of degree $n$,
\[
    \omega_r(z) = z^{n} + \sum_{j=0}^{n-2} \M{c}{r}{j}z^j.
\]
We study the differential operators formally defined on functions $\phi_r$ by
\[
    \mathcal{L}_r\phi_r := \omega_r(-\ri\partial_x)\phi_r = (-\ri)^{n}\phi_r^{(n)} + \sum_{j=0}^{n-2} (-\ri)^{j}\M{c}{r}{j}\phi_r^{(j)}, \qquad r = 1,\ldots, m.
\]
We also study the vector formal differential operator $\mathcal{L}$ formed from such operators; for functions $\phi=(\phi_1,\phi_2,\ldots,\phi_m)$,
\[
    \mathcal{L}\phi = \left(\mathcal{L}_1\phi_1,\mathcal{L}_2\phi_2,\ldots,\mathcal{L}_m\phi_m\right),
\]
or, equivalently,
\[
    \mathcal{L} := \left(\mathcal{L}_1,\mathcal{L}_2,\ldots,\mathcal{L}_m\right)\circ,
\]
where $\circ$ represents the entrywise action of operators.

We define $\Phi = \prod^{m}_{r=1} \AC^{n-1}[0, 1]$, a product of function spaces $\AC^{n-1}[0, 1]$.
The inner product on this space is defined as the sum of the ordinary unweighted sesquilinear $\mathrm{L}^2$ inner products on the constituent spaces.

Suppose that matrices of complex \emph{boundary coefficients} $b^r,\beta^r\in\CC^{\nall \times n}$, $r \in \{1, \ldots, m\}$ are chosen such that the concatenated matrix $(b^1:\beta^1:\ldots:b^m:\beta^m)\in\CC^{\nall\times2\nall}$ has full rank.
Define \emph{boundary forms} $B_k:\Phi\to\CC$ by
\[
    B_k(\phi) = \sum_{r=1}^m \sum_{j=1}^{n} \left( \Msup{b}{k}{j}{r}\phi_r^{(j-1)}(0) +  \Msup{\beta}{k}{j}{r}\phi_r^{(j-1)}(1) \right), \qquad k\in\left\{1,2,\ldots,\nall\right\},
\]
and let $\BVec{B} = (B_1, B_2, \ldots, B_{\nall})$ be the vector of boundary forms $B_k$.
On the space
\[
    \Phi_{\BVec{B}}=\{\phi\in\Phi: \BVec{B}\phi=\BVec{0}\},
\]
we define $L:\Phi_\BVec{B}\to \prod_{r=1}^m \mathrm{L}^1[0,1]$ by $L\phi=\mathcal{L}\phi$.

\subsection{Initial interface value problems} \label{ssec:Interface.IIVP}

The operator $L$ represents the spatial part of the initial interface value problem (IIVP)
\begin{subequations} \label{eqn:Interface.IBVP}
\begin{align}
    \label{eqn:Interface.IBVP.PDE} \tag{\theparentequation.PDE}
    \partial_t q(x,t) + \BVec{a}\circ L q(x,t) &= \BVec{0} & (x,t) &\in (0,1) \times (0,T), \\
    \label{eqn:Interface.IBVP.IC} \tag{\theparentequation.IC}
    q(x,0) &= Q(x) & x &\in [0,1], \\
    \label{eqn:InterfaceIBVP.BC} \tag{\theparentequation.BC}
    \BVec{B}q(\argdot,t) &= \BVec{0} & t &\in[0,T],
\end{align}
\end{subequations}
in which we assume $Q = (Q_1, \ldots, Q_m)\in\Phi_\BVec{B}$ and $\BVec{a}$ is a list of $m$ nonzero complex numbers so that each entry $a_r$ obeys the same criteria with relation to $\omega_r$ as did $a$ for $\omega$ in~\S\ref{ssec:Introduction.IBVP}, and $\circ$ represents the entrywise product.

IIVP~\eqref{eqn:Interface.IBVP} is a notational shorthand for problems in which $q = (q_1, \ldots, q_m)$ is a list of solutions of partial differential equations
\begin{align*}
    \partial_t q_r(x,t) + a_r \mathcal{L}_r q_r(x,t) &= 0 & (x,t) &\in (0,1) \times (0,T), & r &\in \{1,\ldots,m\}, \\
\intertext{subject to initial conditions}
    \partial_t q_r(x,0) &= Q_r(x) & x &\in [0,1], & r &\in \{1,\ldots,m\},
\end{align*}
and simultaneous boundary conditions~\eqref{eqn:InterfaceIBVP.BC}.

IIVP~\eqref{eqn:Interface.IBVP} can represent a broad range of multipoint and interface problems.
This claim is illustrated with the following few examples.
In the space $\Phi$ defined in~\S\ref{ssec:Interface.Operator}, each of the $m$ spatial intervals is assumed to be of the same length $1$, and the $x$ values all extend from $0$ at the left.
But a simple change of variables can reduce a wide range of similar problems to those we explicitly study.

\subsubsection{Multipoint example with irregularly spaced points}
Consider the initial multipoint value problem (IMVP) for the diffusion equation
\begin{align*}
    \partial_t u(y,t) - \partial_{yy}u(\argdot,t) &= 0 & (y,t) &\in (0,3) \times (0,T), \\
    u(y,0) &= U(y) & y &\in[0,3], \\
    \partial_yu(0,t) &= \partial_yu(1,t) & t &\in[0,T], \\
    u(3,t) &= 0 & t &\in[0,T],
\end{align*}
where $u(\argdot,t)$ is continuously differentiable on $[0,3]$ and its restrictions to $[0,1]$ and $[1,3]$ are both absolutely continuous along with their first derivatives.
Boundary conditions such as these arise, for example, in diffusion problems~\cite{Can1963a,DM1963a} where the concentration of the diffusive substance is specified as an average over an interval $x\in[0,1]$ of finite width, instead of at an infinitesimal boundary $x=0$.

The change of variables
\begin{align*}
    x &= y, & q_1(x,t) &= u(y,t), & &\mbox{for }y\in[0,1], \\
    x &= (y-1)/2, & q_2(x,t) &= u(y,t), & &\mbox{for }y\in[1,3]
\end{align*}
yields PDE
\[
    \partial_t q(x,t) - \BVec{a}\circ \partial_{xx} q(x,t) = \BVec{0} \qquad\qquad (x,t) \in (0,1) \times (0,T),
\]
for $\BVec{a}=(1,1/4)$, initial condition
\[
    q(x,0) = \begin{bmatrix}U(x) \\ U(2x+1)\end{bmatrix}
\]
and boundary conditions~\eqref{eqn:InterfaceIBVP.BC} with boundary coefficient matrices
\[
    \begin{bmatrix} b^1 : \beta^1 : b^2 : \beta^2 \end{bmatrix}
    =
    \begin{bmatrix}
        0 & 1 & 0 & -1 &  0 &  0 & 0 & 0 \\
        0 & 0 & 0 &  0 &  0 &  0 & 1 & 0 \\
        0 & 0 & 1 &  0 & -1 &  0 & 0 & 0 \\
        0 & 0 & 0 &  2 &  0 & -1 & 0 & 0
    \end{bmatrix},
\]
in which the latter two rows correspond to continuous differentiability of $u(\argdot,t)$ over the interface.

In a similar way, any IMVP of the form studied in~\cite{PS2018a} can be reduced to IIVP~\eqref{eqn:Interface.IBVP}.
But the formulation of the above IMVP and the class of IMVP studied in~\cite{PS2018a} share the restriction that $u$ is assumed to have $n-1$ continuous partial derivatives over each interface, a restriction which IIVP formulation~\eqref{eqn:Interface.IBVP} removes.
Considering the results of~\cite{Loc1973a,Neu1966a,Zet1966a} that selfadjoint multipoint operators always fail this condition, we feel it is more natural to study the more general class of IMVP that can be reduced to IIVP~\eqref{eqn:Interface.IBVP}.

\subsubsection{Linear Schr\"{o}dinger equation with piecewise constant potential} \label{sssec:LSPot}
The Dirichlet problem for the time dependent, linear Schr\"{o}dinger equation with piecewise constant potential $c(x) = c_r$ for $\frac{r-1}{m}<x<\frac{r}{m}$ may be expressed as
\begin{align*}
    i \partial_t u (x,t) &= \partial_{xx} u(x,t) + c(x) u(x,t) & (x,t) &\in (0,1) \times (0,T), \\
    u(x,0) &= U(x) & x &\in [0,1], \\
    u(0, t)  &= 0 & t &\in[0,T], \\
    u(1, t)  &= 0 & t &\in[0,T],
\end{align*}
for $u(\argdot,t)$ continuously differentiable on $[0,1]$ such that its restriction to each interval $[\frac{r-1}{m},\frac{r}{m}]$ has absolutely continuous derivative.

The change of variables
\[
    x = my+1-r, \qquad q_r(x,t) = u(y,t), \qquad\qquad \mbox{for } \frac{r-1}{m}<y<\frac{r}{m}
\]
reduces this problem to IIVP~\eqref{eqn:Interface.IBVP}.
Via such piecewise approximation of the potential, the linear Schr\"{o}dinger equation for arbitrary potential, and indeed other variable coefficient differential equations, may be solved to arbitrary accuracy using the methods presented below.

\subsubsection{Interface problems on networks} \label{sssec:InterfaceNetwork}
The above example may be seen as an interface problem in which the domain is a collection of intervals arranged as the edges of a linear graph.
This may be generalized to interface problems in which the domain follows a more complicated graph.
For exampe,~\cite{SS2015a} specifies and solves the perfect thermal contact problem for the heat equation on network domains, including a star graph and a graph with arbitrarily many parallel edges.
Mapping each component of any network of finite intervals to a copy of $[0,1]$, and applying the corresponding mappings to the boundary and interface coefficients, one can express any perfect thermal contact problem for the heat equation on a network domain as IIVP~\eqref{eqn:Interface.IBVP}, up to change of variables.

Perfect thermal contact is a more general type of interface condition than the $\mathrm{C}^{n-1}$ interface conditions considered in~\S\ref{sssec:LSPot}.
But it is itself a special case of more general interface conditions, such as imperfect thermal contact (a class of problems which has also been studied previously using the simplified version of the Fokas transform method, albeit only on linear domains~\cite{She2017a}), or the conditions required to model dispersive phenomena at network interfaces.
IIVP~\eqref{eqn:Interface.IBVP} can be used to encode all of these, and any other IIVP with local interface conditions.

\subsection{Formal transform pair} \label{ssec:Interface.TransformPair.Formal}

For each $r\in\{1,2,\ldots,m\}$, we define $\nu_r$ to be a function such that
\[
    \omega_r(\nu_r(\la))=\la^{n}, \qquad \lim_{n\to\infty}\nu_r(\la)/\la=1, \qquad \mbox{and} \qquad\lim_{\la\to\infty}\nu'_r(\la)\to1, \mbox{ uniformly in }\arg(\la).
\]
By proposition~\ref{prop:propertiesnu}, such functions exist, and there exists some $\propertiesnuradius>0$ such that, outside $D(0,\propertiesnuradius)$, all $\nu_r$ are biholomorphic.

We also define the formal forward transform $\rFormalForTransNoArg{r}$ and inverse transform $\rFormalInvTransNoArg{r}$ by
\begin{subequations} \label{eqn:rTransformPair.Formal.Defn}
\begin{align}
    \label{eqn:Interface.TransformPair.Formal.Defn.Forward}
    \rFormalForTrans{r}{\phi_r}(\la) &:=
    \int_0^1 \phi_r(x) \re^{-\ri\nu_r(\la)x} \D x, & \abs{\la} > \propertiesnuradius, \\
    \label{eqn:Interface.TransformPair.Formal.Defn.Inverse}
    \rFormalInvTrans{r}{F_r}(x) &:= \frac{1}{2\pi} \;\CPV_{\gamma} \re^{\ri\nu_r(\la)x} \nu'_r(\la)F_r(\la) \D\la,
\end{align}
\end{subequations}
in which $\gamma$ is a contour defined as in~\S\ref{ssec:TransformPair.Formal}.
The validity of these transform pairs is corollary~\ref{cor:AlternativeFormalTransformValid}.
Just as in the single interval case, as presented in theorem~\ref{thm:FormalTransformValid}, there is an equivalent inverse transform taking a principal value in the $\nu_r(\la)$ domain rather than the $\la$ plane.

\subsection{Fokas transform pair respecting the boundary conditions} \label{ssec:Interface.TransformPair.BC}

We denote by $L^\star$, resectively $\mathcal{L}^\star$, the adjoint of $L$, respectively $\mathcal{L}$, as constructed in Appendix~\ref{sec:AdjointOperator}; entrywise, $\mathcal{L}^\star_r \psi_r = \overline\omega(-\ri\partial_x)\psi_r$ for $\overline\omega$ the Schwarz conjugate of $\omega$.
If $\BVec{B}^\star = (B^\star_1, \ldots, B^\star_{\nall})$ is a vector of boundary forms adjoint to $\BVec{B}$, with matrices $\Msups{b}{}{}{r}{\star}, \Msups{\beta}{}{}{r}{\star}$ of boundary coefficients
\begin{equation}
    B_k^\star(\psi) = \sum_{r=1}^m\sum_{j=1}^{n} \left( \Msups{b}{k}{j}{r}{\star} \psi^{(j-1)}_r(0) + \Msups{\beta}{k}{j}{r}{\star} \psi^{(j-1)}_r(1) \right), \qquad k\in\{1,2,\ldots,\nall\},
\end{equation}
then $L^\star:\Phi_{\BVec{B}^\star}\to\prod_{r=1}^m\mathrm{L}^1[0,1]$.
For convenience, we separate the effects of the adjoint boundary forms at each boundary as
\begin{align*}
    \Msups{B}{k}{}{\star}{+} \psi &= \sum^{m}_{r=1} \sum_{j=1}^{n} \Msups{b}{k}{j}{r}{\star} \psi^{(j-1)}_r(0), & k &\in \{1,2,\ldots,\nall\}, \\
    \Msups{B}{k}{}{\star}{-} \psi &=  \sum^{m}_{r=1} \sum_{j=1}^{n} \Msups{\beta}{k}{j}{r}{\star} \psi^{(j-1)}_r(1), & k &\in \{1,2,\ldots,\nall\}.
\end{align*}

\subsubsection{The forward transform} \label{sssec:Interface.TransformPair.BC.Forward}

Let $\alpha=\re^{2\pi\ri/n}$, a primitive $n$\textsuperscript{th} root of unity.
For each $j\in\{1,2,\ldots,\nall\}$, there exists a unique $r\in\{1,2,\ldots,m\}$ such that
\begin{equation} \label{eqn:Interface.rAsFunctionOfj}
    0 \leq j -1 - (r-1)n \leq n-1.
\end{equation}
For that $r$, define
\begin{equation}
    y_{j}(x;\la) := \re^{\ri\overline{\nu_r\left(\alpha^{j-1}\la\right)}x}, \qquad x \in [0,1], \qquad \la\in\CC\setminus D(0,\propertiesnuradius).
\end{equation}
Fixing $\la$ and leaving $j$ free, each $y_j(\argdot;\la)$ is an eigenfunction of the formal differential operator corresponding to $\mathcal{L}^\star$, with eigenvalue $\bar{\la}^n$, and together they span the eigenspace for that eigenvalue.
The characteristic matrix of the adjoint operator $M(\la)$ has entries
\begin{equation}
    \M{M}{j}{k}(\la) := \overline{B_k^\star y_j(\argdot;\la)},
\end{equation}
and the characteristic determinant of the adjoint operator is
\begin{equation}
    \Delta(\la) := \det M(\la).
\end{equation}
We define $\M{\MMinors}{j}{k}$ as the $(\nall-1)\times(\nall-1)$ submatrix of $\begin{bmatrix}M&M\\M&M\end{bmatrix}$ whose $(1,1)$ entry is the $(j+1,k+1)$ entry of its parent.
As before, we also define the left characteristic matrix and right characteristic matrix by
\begin{equation}
    \Msup{M}{j}{k}{\pm}(\la) := \overline{\Msups{B}{k}{}{\star}{\pm} y_j(\argdot;\la)}.
\end{equation}

The forward transforms are
\begin{subequations} \label{eqn:Interface.FokasForwardTransformDefinition.LeftRight}
\begin{align}
    \rFokForTransP{r}{\phi}(\la) &= \frac{\nu_r'(\la)}{2\pi\Delta(\la)} \sum_{j=1}^{\nall} \sum_{k=1}^{\nall} (-1)^{(\nall-1)(j+k)} \det \M{\MMinors}{j}{k}(\la) \Msup{M}{1+(r-1)n}{k}{+}(\la) \left\langle \phi , y_j(\argdot;\la) \right\rangle, \\
    \rFokForTransM{r}{\phi}(\la) &= \frac{-\nu_r'(\la)}{2\pi\Delta(\la)} \sum_{j=1}^{\nall} \sum_{k=1}^{\nall} (-1)^{(\nall-1)(j+k)} \det \M{\MMinors}{j}{k}(\la) \Msup{M}{1+(r-1)n}{k}{-}(\la) \left\langle \phi , y_j(\argdot;\la) \right\rangle.
\end{align}
\end{subequations}
The above formulae define $2m$ forward transforms $\phi\mapsto \rFokForTransPM{r}{\phi}$.
We combine these into a single forward transform
\begin{equation} \label{eqn:Interface.FokasForwardTransformDefinition}
    \phi(x) \mapsto \FokForTrans{\phi}(\la) := \begin{cases}
        \left(\rFokForTransP{1}{\phi}(\la),\rFokForTransP{2}{\phi}(\la),\ldots,\rFokForTransP{m}{\phi}(\la)\right) & \mbox{certain } \la, \\
        \left(\rFokForTransM{1}{\phi}(\la),\rFokForTransM{2}{\phi}(\la),\ldots,\rFokForTransM{m}{\phi}(\la)\right) & \mbox{certain other } \la,
    \end{cases}
\end{equation}
in which the particular choice of $\la$ appears in~\S\ref{sssec:Interface.TransformPair.BC.Inverse}.

\subsubsection{The inverse transform} \label{sssec:Interface.TransformPair.BC.Inverse}

Let $\Delta^\principalpart$ be the characteristic determinant of the principal part of the adjoint operator.
Let $R>\propertiesnuradius$ and $\epsilon>0$ be such that all zeros of $\Delta$ lie in the union
\[
    D(0,R) \cup \bigcup_{\substack{\sigma\in\CC\setminus D(0,R):\\\Delta^\principalpart(\sigma)=0}} D(\sigma,\epsilon)
\]
but
\[
    D(0,R) \cap \bigcup_{\substack{\sigma\in\CC\setminus D(0,R):\\\Delta^\principalpart(\sigma)=0}}D(\sigma,3\epsilon) = \emptyset,
\]
and $\epsilon$ is less than $1/5$ the infimal separation of the zeros of $\Delta^\principalpart$.
We define the sets
\begin{subequations}
\begin{align}
    \CC_{\nu_r}^\pm &= \{\la\in\CC:\abs\la>\propertiesnuradius,\pm\Im(\nu_r(\la))>0\}, \\
    Z_r^+ &= \{\sigma\in\clos\CC_{\nu_r}^+:\Delta^\principalpart(\sigma)=0,\abs\sigma>R\}, \\
    Z_r^- &= \{\sigma\in\CC_{\nu_r}^-:\Delta^\principalpart(\sigma)=0,\abs\sigma>R\}, \\
    Z_\cuts &= \{\sigma\in\CC:\Delta^\principalpart(\sigma)=0,\abs\sigma<R\},
\end{align}
and the contours
\begin{align}
    \Gamma_r &= \M{\Gamma}{0}{r} \cup \M{\Gamma}{\BVec{a}}{r} \cup \Gamma_\cuts, \\
    \M{\Gamma}{0}{r} &= \Msup{\Gamma}{0}{r}{+} \cup \Msup{\Gamma}{0}{r}{-}, \\
    \Msup{\Gamma}{0}{r}{\pm} &= \bigcup_{\sigma\in Z_r^\pm} C(\sigma,\epsilon), \\
    \M{\Gamma}{\BVec{a}}{r} &= \Msup{\Gamma}{\BVec{a}}{r}{+} \cup \Msup{\Gamma}{\BVec{a}}{r}{-}, \\
    \Msup{\Gamma}{\BVec{a}}{r}{\pm} &= \partial \left( \{\la\in\CC_{\nu_r}^\pm:\Re({a_r}\la^n)>0,\,\abs\la>R\} \setminus \bigcup_{\sigma\in Z_r^+\cup Z_r^-} D(\sigma,2\epsilon) \right), \\
    \Gamma_\cuts &= C(0,R) \\
    \Gamma_r^\pm &= \Msup{\Gamma}{0}{r}{\pm} \cup \Msup{\Gamma}{\BVec{a}}{r}{\pm}.
\end{align}
\end{subequations}

For $F$, a bivalued vector valued function $F^\pm$ of a single complex variable, the inverse Fokas transform is defined by
\begin{subequations} \label{eqn:Interface.FokasInverseTransformDefinition}
\begin{equation}
    F(\la) \mapsto \FokInvTrans{F}(x) = \left( \rFokInvTrans{1}{F}(x) , \rFokInvTrans{2}{F}(x) , \ldots , \rFokInvTrans{m}{F}(x) \right)
\end{equation}
where
\begin{equation}
    \rFokInvTrans{r}{F}(x) := \CPVwithoutintegral \left( \int_{\Gamma_r^+\cup\Gamma_\cuts} \re^{\ri\nu_r(\la)x}F_r^+(\la) \D\la + \int_{\Gamma_r^-} \re^{\ri\nu_r(\la)x}F_r^-(\la) \D\la \right).
\end{equation}
\end{subequations}

\subsubsection{Validity} \label{sssec:Interface.TransformPair.BC.Validity}

We also define contours
\begin{equation}
    \Msup{\hat\Gamma}{\BVec{a}}{r}{\pm} = \partial \left( \{\la\in\CC_{\nu_r}^\pm:\Re(a_r\la^n)<0,\,\abs\la>R\} \setminus \bigcup_{\sigma\in Z_r^+\cup Z_r^-} D(\sigma,2\epsilon) \right).
\end{equation}

We claim that, under certain conditions, the transform~\eqref{eqn:Interface.FokasInverseTransformDefinition} is truly the inverse of transform~\eqref{eqn:Interface.FokasForwardTransformDefinition}.

\begin{asm} \label{asm:Interface.MainAssumption}
    Suppose $\phi$ is such that, for all $x\in(0,1)$, and all $r\in\{1,2,\ldots,m\}$,
    \begin{equation*}
        \CPVwithoutintegral\left( \int_{\Msup{\hat\Gamma}{\BVec{a}}{r}{+}} \re^{\ri\nu_r(\la)x} \rFokForTransP{r}{\phi}(\la)\D\la + \int_{\Msup{\hat\Gamma}{\BVec{a}}{r}{-}} \re^{\ri\nu_r(\la)x} \rFokForTransM{r}{\phi}(\la)\D\la \right) = 0.
    \end{equation*}
\end{asm}

\begin{thm} \label{thm:Interface.FokasTransformValid}
    Suppose $\phi\in \prod_{r=1}^m \mathrm{C}^1[0,1]$ and assumption~\ref{asm:Interface.MainAssumption} holds.
    Then, for all $x\in(0,1)$,
    \begin{equation*}
        \FokInvTrans{\FokForTrans{\phi}}(x) = \phi(x).
    \end{equation*}
\end{thm}

\begin{lem} \label{lem:Interface.FokasTransformCancellation}
    For all $\la\in\CC$ and all $r\in\{1,2,\ldots,m\}$,
\[
    \rFokForTransP{r}{\phi}(\la) - \rFokForTransM{r}{\phi}(\la) = \frac{\nu_r'(\la)}{2 \pi} \rFormalForTrans{r}{\phi}(\la).
\]
\end{lem}

\begin{proof}
    By the definition of transforms $\rFokForTransPM{r}{\phi}$ and the characteristic matrix $M$,
    \begin{multline*}
        \rFokForTransP{r}{\phi}(\la) - \rFokForTransM{r}{\phi}(\la) \\
        = \frac{\nu_r'(\la)}{2\pi\Delta(\la)} \sum_{j=1}^{\nall} \sum_{k=1}^{\nall} (-1)^{(\nall-1)(j+k)} \det \M{\MMinors}{j}{k}(\la) \M{M}{1+(r-1)n}{k}(\la) \left\langle \phi , y_j(\argdot;\la) \right\rangle.
	\end{multline*}
	An application of the cyclic cofactor expansion of determinants yields
    \[
        \sum_{k=1}^{\nall}(-1)^{(\nall-1)(j+k)}\det\M{\MMinors}{j}{k}(\la)\M{M}{1+(r-1)n}{k}(\la) = \delta_{1+(r-1)n, j}\Delta(\la),
    \]
    for $\delta_{1+(r-1)n,j}$ the Kronecker delta.
    By definition,
    \[
        y_{1+(r-1)n}(x;\la) = \re^{\ri\overline{\nu_r\left(\la\right)}x}.
    \]
    Thus,
    \[
        \rFokForTransP{r}{\phi}(\la) - \rFokForTransM{r}{\phi}(\la) = \frac{\nu_r'(\la)}{2\pi \Delta(\la)} \Delta(\la) \int_0^1 \re^{-\ri\nu_r(\la)x} \phi_r(x) \D x = \frac{\nu_r'(\la)}{2 \pi} \rFormalForTrans{r}{\phi}(\la).
        \qedhere
    \]
\end{proof}

\begin{proof}[Proof of theorem~\ref{thm:Interface.FokasTransformValid}]
    The claimed identity is a list of $m$ scalar equations.
    The proof of each proceeds exactly as the proof of theorem~\ref{thm:FokasTransformValid}, but with assumption~\ref{asm:Interface.MainAssumption} in place of assumption~\ref{asm:MainAssumption}, and lemma~\ref{lem:Interface.FokasTransformCancellation} playing the role of lemma~\ref{lem:FokasTransformCancellation}.
\end{proof}

\subsection{Diagonalization} \label{ssec:Interface.Diagonalization}

\begin{thm} \label{thm:Interface.Diag}
    There exists a \emph{remainder transform} $\FokRemTransNoArg$, such that, for all $\phi\in\Phi_{\BVec{B}}$,
    \begin{subequations} \label{eqn:Interface.DiagThm}
    \begin{equation} \label{eqn:Interface.DiagThm.1}
        \FokForTrans{L\phi}(\la) = \la^n \FokForTrans{\phi}(\la) + \FokRemTrans{\phi}(\la)
    \end{equation}
    and,
    if
    $q=(q_1,q_2,\ldots,q_m)$, with each $q_r:[0,1]\times[0,T]\to\CC$, is such that, for all $t\in[0,T]$, $q(\argdot,t)\in\Phi$ and, for all $r\in\{1,2,\ldots,m\}$ and $j\in\{0,1,\ldots,n-1\}$, uniformly for all $x\in[0,1]$, $\partial_x^jq_r(x,\argdot)$ is a function of bounded variation,
    then,
    for all $t\in[0,T]$ and all $x\in(0,1)$,
    \begin{equation} \label{eqn:Interface.DiagThm.2}
        \FokInvTrans{\int_0^t \left( \re^{a_1\argdot^{n}(s-t)}, \re^{a_2\argdot^{n}(s-t)}, \ldots, \re^{a_m\argdot^{n}(s-t)} \right) \circ \FokRemTrans{q}(\argdot;s) \D s}(x) = 0.
    \end{equation}
    \end{subequations}
\end{thm}

\begin{lem} \label{lem:Interface.PolynomialGrowth}
    For each $j\in\{1,2,\ldots,\nall\}$, let the corresponding $r$ be that defined by inequalities~\eqref{eqn:Interface.rAsFunctionOfj}.
    The function $\Msup{M}{j}{k}{+}(\la)$ is a polynomial in $\nu_r(\alpha^{j-1}\la)$ (so a holomorphic function of $\la$ outside a finite disc) and has polynomial growth rate $\bigoh{\abs{\la}^{n-1}}$, uniformly in $\arg(\la)$, as $\la\to\infty$.
    The function $\Msup{M}{j}{k}{-}(\la)$ is the product of $\re^{-\ri\nu_r(\alpha^{j-1}\la)}$ with another function; the latter function is a polynomial in $\nu_r(\alpha^{j-1}\la)$ (so a holomorphic function of $\la$ outside a finite disc) and has polynomial growth rate $\bigoh{\abs{\la}^{n-1}}$, uniformly in $\arg(\la)$, as $\la\to\infty$.
\end{lem}

\begin{proof}
    Noting that $y_j$ is a list in which only one entry, the $r$\textsuperscript{th} entry, is nonzero, the definition of $\Msup{M}{j}{k}{\pm}(\la)$ reduces to a formula exactly like that appearing in the proof of lemma~\ref{lem:PolynomialGrowth}.
    The proof proceeds as for that lemma.
\end{proof}

\begin{lem} \label{lem:Interface.DiagMainLemma}
    For $\phi\in\Phi_{\BVec{B}}$, for each $r\in\{1,2,\ldots,m\}$, there exist functions, $\Msup{P}{r}{\phi}{\pm}(\la)$, analytic in $\CC\setminus D(0,\propertiesnuradius)$, with $\Msup{P}{r}{\phi}{\pm}(\la) = \bigoh{\abs{\la}^{n-1}}$, uniformly in $\arg(\la)$, as $\la\to\infty$ and
    \begin{subequations} \label{eqn:Interface.FokasMethodSolRepMainLemma.ApproxEigs}
    \begin{align}
    	\rFokForTransP{r}{L\phi}(\la) &= \lambda^{n} \rFokForTransP{r}{\phi}(\la) + \Msup{P}{r}{\phi}{+}(\la), \\
		\rFokForTransM{r}{L\phi}(\la) &= \lambda^{n} \rFokForTransM{r}{\phi}(\la) + \re^{-\ri \nu_r(\la)}\Msup{P}{r}{\phi}{-}(\la).
    \end{align}
    \end{subequations}
    Moreover, the maps $\phi \mapsto \Msup{P}{r}{\phi}{\pm}$ are linear.
    Indeed, such functions are given by
    \begin{subequations} \label{eqn:Interface.FokasMethodSolRepMainLemma.PpmFormulae}
    \begin{align}
        \label{eqn:Interface.FokasMethodSolRepMainLemma.PpmFormulae.+}
        \Msup{P}{r}{\phi}{+}(\la) &= \frac{\nu'_r(\la)}{2\pi} \BVec{B}_\complementary \phi \odot \left( \overline{\Msup{M}{1+(r-1)n}{1}{+}(\la)} , \overline{\Msup{M}{1+(r-1)n}{2}{+}(\la)} , \ldots , \overline{\Msup{M}{1+(r-1)n}{\nall}{+}(\la)} \right), \\
        \label{eqn:Interface.FokasMethodSolRepMainLemma.PpmFormulae.-}
        \re^{-\ri \nu_r(\la)} \Msup{P}{r}{\phi}{-}(\la) &= \frac{-\nu'_r(\la)}{2\pi} \BVec{B}_\complementary \phi \odot \left( \overline{\Msup{M}{1+(r-1)n}{1}{-}(\la)} , \overline{\Msup{M}{1+(r-1)n}{2}{-}(\la)} , \ldots , \overline{\Msup{M}{1+(r-1)n}{\nall}{-}(\la)} \right),
    \end{align}
    \end{subequations}
    in which $\BVec{B}_\complementary$ represents a vector of boundary forms which is complementary (in the sense of definition~\ref{defn:AdjointOperator.ComplimentaryBoundaryForm}) to the vector of boundary forms $\BVec{B}$ compatible with $\BVec{B}^\star$, and where $\odot$ is the usual sesquilinear dot product on $\CC^{\nall}$.
\end{lem}

\begin{proof}
    Fix $r \in \{ 1, \ldots, m\}$.
    For notational convenience, for each $\la\in\CC\setminus D(0,\propertiesnuradius)$ and $r\in\{1,2,\ldots,m\}$, we define functions $\psi^\pm(\argdot;\la,r)\in\Phi$ via the complex conjugates of their entries,
	\begin{equation*}
		\overline{\psi^{\pm}_{u}(x; \la, r)} =  \frac{\pm\nu'_r(\la)}{2\pi\Delta(\la)} \sum_{j=1+(u-1)n}^{un} \sum_{\ell=1}^{\nall} (-1)^{(\nall-1)(j+\ell)} \det \M{\MMinors}{j}{\ell}(\la) \Msup{M}{1+(r-1)n}{\ell}{\pm}(\la) \re^{-\ri\nu_u(\alpha^{j-1}\la) x},
	\end{equation*}
	so that
	\begin{equation} \label{eqn:Interface.FokasMethodSolRepMainLemma.Proof.InnerProduct}
		\rFokForTransPM{r}{\phi}(\la)
        = \left\langle \phi, \psi^{\pm}(\argdot;\la,r)\right\rangle
        = \sum^{m}_{u=1} \left\langle \phi_u, \psi_u^{\pm}(\argdot;\la,r)\right\rangle
        = \sum^{m}_{u=1} \int^{1}_{0} \phi_{u}(x) \overline{\psi^{\pm}_{u}(x, \la;r)} \D x.
    \end{equation}
    Moreover, for any $\la\in\CC\setminus D(0,\propertiesnuradius)$ and $r\in\{1,2,\ldots,m\}$, $\psi^{\pm}(\argdot; \la, r)$ is an eigenfunction of $\mathcal{L}^\star$ with eigenvalue $\bar{\la}^n$.

    By definition of $B^{\star}_k$ and $\psi^{\pm}$, we have
	\begin{align*}
	 \overline{B^{\star}_k\psi^{\pm}(x;\la,r)} &= \overline{B^{\star}_k\vphantom{()}} \; \overline{\psi^{\pm}(x;\la,r)} \\
		&\hspace{-5em} {}={} \frac{\pm\nu'_r(\la)}{2\pi\Delta(\la)} \sum^{m}_{u=1} \sum_{j=1+(u-1)n}^{un} \sum_{\ell=1}^{\nall} (-1)^{(\nall-1)(j+\ell)} \det \M{\MMinors}{j}{\ell}(\la) \Msup{M}{1+(r-1)n}{\ell}{\pm}(\la) \overline{B^{\star}_k\vphantom{()}}\left(\re^{-\ri\nu_u(\alpha^{j-1}\la) \argdot}\right).
	\end{align*}
	But, for $j$ so dependent on $u$,
    \[
        \overline{B^{\star}_k\vphantom{()}}\left(\re^{-\ri\nu_u(\alpha^{j-1}\la) \argdot}\right) = \overline{B_k^\star\vphantom{()}} \; \overline{y_{j}(x;\la)} = \M{M}{j}{k}(\la),
    \]
    and
    \[
        \sum^{m}_{u=1} \sum_{j=1+(u-1)n}^{un} (-1)^{(\nall-1)(j+\ell)} \det \M{\MMinors}{j}{\ell}(\la) \M{M}{j}{k}(\la)  = \delta_{\ell,k}\Delta(\la),
    \]
    for $\delta_{\ell,k}$ the Kronecker delta.
    Hence
    \begin{equation} \label{eqn:Interface.FokasMethodSolRepMainLemma.Proof.AdjBFpsi1}
        \overline{B^{\star}_k\psi^{\pm}(x;\la,r)} =  \frac{\pm\nu'_r(\la)}{2\pi} \Msup{M}{1+(r-1)n}{k}{\pm}(\la).
    \end{equation}

    Now, by the first of equations~\eqref{eqn:Interface.FokasMethodSolRepMainLemma.Proof.InnerProduct}, the interface boundary form formula theorem~\ref{thm:AdjointOperator.BFF}, and $\psi^\pm$ being an eigenfunction of $\mathcal{L}^\star$,
	\begin{alignat*}{2}
        \rFokForTransPM{r}{L\phi}(\la)
        &= \left\langle L\phi, \psi^{\pm}(\argdot;\la,r)\right\rangle & & \\
        &= \left\langle \phi, \mathcal{L}^\star\psi^{\pm}(\argdot;\la,r)\right\rangle & &+ \BVec{B} \phi \odot \BVec{B}_\complementary^{\star} \psi^{\pm}(\argdot;\la,r) + \BVec{B}_\complementary \phi \odot \BVec{B}^{\star}\psi^{\pm}(\argdot;\la,r) \\
        &= \la^n \left\langle \phi, \psi^{\pm}(\argdot;\la,r)\right\rangle & &+ \BVec{B} \phi \odot \BVec{B}_\complementary^{\star} \psi^{\pm}(\argdot;\la,r) + \BVec{B}_\complementary \phi \odot \BVec{B}^{\star}\psi^{\pm}(\argdot;\la,r),
    \end{alignat*}
	where $\BVec{B}_\complementary^{\star}$ is the vector of boundary forms complementary to $\BVec{B}^\star$ compatible with $\BVec{B}$ and $\BVec{B}_\complementary$.
    The term $\BVec{B} \phi \odot \BVec{B}_\complementary^{\star} \psi^{\pm}$ evaluates to zero by virtue of $\phi\in\Phi_{\BVec{B}}$.
    Using equations~\eqref{eqn:Interface.FokasMethodSolRepMainLemma.Proof.InnerProduct} and~\eqref{eqn:Interface.FokasMethodSolRepMainLemma.Proof.AdjBFpsi1}, we simplify the other terms to arrive at
    \begin{multline*}
        \rFokForTransPM{r}{L\phi}(\la)
        = \la^n \rFokForTransPM{r}{\phi}(\la) \\
        \pm \frac{\nu'_r(\la)}{2\pi} \BVec{B}_\complementary \phi \odot \left( \overline{\Msup{M}{1+(r-1)n}{1}{+}(\la)} , \overline{\Msup{M}{1+(r-1)n}{2}{+}(\la)} , \ldots , \overline{\Msup{M}{1+(r-1)n}{mn}{+}(\la)} \right).
    \end{multline*}
    Defining $\Msup{P}{r}{\phi}{+}(\la)$ according to equations~\eqref{eqn:Interface.FokasMethodSolRepMainLemma.PpmFormulae}, the approximate eigenfunction equations~\eqref{eqn:Interface.FokasMethodSolRepMainLemma.ApproxEigs} follow.

    The $\lambda$ analyticity and asymptotic results follow from equations~\eqref{eqn:Interface.FokasMethodSolRepMainLemma.PpmFormulae} via lemma~\ref{lem:Interface.PolynomialGrowth} and proposition~\ref{prop:propertiesnu}.
    Linear dependence on $\phi$ is a result of linearity of $\BVec{B}_\complementary$ and linearity of $\odot$ in its first argument.
\end{proof}

\begin{proof}[Proof of theorem~\ref{thm:Interface.Diag}]
    For $r\in\{1,2,\ldots,m\}$, we define $\rFokRemTransNoArg{r}$ by
    \[
        \rFokRemTrans{r}{\phi}(\la) =
        \begin{cases}
            \Msup{P}{r}{\phi}{+}(\la) & \la \in \Gamma_r^+ \cup \Gamma_\cuts, \\
            \Msup{P}{r}{\phi}{-}(\la)\re^{-\ri\nu_r(\la)} & \la \in \Gamma_r^-,
        \end{cases}
    \]
    and $\FokRemTransNoArg$ by
    \[
        \FokRemTrans{\phi} = \left( \rFokRemTrans{1}{\phi} , \rFokRemTrans{2}{\phi} , \ldots , \rFokRemTrans{m}{\phi} \right),
    \]
    so that equation~\eqref{eqn:Interface.DiagThm.1} is a consequence of lemma~\ref{lem:Interface.DiagMainLemma}.

    The object on the left of equation~\eqref{eqn:Interface.DiagThm.2} is a list of $m$ functions of $x$.
    The $r$\textsuperscript{th} entry in that list is the joint principal value contour integral
    \begin{multline*}
        \CPVwithoutintegral \left( \int_{\Gamma_r^+\cup\Gamma_\cuts} \re^{\ri\nu_r(\la)x-a_r\la^nt} \int_0^t \re^{a_r\la^ns}  \Msup{P}{r}{q(\argdot,s)}{+}(\la) \D s \D\la \right. \\
        \left. + \int_{\Gamma_r^-} \re^{\ri\nu_r(\la)(x-1)-a_r\la^nt} \int_0^t \re^{a_r\la^ns}  \Msup{P}{r}{q(\argdot,s)}{-}(\la) \D s \D\la \right).
    \end{multline*}
    The proof, decomposing this joint principal value into five constituent principal value contour integrals and showing that each evaluates to zero, follows the proof of theorem~\ref{thm:Diag}, except that references to lemmata~\ref{lem:PolynomialGrowth} and~\ref{lem:DiagMainLemma} are replaced with appeals to lemmata~\ref{lem:Interface.PolynomialGrowth} and~\ref{lem:Interface.DiagMainLemma}.
\end{proof}

\subsection{Fokas transform method for IIVP} \label{ssec:Interface.Method}
\begin{thm} \label{thm:Interface.FokasMethodSolRep}
    Suppose the criteria of theorem~\ref{thm:Interface.FokasTransformValid} hold, and there exists a solution $q(x,t)$ of problem~\eqref{eqn:Interface.IBVP} satisfying the criteria of theorem~\ref{thm:Interface.Diag}.
    Then
    \[
        q(x,t) = \FokInvTrans{ \left( \re^{-a_1\argdot^{n}t}, \re^{-a_2\argdot^{n}t}, \ldots, \re^{-a_m\argdot^{n}t} \right) \circ \FokForTrans{Q}}(x).
    \]
\end{thm}

\begin{proof}
    The argument proceeds exactly as in the proof of theorem~\ref{thm:FokasMethodSolRep}, except with problem~\eqref{eqn:IBVP} replaced by problem~\eqref{eqn:Interface.IBVP}, and with calls to theorems~\ref{thm:FokasTransformValid} and~\ref{thm:Diag} substituted for invocations of theorems~\ref{thm:Interface.FokasTransformValid} and~\ref{thm:Interface.Diag}.
\end{proof}

\subsection{Fokas transform method for inhomogeneous IIVP} \label{ssec:Interface.Inhomogeneous}

If $\Phi$ is defined as in~\S\ref{ssec:Interface.Operator}, for any $\BVec{h}\in\CC^{\nall}$, let $\M{\Phi}{\BVec{B}}{\BVec{h}} = \{\phi\in\Phi: \BVec{B}\phi=\BVec{h}\}$.
We study the inhomogeneous IIVP
\begin{subequations} \label{eqn:Interface.Inhomogeneous.IBVP}
\begin{align}
    \label{eqn:Interface.Inhomogeneous.IBVP.PDE} \tag{\theparentequation.PDE}
    \partial_t q(x,t) + \BVec{a} \circ \mathcal{L}q(\argdot,t) &= \mathcal{Q}(x,t) & (x,t) &\in (0,1) \times (0,T), \\
    \label{eqn:Interface.Inhomogeneous.IBVP.IC} \tag{\theparentequation.IC}
    q(x,0) &= Q(x) & x &\in[0,1], \\
    \label{eqn:Interface.Inhomogeneous.IBVP.BC} \tag{\theparentequation.BC}
    \BVec{B}q(\argdot,t) &= \BVec{h}(t) & t &\in[0,T],
\end{align}
\end{subequations}
in which $\mathcal{Q}$ and $\BVec{h}=(h_1,h_2,\ldots,h_{\nall})$ are appropriately smooth functions, $Q\in\M{\Phi}{\BVec{B}}{\BVec{h}(0)}$, and $\BVec{a}$ is as in IIVP~\eqref{eqn:Interface.IBVP}.

To study problems in which $\mathcal{Q}$ and $\BVec{h}$ are nonzero, it is necessary to generalize once again our diagonalization and Fokas transform method theorems.
We combine theorems~\ref{thm:Interface.Diag} and~\ref{thm:Inhomogeneous.Diag} to derive from theorem~\ref{thm:Diag} the inhomogeneous interface diagonalization theorem~\ref{thm:Interface.Inhomogeneous.Diag}.
A synthesis of the advances in theorems~\ref{thm:Interface.FokasMethodSolRep} and~\ref{thm:Inhomogeneous.FokasMethodSolRep} over theorem~\ref{thm:FokasMethodSolRep} provides inhomogeneous interface Fokas transform method theorem~\ref{thm:Interface.Inhomogeneous.FokasMethodSolRep}.
The proofs of both theorems use very similar arguments to those of their progenitors.

\begin{thm} \label{thm:Interface.Inhomogeneous.Diag}
    There exists an \emph{inhomogeneous boundary term}
    \[
        \mathcal{H}[\BVec{h}](\la) = \left( \mathcal{H}_1[\BVec{h}](\la) , \mathcal{H}_2[\BVec{h}](\la) , \ldots , \mathcal{H}_m[\BVec{h}](\la) \right)
    \]
    defined by
    \[
        \mathcal{H}_r[\BVec{h}](\la) = \begin{cases}
            \BVec{h} \odot \BVec{B}_\complementary^\star \psi^+(\argdot;\la,r) & \mbox \la\in\Gamma_r^+\cup\Gamma_\cuts, \\
            \BVec{h} \odot \BVec{B}_\complementary^\star \psi^-(\argdot;\la,r) & \mbox \la\in\Gamma_r^-,
        \end{cases}
    \]
    for $\psi^\pm$ as in the proof of lemma~\ref{lem:Interface.DiagMainLemma},
    and a \emph{remainder transform} $\FokRemTransNoArg$, such that, for all $\phi\in\Phi_{\BVec{B}}$,
    \begin{subequations} \label{eqn:Interface.Inhomogeneous.DiagThm}
    \begin{equation} \label{eqn:Interface.Inhomogeneous.DiagThm.1}
        \FokForTrans{\mathcal{L}\phi}(\la) = \la^n \FokForTrans{\phi}(\la) + \FokRemTrans{\phi}(\la) + \mathcal{H}[\BVec{h}](\la)
    \end{equation}
    and,
    if
    $q=(q_1,q_2,\ldots,q_m)$, with each $q_r:[0,1]\times[0,T]\to\CC$, is such that, for all $t\in[0,T]$, $q(\argdot,t)\in\Phi$ and, for all $r\in\{1,2,\ldots,m\}$ and $j\in\{0,1,\ldots,n-1\}$, uniformly for all $x\in[0,1]$, $\partial_x^jq_r(x,\argdot)$ is a function of bounded variation,
    then,
    for all $t\in[0,T]$ and all $x\in(0,1)$,
    \begin{equation} \label{eqn:Interface.Inhomogeneous.DiagThm.2}
        \FokInvTrans{\int_0^t \left( \re^{a_1\argdot^{n}(s-t)}, \re^{a_2\argdot^{n}(s-t)}, \ldots, \re^{a_m\argdot^{n}(s-t)} \right) \circ \FokRemTrans{q}(\argdot;s) \D s}(x) = 0.
    \end{equation}
    \end{subequations}
\end{thm}

\begin{thm} \label{thm:Interface.Inhomogeneous.FokasMethodSolRep}
    Suppose the criteria of theorem~\ref{thm:Interface.FokasTransformValid} hold, and there exists a solution $q(x,t)$ of problem~\eqref{eqn:Interface.Inhomogeneous.IBVP} satisfying the criteria of theorem~\ref{thm:Interface.Inhomogeneous.Diag}.
    Then
    \begin{multline*}
        q(x,t)
        = \FokInvTrans{ \left( \re^{-a_1\argdot^{n}t}, \re^{-a_2\argdot^{n}t}, \ldots, \re^{-a_m\argdot^{n}t} \right) \circ \FokForTrans{Q}}(x) \\
        + \FokInvTrans{ \int_0^t\left( \re^{a_1\argdot^{n}(s-t)}, \re^{a_2\argdot^{n}(s-t)}, \ldots, \re^{a_m\argdot^{n}(s-t)} \right) \circ \FokForTrans{\mathcal{Q}}(\argdot;s)\D s}(x) \\
        - \BVec{a} \circ \FokInvTrans{ \int_0^t\left( \re^{a_1\argdot^{n}(s-t)}, \re^{a_2\argdot^{n}(s-t)}, \ldots, \re^{a_m\argdot^{n}(s-t)} \right) \circ \mathcal{H}[\BVec{h}](\argdot;s)\D s}(x).
    \end{multline*}
\end{thm}

\section{Conclusion} \label{sec:Conclusion}

We identified in~\S\ref{ssec:Classical.Sketch} a theoretical generalization of the classical spatial spectral transform approach to solution of IBVP using a weaker diagonalization criterion.
In the remainder of the work, we implemented this generalization, and further extended it to inhomogeneous IBVP and to IIVP both homogeneous and inhomogeneous.
For many such problems, this is the first solution method implemented.
For others, it provides an integral solution representation to complement the classical series solution representation, or provides a newly simplified method of deriving the integral form.

Contrasting the solution representation we present here with~\cite{FP2001a,FP2005a}, one immediately notices that the form of the transforms is different, but only a change of variables is required to map from the transform pairs of Fokas and Pelloni to the expressions given here.
Specifically, the spectral parameter $k$ of Fokas and Pelloni satisfies $k=\nu(\alpha^j\la)$ for a different integer $j$ in each connected component of $\Gamma$, and the contour $\Gamma$ is deformed appropriately for this change of variables.
We find the newer presentation preferable, because it simplifies the proof of lemma~\ref{lem:FokasTransformCancellation} by obviating a change of variables, and because it requires the implicit definition of only one biholomorphic function $\nu$ instead of $n^2$ biholomorphic functions.
Moreover, as first demonstrated in~\cite{DS2020a} and advanced in~\S\ref{sec:Interface}, interface problems in which $\omega$ varies across the interface are much more easily approached in the new formulation.
However, we emphasize that, at least in the cases for which the original or simplified version of the Fokas transform method has been implemented, the difference in solution representation amounts to nothing more than choosing the optimal place in the derivation to make a certain change of variables in the spectral parameter; there is no fundamental difference between the transform pair implicitly derived in~\cite{FP2001a,FP2005a} and the transform pair presented in~\ref{sec:TransformPair}.

It is significant that the original and simplified versions of the Fokas transform method have not been detailed for general IBVP~\eqref{eqn:IBVP} or its full generalizations~\eqref{eqn:Inhomogeneous.IBVP},~\eqref{eqn:Interface.IBVP},
or~\eqref{eqn:Interface.Inhomogeneous.IBVP} before this work.
In the original and simplified versions, more complicated boundary conditions represent a great deal of extra work to set up and solve linear systems; with the true transform version that is easily automated in terms of well known matrices associated with two point differential operators.
With the benefit of hindsight, we believe that the true transform version provides a more natural view of the Fokas transform method, couched more explicitly in the spectral theory of the spatial differential operator, instead of in a complex analytic framework inherited from the inverse scattering transform.

\subsection{Open problems}
\subsubsection{Spectral theory}
The main diagonalization theorem~\ref{thm:Diag} is a spectral theorem suitable for a broad class of two point differential operators $L$ including some, such as the spatial part of IBVP~\eqref{eqn:3ordIBVP} that are Locker degenerate irregular.
The delicacy of equation~\eqref{eqn:DiagThm.2} and its intimate connection with the solution of IBVP~\eqref{eqn:IBVP} suggest that a spectral theory derived from theorem~\ref{thm:Diag} will require careful construction.
But polynomials of operators $L$, which are clearly defined without any diagonalization theorem, interact with theorem~\ref{thm:Diag} as one might hope, at least when the polynomial of $L$ is interpreted as the spatial parts of IBVP:
for $\Omega$ a polynomial of degree $K$, $\Omega(L)$ has domain
\[
    \Phi_{\BVec{B}^K} := \left\{ \phi\in\AC^{Kn-1}[0,1] : \forall k\in\{0,1,\ldots,K-1\},\, \BVec{B}\phi^{(nk)} = \BVec{0} \right\},
\]
is defined by $\Omega(L)\phi:=\Omega(\mathcal{L})\phi$,
and satisfies, for a remainder functional $\GenericRemTransNoArg$,
\[
    \FokForTrans{\Omega(L)\phi} = \Omega\left(\la^n\right)\FokForTrans{\phi} + \GenericRemTrans{\phi},
\]
in which, for all appropriate $q$,
\[
    \FokInvTrans{ \int_0^t\re^{a\argdot^{Kn}(s-t)} \GenericRemTrans{q}(\argdot;s) \D s } = 0.
\]
Indeed, $\GenericRemTrans{\phi}$ can be expressed as a sum of polynomials in $\la^n$ multiplying $\FokRemTransNoArg$ applied to powers of $L$ of $\phi$.
It may be possible to extend this concept to a functional calculus of operators $L$ in much the same way as has been achieved for Laplace operators.

\subsubsection{Other classes of I(B/I/N)VP}
In the interest of brevity, we restricted the present study to PDE of the form~\eqref{eqn:IBVP.PDE}, in which there is but a single term with a temporal derivative, and that term has order one.
In contrast, the simplified version of the Fokas transform method has already been formulated for linear systems of PDE, and any PDE with higher order temporal derivatives~\cite{DGSV2018a} may be reduced to such.
Even the original version has been applied to some linear systems~\cite{FP2005b}.
The simplified version has also been adapted to study problems in which the temporally leading order derivative is mixed~\cite{DV2013a}.
But all of these approaches are limited in the types of boundary conditions they study.
It is reasonable to expect that, via explicit formulation of a diagonalization theorem for the appropriate spatial linear system differential operators and rational pseudodifferential operators, a true transform version of the Fokas transform method may be devised for all IBVP in these classes too.
Other nonlocal PDE have not yet been studied via any version of the Fokas transform method, nor diagonalized in a way that permits such general linear boundary conditions.

We have studied PDE on domains with a broad range of one dimensional spatial geometries.
In multiple spatial dimensions without separability, the Fokas transform method enjoys some advantages over classical spectral methods, but still has restricted applicability~\cite{FS2012a}.
Remaining with one spatial dimension, half line domains, and network domains including some semiinfinite components could be analysed using the true transform version of the Fokas transform method.
This was already done for PDE without lower order terms~\cite{PS2016a}, and many interface problems on such domains have been solved using the simplified version.

There remains another class of linear boundary conditions, generalizing even multipoint and interface conditions, not considered here: nonlocal conditions.
Instead of specifying a value at some boundary or interior point of a spatial domain, one may require a weighted spatial integral of the solution take on a particular value at all times.
These initial nonlocal value problems (INVP) are useful in applications where it is difficult to measure a physical quantity at a point, but easier to measure a (possibly weighted) mean of that quantity over an interval.
Such problems for the heat equation have been studied by adapting the simplified version of the Fokas transform method~\cite{MS2018a}, providing the first analytic solution of these problems for nonconstant weights.
Extension of the true transform version to such INVP would require a construction of an appropriate analogue of the Lagrange adjoint of a differential operator with nonlocal conditions.

\subsubsection{Wellposedness and Birkhoff regularity}
The main transform validity results in this paper are predicated on assumption~\ref{asm:MainAssumption} and its interface generalization, assumption~\ref{asm:Interface.MainAssumption}.
Hence so also is the true transform version of the Fokas transform method.
In~\S\ref{sssec:ValidityBirkhoff}, there is some discussion of the first of these assumptions for operators without lower order terms, including the assumption's links with Locker's criteria for Birkhoff regularity, and a sketch justification.
A full investigation of these assumptions is warranted.
It is expected that the weak Birkhoff regularity characterization of the assumptions will extend to Fokas transforms defined for arbitrary two point and interface operators.

\section*{Acknowledgement}

\AckYNCSRP{Aitzhan}{2019}
\AckYNCSRP{Bhandari}{2020}
\AckYNCProj{Smith}
Smith delivered one section of this paper as part of a summer school in July 2018 at ICTS Bengaluru in the program \emph{Integrable systems in mathematics, condensed matter and statistical physics} code ICTS/integrability2018/07, and at the Isaac Newton Institute program \emph{Complex analysis: techniques, applications and computations} supported by EPSRC Grant Number EP/R014604/1.
Smith would like to thank ICTS and the Isaac Newton Institute for their hospitality and the participants in both programs for valuable conversations.

\appendix

\section{Adjoints of multipoint and interface differential operators} \label{sec:AdjointOperator}

The definition of the Fokas transform in~\S\ref{sec:Interface} relies on the classical (or Lagrange) adjoint of the interface differential operator.
In this appendix, we undertake the construction of that adjoint, and its generalization to interface operators in which the component operators $L_r$ are permitted to be of different orders from one another, and are permitted to have variable coefficients.
The construction follows closely that laid out by Coddington and Levinson~\cite[chapter~11]{CL1955a}, and generalizes those results to and beyond what is presented in~\cite{Loc1973a}.
We use notation aligned with the main sections of this paper, but all necessary definitions and arguments are contained within this appendix, so that it may be read independently.

\subsection{Formulation of the Problem} \label{ssec:AdjointOperator.Formulation}

Let $m\in\NN$ and, for each $r\in\{1,\ldots,m\}$, suppose $n_r\in\NN$ and define
\[
    \nallgeneral = \sum_{r=1}^m n_r.
\]
For suitable coefficient functions $\M{c}{r}{j}$, consider the formal differential operators
\[
    \mathcal{L}_r := \left( \frac{\D}{\D x} \right)^{n_r} + \sum_{j=0}^{n_r-2} \M{c}{r}{j}(x) \left( \frac{\D}{\D x} \right)^j, \qquad r \in \{1,\ldots,m\},
\]
in which we have assumed $\M{c}{r}{n_r}=1$ and $\M{c}{r}{n_r-1}=0$, and the vector formal differential operator
\[
    \mathcal{L} := (\mathcal{L}_1,\mathcal{L}_2,\ldots,\mathcal{L}_m) \circ,
\]
where $\circ$ represents entrywise action of operators.

We define $\Phi= \prod^{m}_{r=1} \AC^{n_r-1}[0, 1]$, a product of function spaces $\AC^{n_r-1}[0, 1]$.
The inner product on this space is defined as the sum of the ordinary unweighted sesquilinear $\mathrm{L}^2$ inner products on the constituent spaces,
\[
    \langle \phi,\psi \rangle := \sum^m_{r=1} \int_0^1 \phi_r(x) \overline{\psi_r(x)} \D x.
\]

Let $\bccount \in \ZZ$ such that $0 \leq \bccount \leq 2\nallgeneral$ denote the number of boundary conditions.
Suppose that matrices of complex \emph{boundary coefficients} $b^r,\beta^r\in\CC^{\bccount \times n_r}$, $r\in\{1, \ldots, m\}$ are chosen such that the concatenated matrix $(b^1:\beta^1:\ldots:b^m:\beta^m)\in\CC^{\bccount\times2\nallgeneral}$ has full rank.
Define \emph{boundary forms} $B_k:\Phi\to\CC$ by
\begin{equation} \label{eqn:AdjointOperator.BoundaryCondition}
    B_k(\phi) = \sum_{r=1}^m \sum_{j=1}^{n_r} \left( \Msup{b}{k}{j}{r}\phi_r^{(j-1)}(0) +  \Msup{\beta}{k}{j}{r}\phi_r^{(j-1)}(1) \right), \qquad k\in\left\{1,2,\ldots,\bccount\right\},
\end{equation}
and let $\BVec{B} = (B_1, B_2, \ldots, B_{\bccount})$ be the vector of boundary forms.
Provided that, for each $r\in\{1,\ldots,m\}$ and $j\in\{0,1,\ldots,n_r-2\}$, $\M{c}{r}{j}\in\AC^{n_r-j-1}[0,1]$, on the space
\[
    \Phi_{\BVec{B}}=\{\phi\in\Phi: \BVec{B}\phi=\BVec{0}\},
\]
we define the differential operator $L:\Phi_\BVec{B}\to \prod_{r=1}^m \mathrm{L}^1[0,1]$ by $L\phi=\mathcal{L}\phi$.

\subsubsection*{Adjoint problem}

For the operator $L$ defined above, we aim to construct the classical adjoint $L^\star:\Phi_{\BVec{B}^\star} \to \prod_{r=1}^m \mathrm{L}^1[0,1]$.
That is, we aim to find a formal differential operator $\mathcal{L}^\star$ and adjoint vector boundary form $\BVec{B}^\star$ such that, for all $\phi\in\Phi_{\BVec{B}}$ and all $\psi\in\Phi_{\BVec{B}^\star}$,
\[
    \langle \mathcal{L}\phi , \psi \rangle = \langle \phi , \mathcal{L}^\star\psi \rangle.
\]
By~\cite[theorem~3.6.3]{CL1955a} applied entrywise, the formal adjoint is the operator
$\mathcal{L}^\star = (\mathcal{L}_1^\star,\mathcal{L}_2^\star,\ldots,\mathcal{L}_m^\star) \circ$,
in which
\begin{align} \label{eqn:formaladjoint}
    \mathcal{L}_{r}^{\star}\psi_{r} := (-1)^{n_r} \frac{\D^{n_r}\psi_{r}}{\D x^{n_r}} + \sum^{n_r-2}_{j=0} (-1)^j \frac{\D^j}{\D x^j}\left(\overline{\M{c}{r}{j}(x)} \psi_{r} \right).
\end{align}
It remains to determine an appropriate adjoint vector boundary form $\BVec{B}^\star$.

\begin{rmk}
    The only application of the results of this appendix in the present work is for the case $n_1=n_2=\ldots=n_m=:n$, all $\M{c}{r}{j}$ constant, and $\bccount=\nallgeneral=\nall$, and therefore the reader may restrict themself to this case in what follows, if they desire.
\end{rmk}

\subsection{Green's formula}\label{ssec:AdjointOperator.Greens-formula}
Following~\cite[theorem~3.6.3]{CL1955a}, for $\phi,\psi\in\Phi$, we define $[\phi\psi]_r$ to be the form in $(\phi_r,\phi'_r,\ldots,\phi^{n_r-1}_r)$ and $(\psi_r,\psi'_r,\ldots,\psi^{n_r-1}_r)$ given by
\[
    [\phi\psi]_r = \sum_{\ell=1}^{n_r}\sum_{\substack{j+k=\ell-1\\j,k\geq0}} (-1)^j \phi_{r}^{(k)} \left( \M{c}{r}{n_r-\ell}\overline{\psi_{r}} \right)^{(j)}.
\]
Application of Green's formula~\cite[corollary to theorem~3.6.3]{CL1955a} yields
\begin{align*}
    \langle \mathcal{L}\phi , \psi \rangle - \langle \phi , \mathcal{L}^\star\psi \rangle
    &= \sum_{r=1}^{m} \left( [\phi\psi]_r (1) - [\phi\psi]_r (0) \right) \\
    &= \sum_{r=1}^{m}\sum_{j,k=1}^{n_r} \left(\Msup{F}{j}{k}{r}(1) \phi_r^{(k-1)}(1) \psi_r^{(j-1)}(1) - \Msup{F}{j}{k}{r}(0) \phi_r^{(k-1)}(0) \psi_r^{(j-1)}(0) \right),
\end{align*}
where $F^r(x)$ denotes an $n_r \times n_r$ matrix at the point $x \in [0,1]$.
Following~\cite[\S11.1]{CL1955a}, the entries of $F^r(x)$, for $j,k \in \{1, \ldots, n_r\}$, are given by
\begin{equation} \label{eqn:AdjointOperator.F}
    F_{j\, k}^{r}(x) = \begin{cases}
        \sum^{n_r-k}_{\ell = j-1} (-1)^\ell \binom{\ell}{j-1} \left( \frac{\D}{\D x}\right)^{\ell- j +1} \M{c}{r}{\ell + k}(x) & j + k< n_r+ 1, \\
        (-1)^{j-1} \M{c}{r}{n_r}(x) = (-1)^{j-1} & j+k = n_r + 1, \\
        0 & j+k > n_r + 1.
    \end{cases}
\end{equation}
Observe that since $\mathrm{det}\left(F^r(x)\right) = 1$, the matrix $F^r$ is non-singular for all $r \in \{1, \ldots, m\}$.

Our goal is to rewrite the right side of Green's formula as a sesquilinear form $\mathcal{S}$.
First, let
\[
    \vec{\phi_r} := \left(\phi_r, \phi'_r, \ldots, \phi_r^{(n_r - 1)} \right)^\transpose,
\]
and observe that
\begin{align*}
    [\phi\psi]_r(x) = \sum_{k,j=1}^{n_r}\Msup{F}{j}{k}{r} (x) \phi_r^{(k-1)} (x) \psi_r^{(j-1)} (x) &= \sum_{j=1}^{n}\left[ \left(\sum_{k=1}^{n} \Msup{F}{j}{k}{r} \phi_r^{(k-1)}\right) \psi_r^{(j-1)}\right](x) \\
    &= F^r (x) \vec{\phi_r}(x) \odot \vec{\psi_r}(x),
\end{align*}
where $\odot$ refers to the sesquilinear dot product on $\CC^{n_r}$.
Green's formula can then be rewritten as
\begin{align}
\notag
    \langle \mathcal{L}\phi , \psi \rangle - \langle \phi , \mathcal{L}^\star\psi \rangle
    &= \sum_{r=1}^{m} [\phi\psi]_r(1) - [\phi\psi]_r(0) \\
\notag
    &= \sum_{r=1}^{m} F^r (1) \vec{\phi_r} (1) \odot \vec{\psi_r} (1) - F^r (0) \vec{\phi_r} (0) \odot \vec{\psi_r} (0) \\
\label{eqn:AdjointOperator.GF}
    &= \sum_{r=1}^{m}
        \begin{bmatrix}
            - F^r(0) & 0_{n_r \times n_r} \\
            0_{n_r \times n_r} &  F^r(1) \\
        \end{bmatrix}
        \begin{bmatrix}
            \vec{\phi_r}(0)  \\
            \vec{\phi_r}(1)  \\
        \end{bmatrix}
        \odot
        \begin{bmatrix}
            \vec{\psi_r}(0)  \\
            \vec{\psi_r}(1)  \\
        \end{bmatrix}.
\end{align}
Expansion of the sum yields
\begin{align}
\notag
    &\hspace{-5em} \langle \mathcal{L}\phi , \psi \rangle - \langle \phi , \mathcal{L}^\star\psi \rangle \\
\notag
    &=
    \begin{bmatrix}
        - F^1(0) & 0_{n_1\times n_1} \\
        0_{n_1\times n_1} &  F^1(1) \\
    \end{bmatrix}
    \begin{bmatrix}
        \vec{\phi_1}(0)  \\
        \vec{\phi_1}(1)  \\
    \end{bmatrix}
    \odot
    \begin{bmatrix}
        \vec{\psi_1}(0)  \\
        \vec{\psi_1}(1)  \\
    \end{bmatrix} \\
\notag
    &\hspace{9em} + \ldots +
    \begin{bmatrix}
        - F^m(0) & 0_{n_m\times n_m} \\
        0_{n_m\times n_m} & F^m(1) \\
    \end{bmatrix}
    \begin{bmatrix}
        \vec{\phi_m}(0)  \\
        \vec{\phi_m}(1)  \\
    \end{bmatrix}
    \odot
    \begin{bmatrix}
        \vec{\psi_m}(0)  \\
        \vec{\psi_m}(1)  \\
    \end{bmatrix} \\
\notag
    &= \underbrace{
    \begin{bmatrix}
        - F^1(0) & 0 & \cdots & 0 & 0 \\
        0 & F^1(1) & \cdots & 0 & 0 \\
        \vdots & \vdots & \ddots & \vdots & \vdots \\
        0 & 0 & \cdots & -F^m(0) & 0 \\
        0 & 0 & \cdots & 0 &  F^m(1)
    \end{bmatrix}}_\text{$2\nallgeneral \times 2\nallgeneral$}
    \begin{bmatrix}
        \vec{\phi_1}(0)  \\
        \vec{\phi_1}(1) \\
        \vec{\phi_2}(0)  \\
        \vec{\phi_2}(1) \\
        \vdots \\
        \vec{\phi_m}(0)  \\
        \vec{\phi_m}(1)
    \end{bmatrix}
    \odot
    \begin{bmatrix}
        \vec{\psi_1}(0)  \\
        \vec{\psi_1}(1) \\
        \vec{\psi_2}(0)  \\
        \vec{\psi_2}(1) \\
        \vdots \\
        \vec{\psi_m}(0)  \\
        \vec{\psi_m}(1)
    \end{bmatrix} \\
\label{eqn:AdjointOperator.GR-Sform}
    &=: S
    \begin{bmatrix}
        \vec{\phi_1}(0)  \\
        \vec{\phi_1}(1) \\
        \vdots \\
        \vec{\phi_m}(0)  \\
        \vec{\phi_m}(1)
    \end{bmatrix}
    \odot
    \begin{bmatrix}
        \vec{\psi_1}(0)  \\
        \vec{\psi_1}(1) \\
        \vdots \\
        \vec{\psi_m}(0)  \\
        \vec{\psi_m}(1)
    \end{bmatrix}
    =: \mathcal{S}
    \left( \begin{bmatrix}
        \vec{\phi_1}(0)  \\
        \vec{\phi_1}(1) \\
        \vdots \\
        \vec{\phi_m}(0)  \\
        \vec{\phi_m}(1)
    \end{bmatrix},
    \begin{bmatrix}
        \vec{\psi_1}(0)  \\
        \vec{\psi_1}(1) \\
        \vdots \\
        \vec{\psi_m}(0)  \\
        \vec{\psi_m}(1)
    \end{bmatrix}
    \right),
\end{align}
where the matrix $S$ is associated with the sesquilinear form $\mathcal{S}$, and $S$ is a block diagonal matrix whose diagonal blocks are $n_r \times n_r$.
This may be compared with~\cite[equation~(11.1.3)]{CL1955a}.

\subsection{Interface boundary form formula} \label{ssec:AdjointOperator.BFF}
We turn to characterising adjoint interface boundary conditions, by extending the boundary form formula for two point problems, as given in~\cite[theorem~11.2.1]{CL1955a}.

Equation~\eqref{eqn:AdjointOperator.BoundaryCondition} can be expressed as
\begin{equation}
\label{eqn:AdjointOperator.BoundryForm1}
    \BVec{B}\phi
    = \sum^m_{r=1} \sum^{n_r-1}_{j=0} \begin{bmatrix} \Msup{b}{1}{j}{r} \\ \vdots\\ \Msup{b}{\bccount}{j}{r}\end{bmatrix} \phi_r^{(j)}(0) + \begin{bmatrix} \Msup{\beta}{1}{j}{r}\\ \vdots \\ \Msup{\beta}{\bccount}{j}{r}\end{bmatrix} \phi_r^{(j)}(1) \\
    = \sum^m_{r=1} b^r \vec{\phi_r}(0) + \beta^r \vec{\phi_r}(1).
\end{equation}
Using $(b^r:\beta^r)$ to represent concatenation of matrices, we can write
\begin{equation} \label{eqn:AdjointOperator.BoundryForm2}
    \BVec{B}\phi = \sum^m_{r=1}\begin{bmatrix} b^r : \beta^r \end{bmatrix}
    \begin{bmatrix}
        \vec{\phi_r}(0)  \\
        \vec{\phi_r}(1)
    \end{bmatrix} =
    \underbrace{\begin{bmatrix} b^1 : \beta^1 :{} \ldots {}: b^m : \beta^m \end{bmatrix}}_\text{$\bccount \times 2\nallgeneral$}
    \underbrace{\begin{bmatrix}
        \vec{\phi_1}(0)  \\
        \vec{\phi_1}(1) \\
        \vdots \\
        \vec{\phi_m}(0)  \\
        \vec{\phi_m}(1)
    \end{bmatrix}}_\text{$2\nallgeneral \times 1$}.
\end{equation}
Thus we have two compact ways to write vectors of boundary forms, namely equations~\eqref{eqn:AdjointOperator.BoundryForm1} and~\eqref{eqn:AdjointOperator.BoundryForm2}.
Next, we extend the notion of complementary vectors of boundary forms from that in~\cite[\S11.2]{CL1955a} to the present setting.
\begin{defn} \label{defn:AdjointOperator.ComplimentaryBoundaryForm}
    Let $\bccount \in\ZZ$, with $0 \leq \bccount \leq 2\nallgeneral$.
    If $\BVec{B} = (B_1, \ldots, B_{\bccount})$ is any vector of boundary forms with $\mathrm{rank}(\BVec{B}) = \bccount$, and $\BVec{B}_\complementary = (B_{\bccount+1}, \ldots, B_{2\nallgeneral})$ is a vector of forms with $\mathrm{rank}(\BVec{B}_\complementary) = 2\nallgeneral-\bccount$ such that $\mathrm{rank}(B_{1}, \ldots, B_{2\nallgeneral}) = 2\nallgeneral$, then $\BVec{B}$ and $\BVec{B}_\complementary$ are said to be \emph{complementary vectors of boundary forms}.
\end{defn}
Note that extending $(B_1, \ldots, B_\bccount)$ to $(B_{1}, \ldots, B_{2\nallgeneral})$ is equivalent to embedding the matrices $b^r, \beta^r$ in a $2\nallgeneral \times 2\nallgeneral$ non-singular matrix.
That is
\begin{align}
\notag
    \begin{bmatrix}
        \BVec{B}\phi \\
        \BVec{B}_\complementary \phi
    \end{bmatrix}
    &=
    \sum^m_{r=1}
    \begin{bmatrix}
        b^r & \beta^r \\
        b^r_\complementary & \beta^r_\complementary
    \end{bmatrix}
    \begin{bmatrix}
        \vec{\phi_r}(0)  \\
        \vec{\phi_r}(1)
    \end{bmatrix} \\
\label{eqn:AdjointOperator.BFF-H}
    &=
    \underbrace{
    \begin{bmatrix}
        b^1 & \beta^1 & b^2 & \beta^2 & \cdots & b^m & \beta^m \\
        b^1_\complementary & \beta^1_\complementary & b^2_\complementary & \beta^2_\complementary & \cdots & b^m_\complementary & \beta^m_\complementary
    \end{bmatrix} }_\text{$2\nallgeneral \times 2\nallgeneral$}
    \underbrace{\begin{bmatrix}
        \vec{\phi_1}(0) \\
        \vec{\phi_1}(1) \\
        \vec{\phi_2}(0) \\
        \vec{\phi_2}(1) \\
        \vdots \\
        \vec{\phi_m}(0) \\
        \vec{\phi_m}(1)
    \end{bmatrix}}_\text{$2\nallgeneral\times 1$}
    =:
    H
    \begin{bmatrix}
        \vec{\phi_1}(0) \\
        \vec{\phi_1}(1) \\
        \vec{\phi_2}(0) \\
        \vec{\phi_2}(1) \\
        \vdots \\
        \vec{\phi_m}(0) \\
        \vec{\phi_m}(1)
    \end{bmatrix}.
\end{align}
where $\mathrm{rank}(H) = 2\nallgeneral$ and $b^r_\complementary, \beta^r_\complementary\in\CC^{(2\nallgeneral-\bccount)\times n_r}$.
We will use equation~\eqref{eqn:AdjointOperator.BFF-H} to express Green's formula as a combination of vector boundary forms $\BVec{B}$ and $\BVec{B}_\complementary$ in theorem~\ref{thm:AdjointOperator.BFF}.
Its proof also requires the following lemma, whose proof is immediate from the fact that, in the sesquilinear dot product, the conjugate transpose of a matrix is the adjoint of the matrix.

\begin{lem} \label{lem:FactoringSesquilinearForms}
    Let $\sigma$ be the sesquilinear form associated with a nonsingular matrix $\Sigma$;
    \[
        \sigma(f,g) = \Sigma f \odot g.
    \]
    For each nonsingular matrix $F$, there exists a unique nonsingular matrix $G$ such that $\sigma(f,g) = Ff \odot Gg$ for all $f,g$.
    Moreover, $G=(\Sigma F^{-1})^\conjtrans$, in which $^\conjtrans$ represents the (Hermitian) conjugate transpose.
\end{lem}

\begin{thm}[Interface boundary form formula] \label{thm:AdjointOperator.BFF}
    Given any vector of forms $\BVec{B}$ of rank $\bccount$, and any complementary vector of forms $\BVec{B}_\complementary$, there exist unique vectors of forms $\BVec{B}^{\star}_\complementary$ and $\BVec{B}^{\star}$ of rank $\bccount$ and $2\nallgeneral-\bccount$, respectively, such that
    \begin{equation} \label{eqn:AdjointOperator.BFF}
        \sum_{r=1}^{m} [\phi\psi]_r(1) - [\phi\psi]_r(0) = \BVec{B}\phi \odot \BVec{B}^{\star}_\complementary\psi + \BVec{B}_\complementary\phi \odot \BVec{B}^{\star}\psi.
    \end{equation}
    Moreover, the complementary adjoint boundary coefficient matrices $\Msups{b}{\complementary}{}{r}{\star},\Msups{\beta}{\complementary}{}{r}{\star}\in\CC^{\bccount\times n_r}$ and the adjoint boundary coefficient matrices $\Msups{b}{}{}{r}{\star},\Msups{\beta}{}{}{r}{\star}\in\CC^{(2\nallgeneral-\bccount)\times n_r}$ are given by
    \begin{equation} \label{eqn:AdjointOperator.BFF.BoundaryCoeffMatrices}
        \begin{bmatrix}
            \Msups{b}{\complementary}{}{1}{\star} & \Msups{\beta}{\complementary}{}{1}{\star} & \Msups{b}{\complementary}{}{2}{\star} & \Msups{\beta}{\complementary}{}{2}{\star} & \cdots & \Msups{b}{\complementary}{}{m}{\star} & \Msups{\beta}{\complementary}{}{m}{\star} \\
            \Msups{b}{}{}{1}{\star} & \Msups{\beta}{}{}{1}{\star} & \Msups{b}{}{}{2}{\star} & \Msups{\beta}{}{}{2}{\star} & \cdots & \Msups{b}{}{}{m}{\star} & \Msups{\beta}{}{}{m}{\star}
        \end{bmatrix}
        = \left(SH^{-1}\right)^\conjtrans,
    \end{equation}
    for $H$ as defined in equation~\eqref{eqn:AdjointOperator.BFF-H} and $S$ as defined in equation~\eqref{eqn:AdjointOperator.GR-Sform}.
\end{thm}

\begin{proof}
    Let $H$ be as defined by equation~\eqref{eqn:AdjointOperator.BFF-H}.
    By lemma~\ref{lem:FactoringSesquilinearForms}, there exists a unique $2\nallgeneral \times 2\nallgeneral $ nonsingular matrix $J$ such that
    \begin{equation} \label{eqn:AdjointOperator.BFF.proof.1}
        \mathcal{S} \left( \begin{bmatrix}
            \vec{\phi_1}(0)  \\
            \vec{\phi_1}(1) \\
            \vdots \\
            \vec{\phi_m}(0)  \\
            \vec{\phi_m}(1)
        \end{bmatrix},
        \begin{bmatrix}
            \vec{\psi_1}(0)  \\
            \vec{\psi_1}(1) \\
            \vdots \\
            \vec{\psi_m}(0)  \\
            \vec{\psi_m}(1)
        \end{bmatrix}
        \right)
        = H
        \begin{bmatrix}
            \vec{\phi_1}(0)  \\
            \vec{\phi_1}(1) \\
            \vdots \\
            \vec{\phi_m}(0)  \\
            \vec{\phi_m}(1)
        \end{bmatrix} \odot
        J\begin{bmatrix}
            \vec{\psi_1}(0)  \\
            \vec{\psi_1}(1) \\
            \vdots \\
            \vec{\psi_m}(0)  \\
            \vec{\psi_m}(1)
        \end{bmatrix}.
    \end{equation}
    Moreover, $J=\left(SH^{-1}\right)^\conjtrans$.
    We define $\BVec{B}^{\star}, \BVec{B}^{\star}_\complementary$ by
    \begin{equation} \label{eqn:AdjointOperator.BFF.proof.2}
        \begin{bmatrix}
            \BVec{B}^{\star}_\complementary \psi \\
            \BVec{B}^{\star} \psi
        \end{bmatrix}
        = J\begin{bmatrix}
            \vec{\psi_1}(0)  \\
            \vec{\psi_1}(1) \\
            \vdots \\
            \vec{\psi_m}(0)  \\
            \vec{\psi_m}(1)
        \end{bmatrix},
    \end{equation}
    from which equation~\eqref{eqn:AdjointOperator.BFF.BoundaryCoeffMatrices} follows by the same argument as that used to derive equation~\eqref{eqn:AdjointOperator.BFF-H}.
    Equations~\eqref{eqn:AdjointOperator.GR-Sform},~\eqref{eqn:AdjointOperator.BFF.proof.1},~\eqref{eqn:AdjointOperator.BFF-H}, and~\eqref{eqn:AdjointOperator.BFF.proof.2} yield
    \begin{equation*}
        \sum_{r=1}^{m} [\phi\psi]_r(1) - [\phi\psi]_r(0)
        = \begin{bmatrix}
            \BVec{B}\phi \\
            \BVec{B}_\complementary \phi
        \end{bmatrix} \odot
        \begin{bmatrix}
            \BVec{B}^{\star}_\complementary \psi \\
            \BVec{B}^{\star} \psi
        \end{bmatrix}
        = \BVec{B}\phi \odot \BVec{B}^{\star}_\complementary\psi + \BVec{B}_\complementary\phi \odot \BVec{B}^{\star}\psi.
    \end{equation*}
    By unicity of $J$, no other definition of $\BVec{B}^{\star}_\complementary$ and $\BVec{B}^{\star}$ satisfies equation~\eqref{eqn:AdjointOperator.BFF}.
\end{proof}

The interface boundary form formula theorem~\ref{thm:AdjointOperator.BFF} allows us to define adjoint boundary conditions, whence we get the adjoint boundary value problem.

\begin{defn} \label{defn:AdjointOperator.AdjointBoundaryForm}
    Suppose $\BVec{B} = (B_1, \ldots, B_\bccount)$ is a vector of forms with $\mathrm{rank}(\BVec{B}) = \bccount$.
    Suppose $\BVec{B}^{\star}$ is any vector of forms with $\mathrm{rank}(\BVec{B}^{\star}) = 2\nallgeneral-\bccount$, determined as in theorem~\ref{thm:AdjointOperator.BFF}.
    Then $\BVec{B}^{\star}$ is an \emph{adjoint boundary form} and the equation $\BVec{B}^{\star}\psi = \BVec{0}$ is an \emph{adjoint boundary condition} to boundary condition $\BVec{B}\phi = \BVec{0}$ in function space $\Phi$.
\end{defn}

\begin{defn} \label{defn:AdjointOperator.BVP-rank-l}
    Suppose $\BVec{B} = (B_1, \ldots, B_\bccount)$ is a vector of forms with $\mathrm{rank}(\BVec{B}) = \bccount$ and $\BVec{B}^\star$ is an adjoint boundary form.
    Then the problem of solving
    \[
        L\phi = 0, \qquad \phi\in\Phi_{\BVec{B}}
    \]
    is called a \emph{boundary value problem} of rank $\bccount$.
    The problem of solving
    \[
        L^\star\psi = 0, \qquad \psi\in\Phi_{\BVec{B}^\star}
    \]
    is the \emph{adjoint boundary value problem}.
\end{defn}

In theorem~\ref{thm:AdjointOperator.BFF}, the adjoint boundary conditions $\BVec{B}^\star$ are determined uniquely not by $\BVec{B}$ but by the pair $(\BVec{B},\BVec{B}_\complementary)$.
However, the function space $\Phi_{\BVec{B}^\star}$ is determined uniquely from $\BVec{B}$.
Indeed, because $\BVec{B}_{\complementary}$ is (or, more precisely, the complementary boundary coefficient matrices $b_\complementary^r, \beta_\complementary^r$ from which it is composed are) determined by $\BVec{B}$ uniquely up to a rank $(2\nallgeneral-\bccount)\times(2\nallgeneral-\bccount)$ perturbation, so also is $\BVec{B}^\star$ (or the boundary coefficient matrices from which it is composed).

Finally, we justify that the term ``adjoint boundary condition'' is faithfully applied in the classical (Lagrange) sense.

\begin{cor}
    If $\phi\in\Phi_{\BVec{B}}$ and $\psi\in\Phi_{\BVec{B}^\star}$, then $\langle L\phi, \psi \rangle = \langle \phi, L^{\star}\psi \rangle$.
\end{cor}

\begin{proof}
    We apply the boundary form formula theorem~\ref{thm:AdjointOperator.BFF} and the definitions of $\Phi_{\BVec{B}}$ and $\Phi_{\BVec{B}^\star}$ to obtain
    \[
        \langle L\phi, \psi \rangle - \langle \phi, L^{\star}\psi \rangle = \BVec{B}\phi \odot \BVec{B}^{\star}_\complementary\psi + \BVec{B}_\complementary\phi \odot \BVec{B}^{\star}\psi = \BVec{0} \odot \BVec{B}^{\star}_\complementary\psi + \BVec{B}_\complementary\phi \odot \BVec{0} = 0.
    \qedhere
    \]
\end{proof}

\subsection{Checking Adjointness} \label{sec:AdjointOperator.CA}

\begin{thm} \label{thm:AdjointOperator.CA}
    Suppose that vector boundary form $\BVec{B}$ has full rank boundary coefficient matrices $b^r,\beta^r\in\CC^{\bccount \times n_r}$, and vector boundary form $\BVec{Z}$ has full rank boundary coefficient matrices $z^r,\zeta^r\in\CC^{(2\nallgeneral-\bccount)\times n_r}$.
    The boundary condition $\BVec{Z}$ is adjoint to $\BVec{B}$ if and only if
    \begin{equation} \label{eqn:AdjointOperator.CA}
        \sum^m_{r=1} b^r (F^{r})^{-1}(0) \Msups{z}{}{}{r}{\conjtrans} = \sum^m_{r=1} \beta^r (F^{r})^{-1}(1) \Msups{\zeta}{}{}{r}{\conjtrans},
    \end{equation}
    where $F^r(x)$ is the $n_r\times n_r$ matrix defined by equation~\eqref{eqn:AdjointOperator.F}.
\end{thm}

\begin{proof}[Proof of theorem~\ref{thm:AdjointOperator.CA}: ``only if'']
    Suppose that $\BVec{B}$ and $\BVec{Z}$ are adjoint.
    By definition~\ref{defn:AdjointOperator.AdjointBoundaryForm}, $\BVec{Z}$ is determined as in theorem~\ref{thm:AdjointOperator.BFF}.
    Thus, in determining $\BVec{Z}$, there exist vectors of forms $\BVec{B}_\complementary, \BVec{B}_\complementary^{\star}$ of rank $2\nallgeneral-\bccount$ and $\bccount$ respectively, such that theorem~\ref{thm:AdjointOperator.BFF} holds.
    As such, there exist full rank matrices $b_\complementary^r, \beta_\complementary^r \in\CC^{(2\nallgeneral-\bccount)\times n_r}$, and $\Msups{b}{\complementary}{}{r}{\star},\Msups{\beta}{\complementary}{}{r}{\star}\in\CC^{\bccount\times n_r}$ such that
    \begin{align}
    \label{eqn:AdjointOperator.CA.Bc}
        \BVec{B}_\complementary \phi &= \sum^m_{r=1} b_\complementary^r \vec{\phi}_r(0) + \beta_\complementary^r \vec{\phi}_r(1), &
        &\mathrm{rank}\begin{bmatrix} b_\complementary^1 : \beta_\complementary^1 : {}\ldots{} : b_\complementary^m : \beta_\complementary^m \end{bmatrix} =  2\nallgeneral - \bccount, \\
    \label{eqn:AdjointOperator.CA.Bc*}
        \BVec{B}_\complementary^{\star} \psi &= \sum^m_{r=1} \Msups{b}{\complementary}{}{r}{\star} \vec{\psi}_r(0) + \Msups{\beta}{\complementary}{}{r}{\star}  \vec{\psi}_r(1), &
        &\mathrm{rank}\begin{bmatrix} \Msups{b}{\complementary}{}{1}{\star} : \Msups{\beta}{\complementary}{}{1}{\star} : {}\ldots{} : \Msups{b}{\complementary}{}{m}{\star} : \Msups{\beta}{\complementary}{}{m}{\star} \end{bmatrix} = \bccount.
    \end{align}
    Moreover, it is reasonable to denote $\BVec{Z}$ by $\BVec{B}^\star$ and $z^r,\zeta^r$ by $\Msups{b}{}{}{r}{\star},\Msups{\beta}{}{}{r}{\star}$, respectively.

    By equations~\eqref{eqn:AdjointOperator.BFF},~\eqref{eqn:AdjointOperator.BFF-H},~\eqref{eqn:AdjointOperator.CA.Bc*},~\eqref{eqn:AdjointOperator.CA.Bc},
    and~\eqref{eqn:AdjointOperator.BFF.BoundaryCoeffMatrices}, then distributing the inner products over the sums,
    \begin{multline*}
        \sum^m_{r=1} [\phi\psi]_r(1) - [\phi\psi]_r(0)
        = \sum^m_{r=1}  \sum^m_{i=1} \Bigg(\left( b^r \vec{\phi_r}(0) + \beta^r \vec{\phi_r}(1)\right) \odot \Big( \Msups{b}{\complementary}{}{i}{\star} \vec{\psi}_i(0) + \Msups{\beta}{\complementary}{}{i}{\star} \vec{\psi}_i(1) \Big) \\
        + \left( b_\complementary^r \vec{\phi}_r(0) + \beta_\complementary^r \vec{\phi}_r(1) \right) \odot \Big( \Msups{b}{}{}{i}{\star} \vec{\psi}_i(0) + \Msups{\beta}{}{}{i}{\star} \vec{\psi}_i(1) \Big) \Bigg).
    \end{multline*}
    By additivity of inner product and the fact that the conjugate transpose is the adjoint with respect to the sesquilinear dot product, we write the above as
    \begin{multline} \label{eqn:AdjointOperator.CA.eq1}
        \sum^m_{r=1} \sum^m_{i=1}
        \left(\Msups{\beta}{\complementary}{}{i}{\star\Mspacer\conjtrans} \beta^r + \Msups{\beta}{}{}{i}{\star\Mspacer\conjtrans} \beta_\complementary^r\right) \vec{\phi_r}(1) \odot \vec{\psi_i}(1)
        + \left(\Msups{b}{\complementary}{}{i}{\star\Mspacer\conjtrans} \beta^r + \Msups{b}{}{}{i}{\star\Mspacer\conjtrans} \beta_\complementary^r \right) \vec{\phi_r}(1) \odot \vec{\psi_i}(0) \\
        + \left(\Msups{\beta}{\complementary}{}{i}{\star\Mspacer\conjtrans} b^r + \Msups{\beta}{}{}{i}{\star\Mspacer\conjtrans} b_\complementary^r\right) \vec{\phi_r}(0) \odot \vec{\psi_i}(1)
        + \left(\Msups{b}{\complementary}{}{i}{\star\Mspacer\conjtrans} b^r + \Msups{b}{}{}{i}{\star\Mspacer\conjtrans} b_\complementary^r\right) \vec{\phi_r}(0) \odot \vec{\psi_i}(0).
    \end{multline}
    From Green's formula~\eqref{eqn:AdjointOperator.GF}, we have
    \begin{equation} \label{eqn:AdjointOperator.CA.eq2}
        \sum^m_{r=1}  [\phi\psi]_r(1) - [\phi\psi]_r(0) = \sum_{r=1}^{m} F^r(1) \vec{\phi_r}(1) \odot \vec{\psi_r}(1) - F^r(0) \vec{\phi_r}(0) \odot \vec{\psi_r}(0),
    \end{equation}
    Equating coefficients of each of the inner products in equations~\eqref{eqn:AdjointOperator.CA.eq1} and~\eqref{eqn:AdjointOperator.CA.eq2} reveals that
    \begin{equation} \label{eqn:AdjointOperator.CA.EquateCoeffs}
        \begin{aligned}
            \left(\Msups{\beta}{\complementary}{}{i}{\star\Mspacer\conjtrans} \beta^r + \Msups{\beta}{}{}{i}{\star\Mspacer\conjtrans} \beta_\complementary^r\right) &= \bigg\{\hspace{-0.5em} \begin{array}{l} F^r(1), \\ 0, \end{array}
            &
            \left(\Msups{b}{\complementary}{}{i}{\star\Mspacer\conjtrans} b^r + \Msups{b}{}{}{i}{\star\Mspacer\conjtrans} b_\complementary^r\right) &= \bigg\{\hspace{-0.5em} \begin{array}{l} -F^r(0), \\ 0, \end{array}
            &
            &\begin{array}{l} \mbox{if }i=r, \\ \mbox{otherwise,}\end{array}
            \\
            \left(\Msups{b}{\complementary}{}{i}{\star\Mspacer\conjtrans} \beta^r + \Msups{b}{}{}{i}{\star\Mspacer\conjtrans} \beta_\complementary^r \right) &= \hphantom{\bigg\{\hspace{-0.5em}}\begin{array}{l}0,\end{array}
            &
            \left(\Msups{\beta}{\complementary}{}{i}{\star\Mspacer\conjtrans} b^r + \Msups{\beta}{}{}{i}{\star\Mspacer\conjtrans} b_\complementary^r\right) &= \hphantom{\bigg\{\hspace{-0.5em}}\begin{array}{l}0,\end{array}
            &
            &\begin{array}{l} \mbox{for all } i,r.\end{array}
        \end{aligned}
    \end{equation}
    Note that not all of the $0$ matrices above are the same, as each of the above equations relates matrices in $\CC^{n_i\times n_r}$; the nonzero cases are square matrices in $\CC^{n_r\times n_r}$.

    Interlacing the $i=r$ cases of the upper two of identities~\eqref{eqn:AdjointOperator.CA.EquateCoeffs} as blocks on matrix diagonals, and filling out the matrices with appropriate sized zero blocks elsewhere, we obtain
    \begin{multline}\label{eqn:AdjointOperator.CA.eq3}
        \begin{bmatrix}
            - F^1(0) & 0 & 0 & \cdots & 0 & 0 & 0 \\
            0 & F^1(1) & 0 & \cdots & 0 & 0 & 0 \\
            0 &  0 & -F^2(0) & \cdots & 0 & 0 & 0 \\
            \vdots & \vdots & \vdots & \ddots & \vdots & \vdots & \vdots  \\
            0 &  0 & 0 & \cdots & F^{m-1}(1) & 0 & 0 \\
            0 &  0 & 0 & \cdots & 0 & -F^m(0) & 0 \\
            0 &  0 & 0 & \cdots & 0 & 0 & F^m(1)
        \end{bmatrix} \\
        =
        \begin{bmatrix}
            \Msups{b}{\complementary}{}{1}{\star\Mspacer\conjtrans} b^1 + \Msups{b}{}{}{1}{\star\Mspacer\conjtrans} b_\complementary^1 & 0 & \cdots & 0 & 0 \\
            0 & \Msups{\beta}{\complementary}{}{1}{\star\Mspacer\conjtrans} \beta^1 + \Msups{\beta}{}{}{1}{\star\Mspacer\conjtrans} \beta_\complementary^1 & \cdots & 0 & 0 \\
            \vdots & \vdots & \ddots & \vdots & \vdots  \\
            0 &  0 & \cdots & \Msups{b}{\complementary}{}{m}{\star\Mspacer\conjtrans} b^m + \Msups{b}{}{}{m}{\star\Mspacer\conjtrans} b_\complementary^m & 0 \\
            0 &  0 & \cdots & 0 & \Msups{\beta}{\complementary}{}{m}{\star\Mspacer\conjtrans} \beta^m + \Msups{\beta}{}{}{m}{\star\Mspacer\conjtrans} \beta_\complementary^m
        \end{bmatrix}.
    \end{multline}
    Since the boundary matrices $F^r$ are each nonsingular, the block diagonal matrix on the left of equation~\eqref{eqn:AdjointOperator.CA.eq3} must also be invertible.
    Premultiplying on both sides by the inverse yields the following expression for the identity matrix
    \begin{multline} \label{eqn:AdjointOperator.CA.Product}
        \begin{bmatrix}
            - (F^{1})^{-1}(0)\left(\Msups{b}{\complementary}{}{1}{\star\Mspacer\conjtrans} b^1 + \Msups{b}{}{}{1}{\star\Mspacer\conjtrans} b_\complementary^1\right) & \cdots & 0 \\
            \vdots & \ddots & \vdots  \\
            0 & \cdots & (F^{m})^{-1}(1)\left(\Msups{\beta}{\complementary}{}{m}{\star\Mspacer\conjtrans} \beta^m + \Msups{\beta}{}{}{m}{\star\Mspacer\conjtrans} \beta_\complementary^m\right)
        \end{bmatrix} \\
        =
        \begin{bmatrix}
            - (F^{1})^{-1}(0) \Msups{b}{\complementary}{}{1}{\star\Mspacer\conjtrans} & - (F^{1})^{-1}(0) \Msups{b}{}{}{1}{\star\Mspacer\conjtrans} \\
            (F^{1})^{-1}(1) \Msups{\beta}{\complementary}{}{1}{\star\Mspacer\conjtrans} & (F^{1})^{-1}(1) \Msups{\beta}{}{}{1}{\star\Mspacer\conjtrans} \\
            \vdots & \vdots  \\
            -(F^{m})^{-1}(0) \Msups{b}{\complementary}{}{m}{\star\Mspacer\conjtrans} & -(F^{m})^{-1}(0) \Msups{b}{}{}{m}{\star\Mspacer\conjtrans} \\
            (F^{m})^{-1}(1) \Msups{\beta}{\complementary}{}{m}{\star\Mspacer\conjtrans} & (F^{m})^{-1}(1) \Msups{\beta}{}{}{m}{\star\Mspacer\conjtrans}
        \end{bmatrix}
        \begin{bmatrix}
            b^1 & \beta^1 & \cdots & b^m & \beta^m \\
            b_\complementary^1  & \beta_\complementary^1 & \cdots & b_\complementary^m  & \beta_\complementary^m
        \end{bmatrix};
    \end{multline}
    the equation is justified by using the rest of identities~\eqref{eqn:AdjointOperator.CA.EquateCoeffs} to show that all blocks off the diagonal are zero blocks of appropriate dimension.
    Since the two matrices in the product on the right of equation~\eqref{eqn:AdjointOperator.CA.Product} are square and full rank, they are inverse to each other, and so we have, for identity matrices $\mathrm{Id}$,
    \begin{multline} \label{eqn:AdjointOperator.CA.Product2}
        \begin{bmatrix}
            \mathrm{Id}_{\bccount \times \bccount} & 0_{\bccount \times (2\nallgeneral-\bccount)} \\
            0_{(2\nallgeneral-\bccount) \times \bccount} & \mathrm{Id}_{(2\nallgeneral-\bccount) \times (2\nallgeneral-\bccount)}
        \end{bmatrix}
        \\
        =
        \begin{bmatrix}
            b^1 & \beta^1 & \cdots & b^m & \beta^m \\
            b_\complementary^1  & \beta_\complementary^1 & \cdots & b_\complementary^m  & \beta_\complementary^m
        \end{bmatrix}
        \begin{bmatrix}
            - (F^{1})^{-1}(0) \Msups{b}{\complementary}{}{1}{\star\Mspacer\conjtrans} & - (F^{1})^{-1}(0) \Msups{b}{}{}{1}{\star\Mspacer\conjtrans} \\
            (F^{1})^{-1}(1) \Msups{\beta}{\complementary}{}{1}{\star\Mspacer\conjtrans} & (F^{1})^{-1}(1) \Msups{\beta}{}{}{1}{\star\Mspacer\conjtrans} \\
            \vdots & \vdots  \\
            -(F^{m})^{-1}(0) \Msups{b}{\complementary}{}{m}{\star\Mspacer\conjtrans} & -(F^{m})^{-1}(0) \Msups{b}{}{}{m}{\star\Mspacer\conjtrans} \\
            (F^{m})^{-1}(1) \Msups{\beta}{\complementary}{}{m}{\star\Mspacer\conjtrans} & (F^{m})^{-1}(1) \Msups{\beta}{}{}{m}{\star\Mspacer\conjtrans}
        \end{bmatrix}.
    \end{multline}
    The top right block (that is, the block which contains neither complementary boundary coefficient matrices nor adjoint complementary boundary coefficient matrices) yields
    \begin{multline*}
        - b^1 (F^{1})^{-1}(0) \Msups{b}{}{}{1}{\star\Mspacer\conjtrans}
        + \beta^1 (F^{1})^{-1}(1) \Msups{\beta}{}{}{1}{\star\Mspacer\conjtrans}
        + \ldots \\
        - b^m (F^{m})^{-1}(0) \Msups{b}{}{}{m}{\star\Mspacer\conjtrans}
        + \beta^m (F^{m})^{-1}(1) \Msups{\beta}{}{}{m}{\star\Mspacer\conjtrans}
        = 0_{\bccount \times (2\nallgeneral-\bccount)},
    \end{multline*}
    from which it follows
    \begin{equation*}
        \sum^m_{r=1} b^r (F^{r})^{-1}(0) \Msups{b}{}{}{r}{\star\Mspacer\conjtrans} = \sum^m_{r=1} \beta^r (F^{r})^{-1}(1) \Msups{\beta}{}{}{r}{\star\Mspacer\conjtrans}.
        \qedhere
    \end{equation*}
\end{proof}

\begin{proof}[Proof of theorem~\ref{thm:AdjointOperator.CA}: ``if'']
    Suppose that $\BVec{Z}$ is a vector of boundary forms with boundary coefficient matrices $z^r,\zeta^r$
    \[
        \BVec{Z} \psi = \sum^m_{r=1} z^r \vec{\psi}_r(0) + \zeta^r \vec{\psi}_r(1),
    \]
    and
    \[
        \mathrm{rank}\begin{bmatrix} z^1 : \zeta^1 : {}\ldots{} : z^m : \zeta^m \end{bmatrix} = 2\nallgeneral-\bccount,
    \]
    for which equation~\eqref{eqn:AdjointOperator.CA} holds.
    Then
    \[
        \mathrm{rank}
        \begin{bmatrix}
            \Msups{z}{}{}{1}{\conjtrans} \\
            \Msups{\zeta}{}{}{1}{\conjtrans} \\
            \vdots \\
            \Msups{z}{}{}{m}{\conjtrans} \\
            \Msups{\zeta}{}{}{m}{\conjtrans}
        \end{bmatrix}
        = 2\nallgeneral-\bccount
    \]
    and equation~\eqref{eqn:AdjointOperator.CA} implies
    \begin{equation}\label{eqn:AdjointOperator.system1}
        \begin{bmatrix}
            b^1 : \beta^1 : {}\ldots{} : b^m : \beta^m
        \end{bmatrix}
        \begin{bmatrix}
            -F^1(0)^{-1} \Msups{z}{}{}{1}{\conjtrans} \\
            F^1(1)^{-1} \Msups{\zeta}{}{}{1}{\conjtrans} \\
            \vdots \\
            -F^m(0)^{-1} \Msups{z}{}{}{m}{\conjtrans} \\
            F^m(1)^{-1} \Msups{\zeta}{}{}{m}{\conjtrans}
        \end{bmatrix}
        = 0_{\bccount \times(2\nallgeneral-\bccount)}.
    \end{equation}
    Since $F^r(0), F^r(1)$ are non-singular for each $r$, the $2\nallgeneral-\bccount$ columns of the matrix
    \[
        \mathcal{H} :=
        \begin{bmatrix}
            -F^1(0)^{-1} \Msups{z}{}{}{1}{\conjtrans} \\
            F^1(1)^{-1} \Msups{\zeta}{}{}{1}{\conjtrans} \\
            \vdots \\
            -F^m(0)^{-1} \Msups{z}{}{}{m}{\conjtrans} \\
            F^m(1)^{-1} \Msups{\zeta}{}{}{m}{\conjtrans}
        \end{bmatrix}
    \]
    span the solution space of the system \eqref{eqn:AdjointOperator.system1}, and $\mathrm{rank}(\mathcal{H}) = 2\nallgeneral - \bccount$.

    Suppose that $\BVec{B}^\star$ is adjoint to $\BVec{B}$ and has boundary coefficient matrices $\Msups{b}{}{}{r}{\star},\Msups{\beta}{}{}{r}{\star}$.
    Following the argument in the ``only if'' proof, equation~\eqref{eqn:AdjointOperator.CA.Product2} holds, and both matrices on the right are full rank.
    Therefore, the $2\nallgeneral-\bccount$ columns of the rank $2\nallgeneral - \bccount$ matrix
    \[
        \mathscr{H} :=
        \begin{bmatrix}
            -F^1(0)^{-1} \Msups{b}{}{}{1}{\star\Mspacer\conjtrans} \\
            F^1(1)^{-1} \Msups{\beta}{}{}{1}{\star\Mspacer\conjtrans} \\
            \vdots \\
            -F^m(0)^{-1} \Msups{b}{}{}{m}{\conjtrans} \\
            F^m(1)^{-1} \Msups{\beta}{}{}{m}{\star\Mspacer\conjtrans}
        \end{bmatrix}
    \]
    also span the solution space of equation~\eqref{eqn:AdjointOperator.system1}.
    Therefore, there is a nonsingular matrix $A\in\CC^{(2\nallgeneral - \bccount)\times(2\nallgeneral - \bccount)}$ for which $\mathcal{H}=\mathscr{H}A$.
    But then
    \[
        F^r(0)^{-1} \Msups{z}{}{}{r}{\conjtrans} = F^r(0)^{-1} \Msups{b}{}{}{r}{\star\Mspacer\conjtrans} A,
        \qquad\qquad
        F^1(1)^{-1} \Msups{\zeta}{}{}{r}{\conjtrans} = F^r(1)^{-1} \Msups{\beta}{}{}{r}{\star\Mspacer\conjtrans} A.
    \]
    Because the inverse of $F^r$ is full rank,
    \[
        \Msups{z}{}{}{r}{\conjtrans} = \Msups{b}{}{}{r}{\star\Mspacer\conjtrans} A,
        \qquad\qquad
        \Msups{\zeta}{}{}{r}{\conjtrans} = \Msups{\beta}{}{}{r}{\star\Mspacer\conjtrans} A.
    \]
    That is, the boundary coefficients matrices of $\BVec{Z}$ differ from those of $\BVec{B}^\star$ only by a rank $2\nallgeneral - \bccount$ perturbation; $\BVec{Z}$ is also adjoint to $\BVec{B}$ and $\Phi_{\BVec{Z}}=\Phi_{\BVec{B}^\star}$.
\end{proof}

\bibliographystyle{amsplain}
{\small\bibliography{dbrefs}}

\providecommand{\bysame}{\leavevmode\hbox to3em{\hrulefill}\thinspace}
\providecommand{\MR}{\relax\ifhmode\unskip\space\fi MR }
\providecommand{\MRhref}[2]{%
  \href{http://www.ams.org/mathscinet-getitem?mr=#1}{#2}
}
\providecommand{\href}[2]{#2}
\begin{thebibliography}{10}

\bibitem{BT2019a}
G.~Biondini and T.~Trogdon, \emph{Evolution partial differential equations with
  discontinuous data}, Quart. Appl. Math. \textbf{77} (2019), 689--726.

\bibitem{BT2019b}
\bysame, \emph{{G}ibbs phenomenon for dispersive {PDE}s on the line}, SIAM
  Journal on Applied Mathematics \textbf{77} (2019), 813--837.

\bibitem{Bir1908b}
G.~D. Birkhoff, \emph{Boundary value and expansion problems of ordinary linear
  differential equations}, Trans. Amer. Math. Soc. \textbf{9} (1908), 373--395.

\bibitem{Can1963a}
J.~R. Cannon, \emph{The solution of the heat equation subject to the
  specification of energy}, Quart. Appl. Math \textbf{21} (1963), 155--160.

\bibitem{Chi2006a}
D.~Chilton, \emph{An alternative approach to two-point boundary value problems
  for linear evolution {PDE}s and applications}, Phd, University of Reading,
  2006.

\bibitem{CL1955a}
E.~A. Coddington and N.~Levinson, \emph{Theory of ordinary differential
  equations}, International Series in Pure and Applied Mathematics,
  McGraw-Hill, 1955.

\bibitem{Dav2007a}
E.~B. Davies, \emph{Linear operators and their spectra}, Cambridge Studies in
  Advanced Mathematics, vol. 106, Cambridge University Press, 2007.

\bibitem{DM1963a}
K.~L. Deckert and C.~G. Maple, \emph{Solutions for diffusion equations with
  integral type boundary conditions}, Proc. Iowa Acad. Sci. \textbf{70} (1963),
  354--361.

\bibitem{DGSV2018a}
B.~Deconinck, Q.~Guo, E.~Shlizerman, and V.~Vasan, \emph{{F}okas's unified
  transform method for linear systems}, Quart. of Appl. Math. \textbf{76}
  (2018), no.~3, 463--488.

\bibitem{DPS2014a}
B.~Deconinck, B.~Pelloni, and N.~E. Sheils, \emph{Non-steady-state heat
  conduction in composite walls}, Proc. R. Soc. Lond. Ser. A Math. Phys. Eng.
  Sci. \textbf{470} (2014), no.~2165, 20130605.

\bibitem{DS2014b}
B.~Deconinck and N.~Sheils, \emph{Heat conduction on the ring: interface
  problems with periodic boundary conditions}, Appl. Math. Lett. \textbf{37}
  (2014), 107--111.

\bibitem{DS2020a}
B.~Deconinck and N.~Sheils, \emph{The time-dependent {S}chr\"{o}dinger equation
  with piecewise constant potentials}, Europ. J. of Appl. Math. \textbf{31}
  (2020), 57--83.

\bibitem{DSS2016a}
B.~Deconinck, N.~E. Sheils, and D.~A. Smith, \emph{The linear {K}d{V} equation
  with an interface}, Comm. Math. Phys. \textbf{347} (2016), 489--509.

\bibitem{DT2012a}
B.~Deconinck and T.~Trogdon, \emph{The solution of linear constant-coefficient
  evolution {PDE}s with periodic boundary conditions}, Applicable Analysis
  \textbf{91} (2012), 529--544.

\bibitem{DTV2014a}
B.~Deconinck, T.~Trogdon, and V.~Vasan, \emph{The method of {F}okas for solving
  linear partial differential equations}, SIAM Rev. \textbf{56} (2014), no.~1,
  159--186.

\bibitem{DV2013a}
B.~Deconinck and V.~Vasan, \emph{Well-posedness of boundary-value problems for
  the linear {B}enjamin-{B}ona-{M}ahony equation}, Discrete \& Continuous Dyn.
  Sys. A \textbf{33} (2013), no.~7, 3171--3188.

\bibitem{DS1963a}
N.~Dunford and J.~T. Schwartz, \emph{Linear operators, part {II}: {S}pectral
  theory, self adjoint operators in a {H}ilbert space}, Pure and Applied
  Mathematics, Wiley-Interscience, 1963.

\bibitem{Fok1997a}
A.~S. Fokas, \emph{A unified transform method for solving linear and certain
  nonlinear {PDE}s}, Proc. R. Soc. Lond. Ser. A Math. Phys. Eng. Sci.
  \textbf{453} (1997), 1411--1443.

\bibitem{Fok2000a}
\bysame, \emph{On the integrability of linear and nonlinear partial
  differential equations}, J. Math. Phys. \textbf{41} (2000), 4188--4237.

\bibitem{Fok2002a}
\bysame, \emph{A new transform method for evolution {PDE}s}, IMA J. Appl. Math.
  \textbf{67} (2002), 559--590.

\bibitem{Fok2008a}
\bysame, \emph{A unified approach to boundary value problems}, CBMS-SIAM, 2008.

\bibitem{FG1994a}
A.~S. Fokas and I.~M. Gel'fand, \emph{Integrability of linear and nonlinear
  evolution equations and the associated nonlinear {F}ourier transforms}, Lett.
  Math. Phys. \textbf{32} (1994), 189--210.

\bibitem{FP2001a}
A.~S. Fokas and B.~Pelloni, \emph{Two-point boundary value problems for linear
  evolution equations}, Math. Proc. Cambridge Philos. Soc. \textbf{131} (2001),
  521--543.

\bibitem{FP2005b}
\bysame, \emph{Boundary value problems for {B}oussinesq type systems}, Math.
  Phys. Anal. Geom. \textbf{8} (2005), no.~1, 59--96.

\bibitem{FP2005a}
\bysame, \emph{A transform method for linear evolution {PDE}s on a finite
  interval}, IMA J. Appl. Math. \textbf{70} (2005), 564--587.

\bibitem{FS2016a}
A.~S. Fokas and D.~A. Smith, \emph{Evolution {P}{D}{E}s and augmented
  eigenfunctions. {F}inite interval}, Adv. Differential Equations \textbf{21}
  (2016), no.~7/8, 735--766.

\bibitem{FS2012a}
A.~S. Fokas and E.~A. Spence, \emph{Synthesis, as opposed to separation, of
  variables}, SIAM Rev. \textbf{54} (2012), no.~2, 291--324.

\bibitem{Fou1822a}
J.~B.~J. Fourier, \emph{Th\'{e}orie analytique de la chaleur}, Didot, Paris,
  1822.

\bibitem{Fre2012a}
G.~Freiling, \emph{Irregular boundary value problems}, Results Math.
  \textbf{62} (2012), 265--294.

\bibitem{GS1967a}
I.~M. Gel'fand and G.~E. Shilov, \emph{Generalized functions volume 3: theory
  of differential equations}, Academic Press, 1967, Trans. M. E. Mayer from
  Russian (1958).

\bibitem{GV1964a}
I.~M. Gel'fand and N.~Ya. Vilenkin, \emph{Generalized functions volume 4:
  applications of harmonic analysis}, Academic Press, 1964, Trans. A. Feinstein
  from Russian (1961).

\bibitem{GPV2019a}
R.~Govindarajan, S.~G. Prasath, and V.~Vasan, \emph{Accurate solution method
  for the {M}axey-{R}iley equation, and the effects of {B}asset history}, J.
  Fluid Mechanics \textbf{868} (2019), 428--460.

\bibitem{Hop1919a}
J.~W. Hopkins, \emph{Some convergent developments associated with irregular
  boundary conditions}, Trans. Amer. Math. Soc. \textbf{20} (1919), 245--259.

\bibitem{Jac1915a}
D.~Jackson, \emph{Expansion problems with irregular boundary conditions}, Proc.
  Amer. Acad. Arts Sci. \textbf{51} (1915), no.~7, 383--417.

\bibitem{KPPS2018a}
E.~Kesici, B.~Pelloni, T.~Pryer, and D.~A. Smith, \emph{A numerical
  implementation of the unified {F}okas transform for evolution problems on a
  finite interval}, Euro. J. Appl. Math. \textbf{29} (2018), no.~3, 543--567.

\bibitem{Lan1931a}
R.~E. Langer, \emph{The zeros of exponential sums and integrals}, Bull. Amer.
  Math. Soc. \textbf{37} (1931), 213--239.

\bibitem{SL1837a}
J.~Liouville and C.~Sturm, \emph{Extrait d’un m{\'e}moire sur le
  d{\'e}veloppement des fonctions en s{\'e}ries dont les diff{\'e}rents termes
  sont assujettisa satisfairea une m{\^e}me {\'e}quation diff{\'e}rentielle
  lin{\'e}aire, contenant un parametre variable}, J. Math. Pures Appl
  \textbf{2} (1837), 220--233.

\bibitem{Loc1973a}
J.~Locker, \emph{Self-adjointness for multi-point differential operators},
  Pacific J. Math. \textbf{45} (1973), no.~2, 561--570.

\bibitem{Loc2000a}
\bysame, \emph{Spectral theory of non-self-adjoint two-point differential
  operators}, Mathematical Surveys and Monographs, vol.~73, American
  Mathematical Society, Providence, Rhode Island, 2000.

\bibitem{Loc2008a}
\bysame, \emph{Eigenvalues and completeness for regular and simply irregular
  two-point differential operators}, vol. 195, Memoirs of the American
  Mathematical Society, no. 911, American Mathematical Society, Providence,
  Rhode Island, 2008.

\bibitem{MS2018a}
P.~D. Miller and D.~A. Smith, \emph{The diffusion equation with nonlocal data},
  J. Math. Anal. Appl. \textbf{466} (2018), no.~2, 1119--1143.

\bibitem{Neu1966a}
J.~W. Neuberger, \emph{The lack of self-adjointness in three-point boundary
  value problems}, Pacific J. Math. \textbf{18} (1966), no.~1, 165--168.

\bibitem{Pap2011a}
G.~Papanicolaou, \emph{An example where separation of variable fails}, J. Math.
  Anal. Appl. \textbf{373} (2011), no.~2, 739--744.

\bibitem{Pel2002a}
B.~Pelloni, \emph{Well-posed boundary value problems for integrable evolution
  equations on a finite interval}, Theoret. and Math. Phys. \textbf{133}
  (2002), no.~2, 1598--1606.

\bibitem{Pel2004a}
\bysame, \emph{Well-posed boundary value problems for linear evolution
  equations on a finite interval}, Math. Proc. Cambridge Philos. Soc.
  \textbf{136} (2004), 361--382.

\bibitem{Pel2005a}
\bysame, \emph{The spectral representation of two-point boundary-value problems
  for third-order linear evolution partial differential equations}, Proc. R.
  Soc. Lond. Ser. A Math. Phys. Eng. Sci. \textbf{461} (2005), 2965--2984.

\bibitem{PS2013a}
B.~Pelloni and D.~A. Smith, \emph{Spectral theory of some non-selfadjoint
  linear differential operators}, Proc. R. Soc. Lond. Ser. A Math. Phys. Eng.
  Sci. \textbf{469} (2013), no.~2154, 20130019.

\bibitem{PS2016a}
\bysame, \emph{Evolution {P}{D}{E}s and augmented eigenfunctions. {H}alf line},
  J. Spectr. Theory \textbf{6} (2016), 185--213.

\bibitem{PS2018a}
\bysame, \emph{Nonlocal and multipoint boundary value problems for linear
  evolution equations}, Stud. Appl. Math. \textbf{141} (2018), no.~1, 46--88.

\bibitem{Pin2011b}
M.~A. Pinsky, \emph{Partial differential equations and boundary-value
  problems}, Pure and Applied Undergraduate Texts, American Mathematical
  Society, Providence, Rhode Island, USA, 2011.

\bibitem{RS1975a}
M.~Reed and B.~Simon, \emph{Methods of modern mathematical physics volume 2:
  {F}ourier analysis, self-adjointness}, Academic Press, Cambridge, MA, 1975.

\bibitem{She2017a}
N.~Sheils, \emph{Multilayer diffusion in a composite medium with imperfect
  contact}, Applied Mathematical Modelling \textbf{46} (2017), 450--464.

\bibitem{SS2015a}
N.~E. Sheils and D.~A. Smith, \emph{Heat equation on a network using the
  {F}okas method}, Journal of Physics A: Mathematical and Theoretical
  \textbf{48} (2015), no.~33, 21 pp.

\bibitem{Smi2011a}
D.~A. Smith, \emph{Spectral theory of ordinary and partial linear differential
  operators on finite intervals}, Phd, University of Reading, 2011.

\bibitem{Smi2012a}
\bysame, \emph{Well-posed two-point initial-boundary value problems with
  arbitrary boundary conditions}, Math. Proc. Cambridge Philos. Soc.
  \textbf{152} (2012), 473--496.

\bibitem{Smi2012b}
\bysame, \emph{Well-posedness and conditioning of 3rd and higher order
  two-point initial-boundary value problems}, arXiv:1212.5466 [math.AP], 2012.

\bibitem{Smi2015a}
\bysame, \emph{The unified transform method for linear initial-boundary value
  problems: a spectral interpretation}, Unified transform method for boundary
  value problems: applications and advances, SIAM, Philadelphia, PA, 2015.

\bibitem{ST2021a}
D.~A. Smith and W.-Y. Toh, \emph{Linear evolution equations on the half line
  with dynamic boundary conditions}, Eur. J. Appl. Math. (2021), (to appear)
  arXiv:1910.08764 [math.AP].

\bibitem{STV2019a}
D.~A. Smith, T.~Trogdon, and V.~Vasan, \emph{Linear dispersive shocks},
  (submitted) arXiv:1908.08716 [math.AP], 2019.

\bibitem{Xia2019a}
L.~Xiao, \emph{Algorithmic solution of high order partial differential
  equations in julia via the {F}okas transform method}, Yale-NUS College
  Capstone Project, 2019, \url{https://gitlab.com/linfanxiaolinda/capstone}.

\bibitem{Zet1966a}
A.~Zettl, \emph{The lack of self-adjointness in three point boundary value
  problems}, Proc. Amer. Math. Soc. \textbf{17} (1966), no.~2, 368--371.

\end{thebibliography}

\end{document}